\documentclass[11pt, oneside, reqno]{amsart}
\usepackage[utf8]{inputenc}
\usepackage{mathrsfs}
\usepackage{bm}
\usepackage{amstext}
\usepackage{amsthm}
\usepackage{amssymb}
\usepackage{mathtools}
\usepackage[all]{xy}
\usepackage{cleveref}
\usepackage{graphics}
\usepackage{ytableau}
\usepackage{version}
\usepackage{colonequals}
\usepackage{color}

\makeatletter
\@namedef{subjclassname@2020}{\textup{2020} Mathematics Subject Classification}
\makeatother
\numberwithin{equation}{section}
\numberwithin{figure}{section}
\theoremstyle{plain}
\newtheorem{thm}{Theorem}[section]
  \crefname{thm}{Theorem}{Theorems}
  
  \crefname{lem}{Lemma}{Lemmas}
  \newtheorem{prop}[thm]{Proposition}
  \crefname{prop}{Proposition}{Propositions}
  
	\crefname{cor}{Corollary}{Corollaries}

  \newtheorem*{ack}{Acknowledgments}
\theoremstyle{definition}
  \newtheorem{defi}[thm]{Definition}
  \crefname{defi}{Definition}{Definitions}
  \theoremstyle{remark}
  
  \crefname{ntn}{Notation}{Notations}
	 \theoremstyle{remark}
  \newtheorem{rem}[thm]{Remark}
  \crefname{rem}{Remark}{Remarks}
  \newtheorem{ex}[thm]{Example}
  \crefname{ex}{Example}{Examples}
  
\usepackage{a4wide}

\def\r{\mathbb{R}}

\def\z{\mathbb{Z}}

\makeatother

\makeatletter
\ifcsname phantomsection\endcsname
    \newcommand*{\qrr@gobblenexttocentry}[5]{}
\else
    \newcommand*{\qrr@gobblenexttocentry}[4]{}
\fi
\newcommand*{\addsubsection}{%
    \addtocontents{toc}{\protect\qrr@gobblenexttocentry}%
    \subsection}
\makeatother
 \makeatletter
    
    \@addtoreset{equation}{section}
  \makeatother
\begin{document}
\title{Polyptych lattices and marked chain-order polytopes}

\author{Naoki FUJITA}

\address[Naoki FUJITA]{Faculty of Advanced Science and Technology, Kumamoto University, Kumamoto 860-8555, Japan.}

\email{fnaoki@kumamoto-u.ac.jp}

\author{Akihiro HIGASHITANI}

\address[Akihiro HIGASHITANI]{Department of Pure and Applied Mathematics, Graduate School of Information Science and Technology, Osaka University, Suita, Osaka 565-0871, Japan.}

\email{higashitani@ist.osaka-u.ac.jp}

\subjclass[2020]{Primary 52B20; Secondary 05E10, 06A07, 14M15, 14M25}

\keywords{polyptych lattice, marked chain-order polytope, combinatorial mutation, toric degeneration, Newton--Okounkov body, Cox ring, Gelfand--Tsetlin poset}

\thanks{The work was supported by JSPS Grant-in-Aid for Scientific Research (B) (No.\ 24K00521). 
The work of the first named author was supported by JSPS Grant-in-Aid for Early-Career Scientists (No.\ 24K16902) and by MEXT Japan Leading Initiative for Excellent Young Researchers (LEADER) Project.}

\date{\today}

\begin{abstract}
The theory of polyptych lattices is a framework to obtain a family of toric degenerations whose polytopes are related by piecewise-linear transformations. 
It can be regarded as a generalization of toric degenerations arising from cluster algebras. 
In this paper, we study polyptych lattices consisting of transfer maps for marked chain-order polytopes, and obtain a family of toric degenerations of a projective variety to marked chain-order polytopes for the Gelfand--Tsetlin poset.
We also compute the Cox ring of this projective variety.
\end{abstract}
\maketitle
\tableofcontents 

\section{Introduction}

A marked chain-order polytope $\Delta_{\mathcal{C}, \mathcal{O}}(\Pi, \Pi^\ast, \lambda)$ is a convex polytope constructed from a marked poset $(\Pi, \Pi^\ast, \lambda)$ with a partition $\Pi \setminus \Pi^\ast = \mathcal{C} \sqcup \mathcal{O}$ under the assumption that $\Pi$ is finite and that $\Pi^\ast$ contains all maximal and minimal elements in $\Pi$.
It was introduced by Fang--Fourier \cite{FF} in some specific setting and by Fang--Fourier--Litza--Pegel \cite{FFLP} in general setting as a family relating the marked order polytope $\mathcal{O}(\Pi, \Pi^\ast, \lambda)$ with the marked chain polytope ${\mathcal C}(\Pi, \Pi^\ast, \lambda)$.
When $(\Pi, \Pi^\ast, \lambda)$ is the Gelfand--Tsetlin poset $(\Pi_A, \Pi_A^\ast, \lambda)$ of type $A$, the first named author \cite{Fuj} realized the marked chain-order polytope $\Delta_{\mathcal{C}, \mathcal{O}}(\Pi_A, \Pi_A^\ast, \lambda)$ as a Newton--Okounkov body and as a toric degeneration of a flag variety. 
After that, Makhlin \cite{Mak} realized it as a different kind of toric degeneration (Gr\"{o}bner degeneration) and a Newton--Okounkov body of the flag variety, which was extended by Makedonskyi--Makhlin \cite{MM} to the Gelfand--Tsetlin poset $(\Pi_C, \Pi_C^\ast, \lambda)$ of type $C$. 
Newton--Okounkov bodies were originally introduced by Okounkov \cite{Oko} to study multiplicity functions for representations of a reductive group, and afterward developed independently by Lazarsfeld--Musta\c{t}\u{a} \cite{LM} and by Kaveh--Khovanskii \cite{KK1, KK2} as a generalization of the notion of Newton polytopes for toric varieties. 
The theory of Newton--Okounkov bodies gives a systematic method of constructing toric degenerations and completely integrable systems (see \cite{And, HK}). 

A different family of Newton--Okounkov bodies (and associated toric degenerations) of a flag variety arises from the theory of cluster algebras (see \cite{GHKK, FO}).  
The family of Newton--Okounkov bodies is parametrized by certain cluster-theoretic data, called seeds, and these Newton--Okounkov bodies are mutually related by specific piecewise-linear transformations, called tropicalized cluster mutations (see also \cite{BCMNC, BFMNC, RW}). 
The first named author and Oya \cite{FO} realized Berenstein--Littelmann--Zelevinsky's string polytopes and Nakashima--Zelevinsky's polyhedral realizations of highest weight crystal bases as Newton--Okounkov bodies of flag varieties arising from cluster structures. 
The theory of polyptych lattices was introduced by Escobar--Harada--Manon \cite{EHM} to generalize such framework to more general piecewise-linear transformations. 
In the present paper, we study polyptych lattices consisting of transfer maps for marked chain-order polytopes, and obtain a family of toric degenerations of a projective variety to marked chain-order polytopes for a marked poset with desirable properties like Gelfand--Tsetlin posets of types $A$ and $C$.
In the case of such marked poset, the authors \cite{FH} realized transfer maps for marked chain-order polytopes as combinatorial mutations introduced by Akhtar--Coates--Galkin--Kasprzyk \cite{ACGK} in the context of mirror symmetry for Fano manifolds. 
Hence our result gives a large class of examples of polyptych lattices with desirable properties consisting of combinatorial mutations. 
Although some examples of polyptych lattices with desirable properties were studied in previous researches \cite{EHM, CEHM, Oda}, our class of examples of polyptych lattices is quite different from them.

To state our main result explicitly, let $(\Pi, \Pi^\ast, \lambda)$ be a graded finite marked poset with $\lambda = (\lambda_a)_{a \in \Pi^\ast} \in \mathbb{Z}^{\Pi^\ast}$ satisfying $\lambda_a = \lambda_b$ for all $a, b \in \Pi^\ast$ such that $r(a) = r(b)$, where $r(p)$ for $p \in \Pi$ means the length of a chain from a minimal element in $\Pi$ to $p$, which is independent of the choice of such chain. 
For a partition $\Pi \setminus \Pi^\ast = \mathcal{C} \sqcup \mathcal{O}$, we denote by $\Delta_{\mathcal{C}, \mathcal{O}} (\Pi, \Pi^\ast, \lambda) \subseteq \mathbb{R}^{\Pi \setminus \Pi^\ast}$ the corresponding marked chain-order polytope and by $\phi_{\mathcal{C}, \mathcal{O}} \colon \Delta_{\emptyset, \Pi \setminus \Pi^\ast} (\Pi, \Pi^\ast, \lambda) \rightarrow \Delta_{\mathcal{C}, \mathcal{O}} (\Pi, \Pi^\ast, \lambda)$ the transfer map (see Section \ref{ss:marked_poset_polytopes}).  
Take ${\bm u} \in \Delta_{\emptyset, \Pi \setminus \Pi^\ast} (\Pi, \Pi^\ast, \lambda) \cap \mathbb{Z}^{\Pi \setminus \Pi^\ast}$ satisfying \eqref{eq:assumption}, and set 
\[\widehat{\Delta}_{\mathcal{C}}(\Pi, \Pi^\ast, \lambda) \coloneqq \Delta_{\mathcal{C}, \mathcal{O}}(\Pi, \Pi^\ast, \lambda)  - \phi_{\mathcal{C}, \mathcal{O}}({\bm u}).\]
Then $\mu_{\mathcal{C}} (\widehat{\Delta}_{\emptyset}(\Pi, \Pi^\ast, \lambda)) = \widehat{\Delta}_{\mathcal{C}}(\Pi, \Pi^\ast, \lambda)$ for the piecewise-linear transformation $\mu_{\mathcal{C}}$ given in \eqref{eq:transfer_linear}.
Let us consider a finite polyptych lattice $\mathcal{M}$ defined by gluing $M_{\mathcal{C}} \coloneqq \mathbb{Z}^{\Pi \setminus \Pi^\ast}$ for $\mathcal{C} \in 2^{\Pi \setminus \Pi^\ast}$ via $\mu_{\mathcal{C}_1, \mathcal{C}_2} \coloneqq \mu_{\mathcal{C}_2} \circ \mu_{\mathcal{C}_1}^{-1} \colon M_{\mathcal{C}_1} \rightarrow M_{\mathcal{C}_2}$ with $\mathcal{C}_1, \mathcal{C}_2 \in 2^{\Pi \setminus \Pi^\ast}$ (see Definition \ref{d:polyptych_lattice} for the precise definition of polyptych lattices). 
For $i \in \mathbb{Z}_{\geq 0}$, we write $\Pi(i) \coloneqq \{p \in \Pi \mid r(p) = i\}$, and consider the subposet $\breve{\Pi}(i)$ of $\Pi$ obtained from $\Pi(i) \cup \Pi(i+1)$ by removing $p \in \Pi(i+1)$ such that $p \in \Pi^\ast$ or $|\{q \in \Pi(i) \mid p\ \text{covers}\ q\}| \leq 1$.  
Let $\mathscr{C}(\Pi, i)$ denote the set of \emph{connected components} of the subposet $\breve{\Pi}(i)$, where a connected component means a subset of $\breve{\Pi}(i)$ that corresponds to a connected component of its (marked) Hasse diagram. 
We assume the following condition on $(\Pi, \Pi^\ast, \lambda)$: 
\begin{enumerate}
\item[($\spadesuit$)] for each $i \in \mathbb{Z}_{\geq 0}$, every $\mathcal{C} \in \mathscr{C}(\Pi, i)$ is of the form in Figure \ref{Hasse_nonmarked} or Figure \ref{Hasse_marked} unless $\mathcal{C} \cap \Pi(i+1) = \emptyset$. 
\end{enumerate}
Then the polyptych lattice $\mathcal{M}$ has a strict dual pair in the sense of \cite[Definition 4.1]{EHM} (see Definition \ref{d:strict_dual}, Theorem \ref{t:strict_dual_pair}, and Remark \ref{r:general_case}).
Let $i \in \mathbb{Z}_{\geq 0}$, and consider $\mathcal{C} \in \mathscr{C}(\Pi, i)$ of the form in Figure \ref{Hasse_nonmarked} or Figure \ref{Hasse_marked}.
In particular, we can write $\mathcal{C} \cap \Pi(i) = \{p_1^{(\mathcal{C})}, \ldots, p_{n_\mathcal{C} +1}^{(\mathcal{C})}\}$ and $\mathcal{C} \cap \Pi(i+1) = \{q_1^{(\mathcal{C})}, \ldots, q_{n_\mathcal{C}}^{(\mathcal{C})}\}$ such that $p_k^{(\mathcal{C})}, p_{k+1}^{(\mathcal{C})} \lessdot q_k^{(\mathcal{C})}$ for $1 \leq k \leq n_\mathcal{C}$. 
Note that if $\mathcal{C}$ is of the form in Figure \ref{Hasse_nonmarked}, then there exist two such arrangements of the elements of $\mathcal{C}$. 
Hence we fix such arrangement for each $\mathcal{C}$. 
Let $\Bbbk$ be an algebraically closed field of characteristic $0$, and
consider a polynomial ring $\Bbbk [X_p, Y_p \mid p \in \Pi \setminus \Pi^\ast]$.
For $p \in \Pi \setminus \Pi^\ast$, we define $g_p \in \Bbbk [X_p, Y_p \mid p \in \Pi \setminus \Pi^\ast]$ as follows. 
If $p$ is not of the form $p_k^{(\mathcal{C})}$ for some $\mathcal{C} \in \mathscr{C}(\Pi, r(p))$ and $1 \leq k \leq n_\mathcal{C}$, then $g_p \coloneqq X_p Y_p - 1$. 
If $p = p_k^{(\mathcal{C})}$ for some $\mathcal{C} \in \mathscr{C}(\Pi, r(p))$ and $1 \leq k \leq n_\mathcal{C}$, then consider $\mathcal{C}^\prime \in \mathscr{C}(\Pi, r(p)+1)$ containing $q_k^{(\mathcal{C})}$. 
If $q_k^{(\mathcal{C})}$ is not of the form $p_k^{(\mathcal{C}^\prime)}$ for some $2 \leq k \leq n_{\mathcal{C}^\prime}+1$, then $g_p \coloneqq X_p Y_p - 1 - Y_{q_k^{(\mathcal{C})}}$. 
If $q_k^{(\mathcal{C})} = p_k^{(\mathcal{C}^\prime)}$ for some $2 \leq k \leq n_{\mathcal{C}^\prime}+1$, then $g_p \coloneqq X_p Y_p - 1 - X_{p_{k-1}^{(\mathcal{C}^\prime)}} Y_{q_k^{(\mathcal{C})}}$. 
Let 
\[\mathcal{A} \coloneqq \Bbbk [X_p, Y_p \mid p \in \Pi \setminus \Pi^\ast]/I,\]
where $I \coloneqq (g_p \mid p \in \Pi \setminus \Pi^\ast)$. 
Then $\mathcal{A}$ is a UFD and ${\rm Spec}(\mathcal{A})$ is nonsingular (see Propositions \ref{p:integral_domain}, \ref{p:nonsingular}, \ref{p:UFD} and Remark \ref{r:general_case}).
For $f \in \Bbbk [X_p, Y_p \mid p \in \Pi \setminus \Pi^\ast]$, we set $\overline{f} \coloneqq f \bmod I \in \mathcal{A}$. 
Fix a total order on the variables $\{X_p, Y_p \mid p \in \Pi \setminus \Pi^\ast\}$ such that $X_{p_1} > X_{p_2}$ and $Y_{p_1} > Y_{p_2}$ when $r(p_1) < r(p_2)$, and define a monomial order $<$ on $\Bbbk [X_p, Y_p \mid p \in \Pi \setminus \Pi^\ast]$ to be the corresponding lexicographic order.
Then $\{g_p \mid p \in \Pi \setminus \Pi^\ast\}$ forms a Gr\"{o}bner basis of $I$ with respect to the monomial order $<$ since the leading monomials of these polynomials $g_p$ ($p \in \Pi \setminus \Pi^\ast$) are relatively prime.
See \cite{CLO, Stu} for details on Gr\"{o}bner bases.
Let ${\bf B} \subseteq \Bbbk [X_p, Y_p \mid p \in \Pi \setminus \Pi^\ast]$ denote the set of standard monomials. 
In particular, $\overline{\bf B} \coloneqq \{\overline{b} \mid b \in {\bf B}\}$ forms a $\Bbbk$-basis of $\mathcal{A}$. 
By the definition of $g_p$ and $<$, it follows that 
\[{\bf B} = \left\{\prod_{p \in \Pi \setminus \Pi^\ast} X_p^{a_p} Y_p^{b_p} \mid \min\{a_p, b_p\} = 0\ \text{for all}\ p \in \Pi \setminus \Pi^\ast \right\};\]
see, for instance, \cite[Proposition 1.1]{Stu}.
Hence we obtain a bijective map 
\[{\bf B} \rightarrow \mathbb{Z}^{\Pi \setminus \Pi^\ast},\quad \prod_{p \in \Pi \setminus \Pi^\ast} X_p^{a_p} Y_p^{b_p} \mapsto (a_p -b_p)_{p \in \Pi \setminus \Pi^\ast}.\]
For $p \in \Pi \setminus \Pi^\ast$, we define $\hat{\bm e}_p \in \mathbb{Z}^{\Pi \setminus \Pi^\ast}$ as follows. 
If $p$ is not of the form $p_k^{(\mathcal{C})}$ for some $\mathcal{C} \in \mathscr{C}(\Pi, r(p))$ and $1 \leq k \leq n_\mathcal{C}+1$, then $\hat{\bm e}_p \coloneqq {\bm e}_p$, where ${\bm e}_p$ denotes the unit vector. 
If $p = p_k^{(\mathcal{C})}$ for some $\mathcal{C} \in \mathscr{C}(\Pi, r(p))$ and $1 \leq k \leq n_\mathcal{C}+1$, then $\hat{\bm e}_p \coloneqq {\bm e}_{p_1^{(\mathcal{C})}} + \cdots + {\bm e}_{p_k^{(\mathcal{C})}}$. 
For $b = \prod_{p \in \Pi \setminus \Pi^\ast} X_p^{a_p} Y_p^{b_p} \in {\bf B}$, let $m_b \in \mathcal{M}$ denote the element corresponding to $\sum_{p \in \Pi \setminus \Pi^\ast} (a_p -b_p) \hat{\bm e}_p \in M_\emptyset$. 
Then we obtain a bijective map ${\bf B} \rightarrow \mathcal{M}$, $b \mapsto m_b$.

\begin{thm}[{see Theorem \ref{t:detropicalization_type_C} and Remark \ref{r:general_case}}]\label{t:1}
The map ${\bf B} \rightarrow \mathcal{M}$ induces a valuation $\nu$ on $\mathcal{A}$ to the canonical $\mathbb{Z}_{\geq 0}$-semialgebra $S_{\mathcal{M}}$ associated with $\mathcal{M}$, which makes $\mathcal{A}$ a detropicalization of $\mathcal{M}$ with convex adapted basis $\overline{\bf B}$; see Section \ref{s:polyptych} for the precise definitions of $S_{\mathcal{M}}$, a detropicalization of $\mathcal{M}$, and a convex adapted basis for a detropicalization.
\end{thm}

Assume that 
\begin{enumerate}
\item[($\dagger$)] the element ${\bm u}$ above is an interior lattice point of $\Delta_{\emptyset, \Pi \setminus \Pi^\ast} (\Pi, \Pi^\ast, \lambda)$.  
\end{enumerate}
We define a $\mathbb{Z}_{\geq 0}$-graded $\Bbbk$-subalgebra $\mathcal{A}^{\widehat{\Delta}}$ of the polynomial ring $\mathcal{A}[t]$ with coefficient $\mathcal{A}$ by 
\[\mathcal{A}^{\widehat{\Delta}} \coloneqq \bigoplus_{k \in \mathbb{Z}_{\geq 0}} {\rm Span}_{\Bbbk}\{b \cdot t^k \mid \pi_\emptyset (m_b) \in k \widehat{\Delta}_{\emptyset}(\Pi, \Pi^\ast, \lambda)\},\]
where $\pi_\emptyset (m_b) \in M_\emptyset$ is the element corresponding to $m_b \in \mathcal{M}$.
As an application of Theorem \ref{t:1}, we prove that $\mathcal{M}, \mathcal{A}, {\bf B}$, and $\mathcal{A}^{\widehat{\Delta}} $ fit well into the general theory of polyptych lattices developed in \cite{EHM}, which shows the following theorem.

\begin{thm}[{see Theorem \ref{t:NO_degeneration_type_C} and Remark \ref{r:general_case}}]\label{t:2}
The following hold.
\begin{enumerate}
\item[(1)] The scheme $X_{\mathcal{A}}(\widehat{\Delta}) \coloneqq {\rm Proj}(\mathcal{A}^{\widehat{\Delta}})$ is a normal projective variety over $\Bbbk$ containing ${\rm Spec}(\mathcal{A})$ as an open subvariety.
\item[(2)] For each $\mathcal{C} \in 2^{\Pi \setminus \Pi^\ast}$, the polytope $\widehat{\Delta}_{\mathcal{C}}(\Pi, \Pi^\ast, \lambda)$ is a Newton--Okounkov body of $X_{\mathcal{A}}(\widehat{\Delta})$, and there exists a flat (toric) degeneration of $X_{\mathcal{A}} (\widehat{\Delta})$ to the normal projective toric variety corresponding to the integral polytope $\widehat{\Delta}_{\mathcal{C}}(\Pi, \Pi^\ast, \lambda)$.  
\end{enumerate}
\end{thm}

Let $U(\mathcal{M}, i)$ be the number of $\mathcal{C} \in \mathscr{C}(\Pi, i)$ of the form in Figure \ref{Hasse_nonmarked}, and 
\begin{equation}\label{eq:rank_of_units}
\begin{aligned}
U_{\mathcal{M}} \coloneqq \sum_{i \in \mathbb{Z}_{\geq 0}} U(\mathcal{M}, i).
\end{aligned}
\end{equation} 
Then the unit group $\mathcal{A}^\times$ is isomorphic to $\Bbbk^\times \times \mathbb{Z}^{U_{\mathcal{M}}}$ (see Proposition \ref{p:unit_elements} and Remark \ref{r:general_case}).
We set 
\begin{equation}\label{eq:number_of_divisors}
\begin{aligned}
L_{\mathcal{M}} \coloneqq |\Pi \setminus \Pi^\ast| + |\{(p,p^\prime) \in \Pi^\ast \times (\Pi \setminus \Pi^\ast) \mid p\ \text{covers}\ p^\prime\}|;
\end{aligned}
\end{equation}
cf.\ Proposition \ref{p:PL_description}.
Using the arguments in \cite[Construction 1.2.4.1]{ADHL}, \cite[Theorem 7.19]{EHM}, and \cite[Section 7]{CEHM}, we compute the Cox ring of $X_{\mathcal{A}} (\widehat{\Delta})$, and obtain the following. 

\begin{thm}[{see Theorem \ref{t:Cox_type_C} and Remark \ref{r:general_case}}]\label{t:3}
The Cox ring of $X_{\mathcal{A}} (\widehat{\Delta})$ is isomorphic to the polynomial ring over $\Bbbk$ in $|\Pi \setminus \Pi^\ast| - U_{\mathcal{M}} + L_{\mathcal{M}}$ variables. 
\end{thm}

The present paper is organized as follows. 
In Section 2, we recall basic definitions on polyptych lattices. 
In Section 3, we define polyptych lattices consisting of transfer maps for marked chain-order polytopes, and describe marked chain-order polytopes as PL polytopes. 
The spaces of points are also described in this section for some specific marked posets. 
Section 4 is devoted to studying polyptych lattices consisting of transfer maps for Gelfand--Tsetlin posets of type $C$. 
The proofs of our main results are also given in this section. 
In Section 5, we describe our main results in the case of Gelfand--Tsetlin posets of type $A$. 

\begin{ack}\normalfont
The authors are grateful to Megumi Harada for helpful comments and fruitful discussions.
In particular, she suggested the authors to compute the Cox rings of the compactifications.  
\end{ack}

\section{Basic definitions on polyptych lattices}\label{s:polyptych}

In this section, we review some basic definitions on polyptych lattices, following \cite{EHM}. 
A \emph{lattice} is a free $\mathbb{Z}$-module of finite rank.

\begin{defi}[{see \cite[Definition 2.1]{EHM}}]\label{d:polyptych_lattice}
For $r \in \mathbb{Z}_{>0}$ and a set $\mathcal{I}$, a \emph{polyptych lattice} $\mathcal{M} = (\{M_\alpha\}_{\alpha \in \mathcal{I}}, \{\mu_{\alpha,\beta} \colon M_\alpha \rightarrow M_\beta\}_{\alpha, \beta \in \mathcal{I}})$ of rank $r$ (over $\mathbb{Z}$) is a pair of 
\begin{itemize}
\item a collection $\{M_\alpha\}_{\alpha \in \mathcal{I}}$ of rank $r$ lattices indexed by $\mathcal{I}$, and 
\item a collection $\{\mu_{\alpha,\beta} \colon M_\alpha \rightarrow M_\beta\}_{\alpha, \beta \in \mathcal{I}}$ of bijective piecewise-linear maps $\mu_{\alpha,\beta} \colon M_\alpha \rightarrow M_\beta$ between every pair $(M_\alpha, M_\beta)$ of lattices
\end{itemize}
such that 
\begin{itemize}
\item $\mu_{\alpha, \alpha}$ is the identity map ${\rm id}_{M_\alpha}$ on $M_\alpha$ for all $\alpha \in \mathcal{I}$, 
\item $\mu_{\alpha,\beta} = \mu_{\beta,\alpha}^{-1}$ for all $\alpha, \beta \in \mathcal{I}$, and 
\item $\mu_{\beta,\gamma} \circ \mu_{\alpha,\beta} = \mu_{\alpha,\gamma}$ for all $\alpha, \beta, \gamma \in \mathcal{I}$.
\end{itemize}
Each $\mu_{\alpha,\beta}$ is called a \emph{mutation}, and $M_\alpha$ is said to be the $\alpha$-th chart of $\mathcal{M}$. 
We write $\pi(\mathcal{M}) \coloneqq \mathcal{I}$. 
\end{defi}

In this paper, we consider only the case that $\mathcal{I}$ is a finite set. 
In this case, $\mathcal{M}$ is called a \emph{finite polyptych lattice}.
The polyptych lattice $\mathcal{M}$ can be identified with the quotient space
\[\mathcal{M} = \bigsqcup_{\alpha \in \mathcal{I}} M_\alpha \bigg/ \sim,\]
where $m_{\alpha} \sim m_{\beta}$ for $m_{\alpha} \in M_{\alpha}$, $m_{\beta} \in M_{\beta}$ if and only if $m_{\beta} = \mu_{\alpha, \beta} (m_{\alpha})$. 
An \emph{element} $m \in \mathcal{M}$ is an equivalence class in this quotient space. 
Since $\mu_{\alpha, \beta}$ is bijective, an element $m \in \mathcal{M}$ has a unique representative $m_\alpha \in M_\alpha$ called the \emph{$\alpha$-th coordinate} of $m$ (see \cite[Definition 2.3]{EHM}).
For $\alpha \in \mathcal{I}$, define $\pi_\alpha \colon \mathcal{M} \rightarrow M_\alpha$ by $\pi_\alpha (m) \coloneqq m_\alpha$. 
For $m_1, m_2 \in \mathcal{M}$, we denote by $m_1 +_\alpha m_2$ the sum of $m_1, m_2$ in the lattice $M_\alpha$, that is, 
\[m_1 +_\alpha m_2 \coloneqq \pi_\alpha^{-1} (\pi_\alpha(m_1) + \pi_\alpha (m_2)).\]

\begin{defi}[{\cite[Definition 3.1]{EHM}}]
A \emph{point} of $\mathcal{M}$ is a function $p \colon \mathcal{M} \rightarrow \mathbb{Z}$ such that 
\[p(m_1) + p(m_2) = \min \{p(m_1 +_\alpha m_2) \mid \alpha \in \mathcal{I}\}\]
for all $m_1, m_2 \in \mathcal{M}$ and such that 
\[p(\lambda m) = \lambda p(m)\]
for all $m \in \mathcal{M}$ and $\lambda \in \mathbb{Z}_{\geq 0}$. 
The \emph{space ${\rm Sp}(\mathcal{M})$ of points} of $\mathcal{M}$ is defined to be the set of points $p \colon \mathcal{M} \rightarrow \mathbb{Z}$ of $\mathcal{M}$. 
An \emph{${\mathbb{R}}$-point} $p \colon \mathcal{M} \rightarrow \mathbb{R}$ of $\mathcal{M}$ and the \emph{space ${\rm Sp}_{\mathbb R}(\mathcal{M}) (\supseteq {\rm Sp}(\mathcal{M}))$ of ${\mathbb{R}}$-points} of $\mathcal{M}$ are similarly defined. 
\end{defi}

We define $\mathcal{M}_{\mathbb{R}}$ by replacing $M_\alpha = \mathbb{Z}^r$ with $(M_\alpha)_\mathbb{R} = M_\alpha \otimes_{\mathbb{Z}} \mathbb{R} = \mathbb{R}^r$ in the definition of $\mathcal{M}$. 
Note that each $\mu_{\alpha,\beta} \colon M_\alpha \rightarrow M_\beta$ is naturally extended to a bijective piecewise-linear map $\mu_{\alpha,\beta} \colon (M_\alpha)_\mathbb{R} \rightarrow (M_\beta)_\mathbb{R}$.
Then there exists a minimum fan $\Sigma(\mathcal{M},\mu_{\alpha,\beta})$ in $(M_\alpha)_\mathbb{R}$ such that $\mu_{\alpha,\beta}|_C$ is $\mathbb{R}$-linear for each $C \in \Sigma(\mathcal{M},\mu_{\alpha,\beta})$. 
Let $V(\mathcal{M},\mu_{\alpha,\beta})$ be the linearity space of the fan $\Sigma(\mathcal{M},\mu_{\alpha,\beta})$, that is, $V(\mathcal{M},\mu_{\alpha,\beta})$ is the $\mathbb{R}$-linear subspace of $(M_\alpha)_\mathbb{R}$ determined by the condition that $m_1 + m_2 \in C$ for all $C \in \Sigma(\mathcal{M},\mu_{\alpha,\beta})$, $m_1 \in C$, and $m_2 \in V(\mathcal{M},\mu_{\alpha,\beta})$. 
Denote by $\Sigma(\mathcal{M},\alpha)$ the fan in $(M_\alpha)_\mathbb{R}$ given as the common refinement of $\Sigma(\mathcal{M},\mu_{\alpha,\beta})$ for all $\beta \in \mathcal{I}$. 
Following \cite[Section 2]{EHM}, the \emph{PL fan} $\Sigma(\mathcal{M})$ of $\mathcal{M}$ is defined by 
\[\Sigma(\mathcal{M}) = \{\pi_\alpha^{-1} (C) \mid C \in \Sigma(\mathcal{M},\alpha)\}.\]
By \cite[Lemma 2.10 (d)]{EHM}, this is independent of the choice of $\alpha$. 
In particular, for all $\beta_1, \beta_2 \in \mathcal{I}$, $\mu_{\beta_1,\beta_2}|_{\pi_{\beta_1}(\mathcal{C})}$ is $\mathbb{R}$-linear for each $\mathcal{C} \in \Sigma(\mathcal{M})$.
Let $V(\mathcal{M}, \alpha) \coloneqq \bigcap_{\beta \in \mathcal{I}} V(\mathcal{M},\mu_{\alpha,\beta}) \subseteq (M_\alpha)_\mathbb{R}$, which is the linearity space of $\Sigma(\mathcal{M},\alpha)$.
We can naturally extend every $\mathbb{R}$-point $p \in {\rm Sp}_{\mathbb R}(\mathcal{M})$ to a function $\mathcal{M}_\mathbb{R} \rightarrow \mathbb{R}$, which is also written as $p$ by abuse of notation.
For $\alpha \in \mathcal{I}$, let ${\rm Sp}_{\mathbb R}(\mathcal{M}, \alpha)$ denote the set of $p \in {\rm Sp}_{\mathbb R}(\mathcal{M})$ such that $p \circ \pi_\alpha^{-1} \colon (M_\alpha)_\mathbb{R} \rightarrow \mathbb{R}$ is $\mathbb{R}$-linear, where $\pi_\alpha \colon \mathcal{M}_{\mathbb{R}} \rightarrow (M_\alpha)_{\mathbb{R}}$ is defined in a way similar to $\pi_\alpha \colon \mathcal{M} \rightarrow M_\alpha$. 

\begin{defi}[{\cite[Definition 4.1]{EHM}}]\label{d:strict_dual}
Let $\mathcal{N}$ be another finite polyptych lattice of rank $r$ (over $\mathbb{Z}$).
Then a pair of maps ${\mathsf v} \colon \mathcal{M}_\mathbb{R} \rightarrow {\rm Sp}_\mathbb{R}(\mathcal{N})$ and ${\mathsf w} \colon \mathcal{N}_\mathbb{R} \rightarrow {\rm Sp}_\mathbb{R}(\mathcal{M})$ is called a \emph{strict dual pairing} if the following conditions hold:
\begin{enumerate}
\item[(i)] ${\mathsf v} (\mathcal{M}) \subseteq {\rm Sp}(\mathcal{N})$ and ${\mathsf w} (\mathcal{N}) \subseteq {\rm Sp}(\mathcal{M})$;
\item[(ii)] $({\mathsf v} (m)) (n) = ({\mathsf w} (n)) (m)$ for all $m \in \mathcal{M}$ and $n \in \mathcal{N}$;
\item[(iii)] ${\mathsf v} \colon \mathcal{M} \rightarrow {\rm Sp}(\mathcal{N})$ and ${\mathsf w} \colon \mathcal{N} \rightarrow {\rm Sp}(\mathcal{M})$ are both bijective;
\item[(iv)] $\gamma \mapsto {\mathsf v}^{-1} ({\rm Sp}_\mathbb{R}(\mathcal{N}, \gamma))$ (respectively, $\alpha \mapsto {\mathsf w}^{-1} ({\rm Sp}_\mathbb{R}(\mathcal{M}, \alpha))$) gives a bijective map from $\pi(\mathcal{N})$ (respectively, $\pi(\mathcal{M})$) to the set of $r$-dimensional cones in $\Sigma(\mathcal{M})$ (respectively, $\Sigma(\mathcal{N})$).
\end{enumerate}
In this case, we call $(\mathcal{M}, \mathcal{N}, {\mathsf v}, {\mathsf w})$ (or simply $(\mathcal{M}, \mathcal{N})$) a \emph{strict dual pair} of polyptych lattices.
\end{defi}

For $\varphi \in {\rm Sp}_{\mathbb{R}}(\mathcal{M})$ and $a \in \mathbb{R}$, we set 
\[\mathcal{H}_{\varphi, a} \coloneqq \{m \in \mathcal{M}_\mathbb{R} \mid \varphi(m) \geq a\}. \]
Following \cite[Definition 3.22]{EHM}, we define the \emph{point-convex hull} $\text{p-conv}_\mathbb{Z} (S)$ of $\mathcal{S} \subseteq \mathcal{M}$ over $\mathbb{Z}$ by 
\[\text{p-conv}_\mathbb{Z} (S) \coloneqq \bigcap_{\substack{\varphi  \in {\rm Sp}(\mathcal{M}), a \in \mathbb{Z};\\ S \subseteq \mathcal{H}_{\varphi, a}}} \mathcal{H}_{\varphi, a}.\]
Similarly, $\text{p-conv}_\mathbb{R} (S)$ is defined by replacing $\varphi \in {\rm Sp}(\mathcal{M})$ and $a \in \mathbb{Z}$ with $\varphi \in {\rm Sp}_\mathbb{R}(\mathcal{M})$ and $a \in \mathbb{R}$. 
Let $S_{\mathcal{M}}$ be the canonical $\mathbb{Z}_{\geq 0}$-semialgebra defined in \cite[Definition 3.35]{EHM} with additive operation $\oplus$ and with product operation $\star$, which has the additive identity $\infty$ and the multiplicative identity $0$.
Every element of $S_{\mathcal{M}} \setminus \{\infty\}$ can be written as $\bigoplus_{m \in \mathcal{S}} m$ for some finite subset $\mathcal{S} \subseteq \mathcal{M}$, and we have $\bigoplus_{m \in \mathcal{S}} m = \bigoplus_{m^\prime \in \mathcal{S}^\prime} m^\prime$ for finite subsets $\mathcal{S}, \mathcal{S}^\prime \subseteq \mathcal{M}$ if and only if $\text{p-conv}_\mathbb{Z} (\mathcal{S}) = \text{p-conv}_\mathbb{Z} (\mathcal{S}^\prime)$. 
In addition, for $m_1, m_2 \in \mathcal{M}$, we have 
\[m_1 \star m_2 \coloneqq \bigoplus_{m \in \Upsilon(m_1, m_2)} m,\]
where $\Upsilon(m_1, m_2) \coloneqq \{m_1 +_\alpha m_2 \mid \alpha \in \pi(\mathcal{M})\}$; see \cite[Definition 3.31]{EHM}. 
The semialgebra $S_{\mathcal{M}}$ is \emph{idempotent} in the sense that $a \oplus a = a$ for all $a \in S_{\mathcal{M}}$. 
Hence it is equipped with a partial order defined by $a \leq b$ if and only if $a \oplus b = a$ for $a, b \in S_{\mathcal{M}}$. 
Let $\Bbbk$ be a field, and $\Bbbk^\times \coloneqq \Bbbk \setminus \{0\}$.

\begin{defi}[{\cite[Definition 6.3]{EHM}}]\label{d:detropicalization}
A \emph{detropicalization} $(\mathcal{A}, \nu)$ of $\mathcal{M}$ is a pair of 
\begin{itemize}
\item a Noetherian $\Bbbk$-algebra $\mathcal{A}$ that is an integral domain, and
\item a valuation $\nu \colon \mathcal{A} \rightarrow S_{\mathcal{M}}$
\end{itemize}
such that every element of $\mathcal{M}$ is included in $\nu (\mathcal{A})$ and such that the Krull dimension of $\mathcal{A}$ equals the rank $r$ of $\mathcal{M}$, where 
a map $\nu \colon \mathcal{A} \rightarrow S_{\mathcal{M}}$ is said to be a \emph{valuation} on $\mathcal{A}$ with values in $S_{\mathcal{M}}$ if for $f, g \in \mathcal{A}$ and $c \in \Bbbk^\times$, 
\begin{enumerate}
\item[(i)] $\nu (f g) = \nu (f) \star \nu (g)$, 
\item[(ii)] $\nu (f + g) \geq \nu (f) \oplus \nu (g)$, 
\item[(iii)] $\nu (c f) = \nu (f)$, 
\item[(iv)] $\nu (f) = \infty$ if and only if $f = 0$.
\end{enumerate}
\end{defi}

\begin{defi}[{\cite[Definition 6.5]{EHM}}]\label{d:adapted_basis}
Let $(\mathcal{A}, \nu)$ be a detropicalization of $\mathcal{M}$.
A $\Bbbk$-basis ${\bf B}$ of $\mathcal{A}$ is a \emph{convex adapted basis} for $\nu$ if 
\begin{enumerate}
\item[(i)] $\nu (\sum_i c_i b_i) = \bigoplus_i \nu (b_i)$\quad for each finite collection $c_i \in \Bbbk^\times$ and $b_i \in {\bf B}$,
\item[(ii)] $\nu (b)$ is included in $\mathcal{M}$ for all $b \in {\bf B}$.
\end{enumerate}
\end{defi}

\section{Relation with marked chain-order polytopes}

\subsection{Marked chain-order polytopes}\label{ss:marked_poset_polytopes}

In this subsection, we recall some basic definitions and properties of marked chain-order polytopes, following \cite{FFLP, FFP}. 
Let $\Pi$ be a finite poset with a partial order $\preceq$, and take a subset $\Pi^\ast \subseteq \Pi$ of $\Pi$ such that all minimal and maximal elements in $\Pi$ are included in $\Pi^\ast$. 
For an element $\lambda = (\lambda_a)_{a \in \Pi^\ast} \in \r^{\Pi^\ast}$, called a \emph{marking}, such that $\lambda_a \leq \lambda_b$ if $a \preceq b$ in $\Pi$, we call the triple $(\Pi, \Pi^\ast, \lambda)$ a \emph{marked poset}.
For $p, q \in \Pi$, write $q \lessdot p$ if $p$ \emph{covers} $q$, that is, $q \prec p$ and there is no $q' \in \Pi \setminus \{p, q\}$ with $q \prec q' \prec p$. 

\begin{defi}[{\cite[Section 1.3]{FFLP}}]
For a partition $\Pi \setminus \Pi^\ast = \mathcal{C} \sqcup \mathcal{O}$, the \emph{marked chain-order polytope} $\Delta_{\mathcal{C}, \mathcal{O}} (\Pi, \Pi^\ast, \lambda)$ is defined to be the set of $(x_p)_{p \in \Pi \setminus \Pi^\ast} \in \r^{\Pi \setminus \Pi^\ast}$ satisfying the following conditions:
\begin{itemize}
\item $x_p \geq 0$ for all $p \in \mathcal{C}$, 
\item $\sum_{i = 1}^\ell x_{p_i} \leq y_b - y_a$ for $a \lessdot p_1 \lessdot \cdots \lessdot p_\ell \lessdot b$ with $p_i \in \mathcal{C}$ and $a,b \in \Pi^\ast \sqcup \mathcal{O}$,
\end{itemize}
where we set 
\[y_c \coloneqq 
\begin{cases} 
\lambda_c &\text{if}\ c \in \Pi^\ast,\\ 
x_c &\text{if}\ c \in \mathcal{O}
\end{cases}\]
for $c \in \Pi^\ast \sqcup \mathcal{O}$. 
\end{defi}

By definition, $\Delta_{\emptyset, \Pi \setminus \Pi^\ast}(\Pi, \Pi^\ast, \lambda)$ (respectively, $\Delta_{\Pi \setminus \Pi^\ast, \emptyset}(\Pi, \Pi^\ast, \lambda)$) coincides with the marked order polytope $\mathcal{O}(\Pi, \Pi^\ast, \lambda)$ (respectively, the marked chain polytope $\mathcal{C}(\Pi, \Pi^\ast, \lambda)$) introduced in \cite[Definition 1.2]{ABS}. 
When $\lambda \in \z^{\Pi^\ast}$, it follows by \cite[Proposition 2.4]{FFLP} that $\Delta_{\mathcal{C}, \mathcal{O}}(\Pi, \Pi^\ast, \lambda)$ is an integral polytope for each partition $\Pi \setminus \Pi^\ast = \mathcal{C} \sqcup \mathcal{O}$. 
In the present paper, we always assume that $\lambda \in \z^{\Pi^\ast}$.
Generalizing Stanley's and Ardila--Bliem--Salazar's transfer maps \cite{Sta, ABS}, Fang--Fourier--Litza--Pegel \cite{FFLP} gave a \emph{transfer map} $\phi_{\mathcal{C}, \mathcal{O}}$ for marked chain-order polytopes. 
More precisely, they defined a piecewise-affine transformation $\phi_{\mathcal{C}, \mathcal{O}} \colon \r^{\Pi \setminus \Pi^\ast} \rightarrow \r^{\Pi \setminus \Pi^\ast}$, $(x_p)_p \mapsto (x_p^\prime)_p$, by 
\begin{align*}
&x_p^\prime \coloneqq 
\begin{cases}
x_p + \min(\{-x_q \mid q \lessdot p, \ q \in \Pi \setminus \Pi^\ast\} \cup \{-\lambda_q \mid q \lessdot p, \ q \in \Pi^\ast\}) &\text{if}\ p \in \mathcal{C}, \\
x_p &\text{otherwise}
\end{cases}
\end{align*}
for $p \in \Pi \setminus \Pi^\ast$.

\begin{thm}[{see \cite[Theorem 2.1 and Corollary 2.5]{FFLP}}]\label{t:marked_CO_transfer}
For every partition $\Pi \setminus \Pi^\ast = \mathcal{C} \sqcup \mathcal{O}$, the piecewise-affine transformation $\phi_{\mathcal{C}, \mathcal{O}}$ gives a bijective map from $\mathcal{O}(\Pi, \Pi^\ast, \lambda)$ to $\Delta_{\mathcal{C}, \mathcal{O}}(\Pi, \Pi^\ast, \lambda)$, which induces a bijective map between the sets of lattice points.
\end{thm}

Assume that $\Pi$ is \emph{graded}, that is, every maximal chain in $\Pi$ has the same length. 
Then all chains starting from a minimal element in $\Pi$ and ending at $p$ have the same length for each $p \in \Pi$; 
the length of such chains is denoted by $r(p)$. 
Assume that the marking $\lambda$ satisfies $\lambda_a=\lambda_b$ for all $a,b \in \Pi^\ast$ with $r(a)=r(b)$. 
For instance, the marking $\lambda^r$ given by $(\lambda^r)_a = r(a)$ for $a \in \Pi^\ast$ satisfies this condition. 
Then there exists ${\bm u} = (u_p)_{p \in \Pi \setminus \Pi^\ast} \in {\mathcal O}(\Pi, \Pi^\ast, \lambda)  \cap \z^{\Pi \setminus \Pi^\ast}$ such that
\begin{equation}\label{eq:assumption}
\begin{split}
u_p &= u_{p'} \text{ for all }p,p' \in \Pi \setminus \Pi^\ast \text{ with } r(p)=r(p'), \text{ and }\\
u_p &= \lambda_a \text{ for all }p \in \Pi \setminus \Pi^\ast \text{ and }a \in \Pi^\ast\text{ with }r(p)=r(a). 
\end{split}
\end{equation}
Fix such ${\bm u}$, and denote the translation ${\mathcal O}(\Pi, \Pi^\ast, \lambda) - {\bm u}$ by $\widehat{\mathcal O}(\Pi, \Pi^\ast, \lambda)$. 
Similarly, we set
\[\widehat{\Delta}_{\mathcal{C}}(\Pi, \Pi^\ast, \lambda) \coloneqq \Delta_{\mathcal{C}, \mathcal{O}}(\Pi, \Pi^\ast, \lambda)  - \phi_{\mathcal{C}, \mathcal{O}}({\bm u})\]
for $\mathcal{C} \in 2^{\Pi \setminus \Pi^\ast}$, where $\mathcal{O} \coloneqq \Pi \setminus (\Pi^\ast \cup \mathcal{C})$.
Define a piecewise-linear transformation $\mu_{\mathcal{C}} \colon \r^{\Pi \setminus \Pi^\ast} \rightarrow \r^{\Pi \setminus \Pi^\ast}$, $(x_p)_p \mapsto (x_p^\prime)_p$, by 
\begin{equation}\label{eq:transfer_linear}
\begin{aligned}
&x_p^\prime \coloneqq 
\begin{cases}
x_p + \min(\{-x_q \mid q \lessdot p, \ q \in \Pi \setminus \Pi^\ast\} \cup \{0 \mid q \lessdot p, \ q \in \Pi^\ast\}) &\text{if}\ p \in \mathcal{C}, \\
x_p &\text{otherwise}
\end{cases}
\end{aligned}
\end{equation}
for $p \in \Pi \setminus \Pi^\ast$.
Then the authors \cite[Proof of Theorem 5.3]{FH} proved that the piecewise-linear transformation $\mu_{\mathcal{C}}$ can be described as a composition of dual operations of combinatorial mutations. 
In addition, it is proved in \cite[Proof of Theorem 5.3]{FH} that $\mu_{\mathcal{C}} \colon \r^{\Pi \setminus \Pi^\ast} \rightarrow \r^{\Pi \setminus \Pi^\ast}$ coincides with a composition of the transfer map $\phi_{\mathcal{C}, \Pi \setminus (\Pi^\ast \cup \mathcal{C})}$ with translations. 
More precisely, we have 
\[\mu_{\mathcal{C}} ({\bm a}) = \phi_{\mathcal{C}, \Pi \setminus (\Pi^\ast \cup \mathcal{C})} ({\bm a} + {\bm u}) - \phi_{\mathcal{C}, \Pi \setminus (\Pi^\ast \cup \mathcal{C})}({\bm u})\]
for all ${\bm a} \in \r^{\Pi \setminus \Pi^\ast}$. 
In particular, the map $\mu_{\mathcal{C}}$ gives a bijective map from $\widehat{\mathcal{O}}(\Pi, \Pi^\ast, \lambda)$ to $\widehat{\Delta}_{\mathcal{C}}(\Pi, \Pi^\ast, \lambda)$, which induces a bijective map between the sets of lattice points. 

\subsection{Marked chain-order polytopes as PL polytopes}

We take $\mathcal{I}$ to be the power set $2^{\Pi \setminus \Pi^\ast}$ of $\Pi \setminus \Pi^\ast$. 
For each $\mathcal{C} \in 2^{\Pi \setminus \Pi^\ast}$, which is a subset of $\Pi \setminus \Pi^\ast$, set $M_{\mathcal{C}} \coloneqq \mathbb{Z}^{\Pi \setminus \Pi^\ast}$.
For $\mathcal{C}_1, \mathcal{C}_2 \in 2^{\Pi \setminus \Pi^\ast}$, define $\mu_{\mathcal{C}_1, \mathcal{C}_2} \colon M_{\mathcal{C}_1} \rightarrow M_{\mathcal{C}_2}$ by 
\[\mu_{\mathcal{C}_1, \mathcal{C}_2} \coloneqq \mu_{\mathcal{C}_2} \circ \mu_{\mathcal{C}_1}^{-1}.\]
Then it follows easily by definition that $\mathcal{M} = (\{M_{\mathcal{C}}\}_{\mathcal{C} \in 2^{\Pi \setminus \Pi^\ast}}, \{\mu_{\mathcal{C}_1, \mathcal{C}_2}\}_{\mathcal{C}_1, \mathcal{C}_2 \in 2^{\Pi \setminus \Pi^\ast}})$ becomes a finite polyptych lattice over $\mathbb{Z}$. 
Note that the polyptych lattice $\mathcal{M}$ is independent of the choice of the marking $\lambda$.
For $p \in \Pi \setminus \Pi^\ast$, define $\varepsilon_{p}, \varepsilon_{p}^\prime \in \mathcal{M}$ by $\pi_\emptyset (\varepsilon_{p}) = {\bm e}_{p} \in M_\emptyset$ and $\pi_\emptyset (\varepsilon_{p}^\prime) = -{\bm e}_{p} \in M_\emptyset$, where ${\bm e}_p \in \mathbb{Z}^{\Pi \setminus \Pi^\ast}$ denotes the unit vector corresponding to $p \in \Pi \setminus \Pi^\ast$.
Since $\mu_{\mathcal{C}} (\widehat{\mathcal{O}}(\Pi, \Pi^\ast, \lambda)) = \widehat{\Delta}_{\mathcal{C}}(\Pi, \Pi^\ast, \lambda)$, we obtain a subset $\widehat{\Delta}(\Pi, \Pi^\ast, \lambda) \subseteq \mathcal{M}_\mathbb{R}$ such that $\pi_{\mathcal{C}} (\widehat{\Delta}(\Pi, \Pi^\ast, \lambda)) = \widehat{\Delta}_{\mathcal{C}}(\Pi, \Pi^\ast, \lambda)$ for all $\mathcal{C} \in 2^{\Pi \setminus \Pi^\ast}$.
For $p \in \Pi \setminus \Pi^\ast$, define $\varphi_p \colon \mathcal{M} \rightarrow \mathbb{Z}$ by 
\[\varphi_p (m) \coloneqq y_p\]
for $m \in \mathcal{M}$, 
where $\pi_{\{p\}} (m) = (y_q)_{q \in \Pi \setminus \Pi^\ast}$. 
Writing $\pi_{\emptyset} (m)$ as $\pi_{\emptyset} (m) = (x_q)_{q \in \Pi \setminus \Pi^\ast}$, we have 
\[\varphi_p (m) = x_p + \min(\{-x_q \mid q \lessdot p, \ q \in \Pi \setminus \Pi^\ast\} \cup \{0 \mid q \lessdot p, \ q \in \Pi^\ast\})\]
by the definition of $\mu_{\emptyset, \{p\}} = \mu_{\{p\}} \colon M_{\emptyset} \rightarrow M_{\{p\}}$.

\begin{prop}
For $p \in \Pi \setminus \Pi^\ast$, it follows that $\varphi_p \in {\rm Sp}(\mathcal{M})$.
\end{prop}

\begin{proof}
By the definition of $\varphi_p$, we have $\varphi_p (k m) = k \varphi_p(m)$ for all $m \in \mathcal{M}$ and $k \in \mathbb{Z}_{\geq 0}$. 
Take $m_1, m_2 \in \mathcal{M}$ and $\mathcal{C} \in 2^{\Pi \setminus \Pi^\ast}$. 
For $i = 1, 2$, we write $\pi_{\{p\}} (m_i) = (y_q^{(i)})_{q \in \Pi \setminus \Pi^\ast}$ and $\pi_{\emptyset} (m) = (x_q^{(i)})_{q \in \Pi \setminus \Pi^\ast}$. 
By the definition of $\mu_{\mathcal{C}}$, we have $\varphi_p (m_1 +_{\mathcal{C}} m_2) = y_p^{(1)} + y_p^{(2)}$ if $p \in \mathcal{C}$ and 
\[\varphi_p (m_1 +_{\mathcal{C}} m_2) = x_p^{(1)} + x_p^{(2)} + \min(\{-x_q^{(1)}-x_q^{(2)} \mid q \lessdot p, \ q \in \Pi \setminus \Pi^\ast\} \cup \{0 \mid q \lessdot p, \ q \in \Pi^\ast\}) 
\]
if $p \notin \mathcal{C}$. 
Since 
\begin{align*}
&\min(\{-x_q^{(1)} \mid q \lessdot p, \ q \in \Pi \setminus \Pi^\ast\} \cup \{0 \mid q \lessdot p, \ q \in \Pi^\ast\}) \\
&+ \min(\{-x_q^{(2)} \mid q \lessdot p, \ q \in \Pi \setminus \Pi^\ast\} \cup \{0 \mid q \lessdot p, \ q \in \Pi^\ast\}) \\
\leq &\min(\{-x_q^{(1)}-x_q^{(2)} \mid q \lessdot p, \ q \in \Pi \setminus \Pi^\ast\} \cup \{0 \mid q \lessdot p, \ q \in \Pi^\ast\}), 
\end{align*}
it follows that 
\begin{align*}
y_p^{(1)} + y_p^{(2)} \leq x_p^{(1)} + x_p^{(2)} + \min(\{-x_q^{(1)}-x_q^{(2)} \mid q \lessdot p, \ q \in \Pi \setminus \Pi^\ast\} \cup \{0 \mid q \lessdot p, \ q \in \Pi^\ast\}).
\end{align*}
Hence we have 
\[\min \{\varphi_p (m_1 +_{\mathcal{C}} m_2) \mid \mathcal{C} \in 2^{\Pi \setminus \Pi^\ast}\} = y_p^{(1)} + y_p^{(2)} = \varphi_p (m_1) + \varphi_p (m_2),\]
which implies that $\varphi_p \in {\rm Sp}(\mathcal{M})$. 
\end{proof}

For $p \in \Pi^\ast$ and $p^\prime \in \Pi \setminus \Pi^\ast$ with $p^\prime \lessdot p$, define $\varphi_{p, p^\prime} \colon \mathcal{M} \rightarrow \mathbb{Z}$ by 
\[\varphi_{p, p^\prime} (m) \coloneqq -x_{p^\prime}\]
for $m \in \mathcal{M}$, 
where $\pi_{\emptyset} (m) = (x_q)_{q \in \Pi \setminus \Pi^\ast}$. 

\begin{prop}
For $p \in \Pi^\ast$ and $p^\prime \in \Pi \setminus \Pi^\ast$ with $p^\prime \lessdot p$, it follows that $\varphi_{p, p^\prime} \in {\rm Sp}(\mathcal{M})$.
\end{prop}

\begin{proof}
By the definition of $\varphi_{p, p^\prime}$, we have $\varphi_{p, p^\prime} (k m) = k \varphi_{p, p^\prime}(m)$ for all $m \in \mathcal{M}$ and $k \in \mathbb{Z}_{\geq 0}$. 
Take $m_1, m_2 \in \mathcal{M}$ and $\mathcal{C} \in 2^{\Pi \setminus \Pi^\ast}$. 
For $i = 1, 2$, we write $\pi_{\emptyset} (m) = (x_q^{(i)})_{q \in \Pi \setminus \Pi^\ast}$. 
By the definition of $\mu_{\mathcal{C}}$, we have $\varphi_{p, p^\prime} (m_1 +_{\mathcal{C}} m_2) = -x_{p^\prime}^{(1)} - x_{p^\prime}^{(2)}$ if $p^\prime \notin \mathcal{C}$ and 
\begin{align*}
\varphi_{p, p^\prime} (m_1 +_{\mathcal{C}} m_2) = &-x_{p^\prime}^{(1)} - x_{p^\prime}^{(2)} - \min(\{-x_q^{(1)} \mid q \lessdot p^\prime, \ q \in \Pi \setminus \Pi^\ast\} \cup \{0 \mid q \lessdot p^\prime, \ q \in \Pi^\ast\})\\ 
&- \min(\{-x_q^{(2)} \mid q \lessdot p^\prime, \ q \in \Pi \setminus \Pi^\ast\} \cup \{0 \mid q \lessdot p^\prime, \ q \in \Pi^\ast\}) \\
&+ \min(\{-x_q^{(1)}-x_q^{(2)} \mid q \lessdot p^\prime, \ q \in \Pi \setminus \Pi^\ast\} \cup \{0 \mid q \lessdot p^\prime, \ q \in \Pi^\ast\}) 
\end{align*}
if $p^\prime \in \mathcal{C}$. 
Since 
\begin{align*}
&\min(\{-x_q^{(1)} \mid q \lessdot p^\prime, \ q \in \Pi \setminus \Pi^\ast\} \cup \{0 \mid q \lessdot p^\prime, \ q \in \Pi^\ast\}) \\
&+ \min(\{-x_q^{(2)} \mid q \lessdot p^\prime, \ q \in \Pi \setminus \Pi^\ast\} \cup \{0 \mid q \lessdot p^\prime, \ q \in \Pi^\ast\}) \\
\leq &\min(\{-x_q^{(1)}-x_q^{(2)} \mid q \lessdot p^\prime, \ q \in \Pi \setminus \Pi^\ast\} \cup \{0 \mid q \lessdot p^\prime, \ q \in \Pi^\ast\}),
\end{align*}
it follows that 
\[\min \{\varphi_{p, p^\prime} (m_1 +_{\mathcal{C}} m_2) \mid \mathcal{C} \in 2^{\Pi \setminus \Pi^\ast}\} = -x_{p^\prime}^{(1)} - x_{p^\prime}^{(2)} = \varphi_{p, p^\prime} (m_1) + \varphi_{p, p^\prime} (m_2),\]
which implies that $\varphi_{p, p^\prime} \in {\rm Sp}(\mathcal{M})$. 
\end{proof}

Recall that $\Pi$ is graded and that we fix ${\bm u} = (u_p)_{p \in \Pi \setminus \Pi^\ast} \in {\mathcal O}(\Pi, \Pi^\ast, \lambda)  \cap \z^{\Pi \setminus \Pi^\ast}$ satisfying \eqref{eq:assumption} to define $\widehat{\Delta}_{\mathcal{C}}(\Pi, \Pi^\ast, \lambda)$. 
For $p \in \Pi \setminus \Pi^\ast$, since $p$ is not minimal, there exists $q \in \Pi$ such that $q \lessdot p$.
Then we set $a_p \coloneqq u_q - u_p \in \mathbb{Z}$, where $u_q \coloneqq \lambda_q$ if $q \in \Pi^\ast$.
This $a_p$ is independent of the choice of $q \in \Pi$ such that $q \lessdot p$.
For $p \in \Pi^\ast$ and $p^\prime \in \Pi \setminus \Pi^\ast$ with $p^\prime \lessdot p$, define $a_{p,p^\prime} \in \mathbb{Z}$ by $a_{p,p^\prime} \coloneqq u_{p^\prime} - \lambda_p$.

\begin{prop}\label{p:PL_description}
With the notation above, it holds that 
\[\widehat{\Delta}(\Pi, \Pi^\ast, \lambda) = \left(\bigcap_{p \in \Pi \setminus \Pi^\ast} \mathcal{H}_{\varphi_p, a_p}\right) \cap \left(\bigcap_{p \in \Pi^\ast, p^\prime \in \Pi \setminus \Pi^\ast; p^\prime \lessdot p} \mathcal{H}_{\varphi_{p, p^\prime} , a_{p, p^\prime}}\right).\]
In particular, the set $\widehat{\Delta}(\Pi, \Pi^\ast, \lambda) \subseteq \mathcal{M}_{\mathbb{R}}$ is a \emph{PL polytope} over $\mathbb{Z}$ in the sense of \cite[Definition 5.1]{EHM}.
\end{prop}

\begin{proof}
Since $\pi_\emptyset (\widehat{\Delta}(\Pi, \Pi^\ast, \lambda)) = \widehat{\mathcal{O}}(\Pi, \Pi^\ast, \lambda)$, it suffices to prove that 
\begin{equation}\label{eq:PL_description_order}
\widehat{\mathcal{O}}(\Pi, \Pi^\ast, \lambda) = \left(\bigcap_{p \in \Pi \setminus \Pi^\ast} \pi_\emptyset (\mathcal{H}_{\varphi_p, a_p})\right) \cap \left(\bigcap_{p \in \Pi^\ast, p^\prime \in \Pi \setminus \Pi^\ast; p^\prime \lessdot p} \pi_\emptyset (\mathcal{H}_{\varphi_{p, p^\prime} , a_{p, p^\prime}})\right).
\end{equation}
For $p \in \Pi \setminus \Pi^\ast$, we see that $\pi_\emptyset (\mathcal{H}_{\varphi_p, a_p})$ is the set of $(x_q)_{q \in \Pi \setminus \Pi^\ast} \in M_\emptyset$ satisfying 
\[x_p + \min(\{-x_q \mid q \lessdot p, \ q \in \Pi \setminus \Pi^\ast\} \cup \{0 \mid q \lessdot p, \ q \in \Pi^\ast\}) \geq a_p,\]
that is, $x_p + u_p \geq x_q + u_q$ for all $q \in \Pi \setminus \Pi^\ast$ with $q \lessdot p$ and $x_p + u_p \geq \lambda_q$ for all $q \in \Pi^\ast$ with $q \lessdot p$. 
For $p \in \Pi^\ast$ and $p^\prime \in \Pi \setminus \Pi^\ast$ with $p^\prime \lessdot p$, we see that $\pi_\emptyset (\mathcal{H}_{\varphi_{p, p^\prime} , a_{p, p^\prime}})$ is the set of $(x_q)_{q \in \Pi \setminus \Pi^\ast} \in M_\emptyset$ satisfying $-x_{p^\prime} \geq a_{p, p^\prime}$, that is, $\lambda_p \geq x_{p^\prime} + u_{p^\prime}$. 
Since $\widehat{\mathcal{O}}(\Pi, \Pi^\ast, \lambda)$ is defined by these inequalities, we conclude \eqref{eq:PL_description_order}.
\end{proof}

\begin{rem}
Consider the marking $\lambda^r = ((\lambda^r)_a)_{a \in \Pi^\ast}$ given by $(\lambda^r)_a = r(a)$ for $a \in \Pi^\ast$, and define ${\bm u} = (u_p)_{p \in \Pi \setminus \Pi^\ast} \in {\mathcal O}(\Pi, \Pi^\ast, \lambda)  \cap \z^{\Pi \setminus \Pi^\ast}$ by $u_p \coloneqq r(p)$, which satisfies \eqref{eq:assumption}.
In this case, we have $a_p = -1$ (respectively, $a_{p,p^\prime} = -1$) for all $p \in \Pi \setminus \Pi^\ast$ (respectively, for all $p \in \Pi^\ast$ and $p^\prime \in \Pi \setminus \Pi^\ast$ with $p^\prime \lessdot p$).
Hence the PL polytope $\widehat{\Delta}(\Pi, \Pi^\ast, \lambda) \subseteq \mathcal{M}_{\mathbb{R}}$ is \emph{chart-Gorenstein-Fano} in the sense of \cite[Definition 5.21]{EHM}.
\end{rem}

\subsection{Marked posets and spaces of points}\label{ss:basic_marked_posets}

Let us consider the poset $\Pi_1$ whose Hasse diagram is given in Figure \ref{Hasse_nonmarked}, where the circles denote the elements of $\Pi_1$, and we write 
\[\Pi_1 = \{q_1, \ldots, q_n, p_1, \ldots, p_{n+1}\}.\]
Adding a new minimum element $\hat{0}$ and maximum element $\hat{1}$ to $\Pi_1$, we have $\widehat{\Pi}_1 \coloneqq \Pi_1 \sqcup \{\hat{0}, \hat{1}\}$.
The poset $\widehat{\Pi}_1$ is regarded as a marked poset with $\widehat{\Pi}_1^\ast = \{\hat{0}, \hat{1}\}$ and with some marking $\lambda$.
\begin{figure}[!ht]
\begin{center}
   \includegraphics[width=7.0cm,bb=50mm 210mm 160mm 230mm,clip]{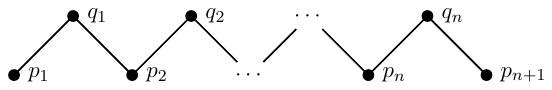}
	\caption{The Hasse diagram of $\Pi_1$.}
	\label{Hasse_nonmarked}
\end{center}
\end{figure}

Let $\mathcal{M}^{(1)}$ denote the polyptych lattice corresponding to $(\widehat{\Pi}_1, \widehat{\Pi}_1^\ast, \lambda)$, and write $M_{\mathcal{C}}^{(1)} \coloneqq \pi_{\mathcal{C}} (\mathcal{M}^{(1)})$ for $\mathcal{C} \in 2^{\widehat{\Pi}_1 \setminus \widehat{\Pi}_1^\ast} = 2^{\Pi_1}$.
Note that it suffices to consider $\mu_\mathcal{C}$ only for $\mathcal{C} \subseteq \{q_1, \ldots, q_n\}$ since $\mu_\mathcal{C} = \mu_{\mathcal{C} \cap \{q_1, \ldots, q_n\}}$ for $\mathcal{C} \in 2^{\widehat{\Pi}_1 \setminus \widehat{\Pi}_1^\ast}$.
For $\epsilon_1, \ldots, \epsilon_n \in \{\pm 1\}$, define $\sigma_{\epsilon_1, \ldots, \epsilon_n} \subseteq (\mathcal{M}^{(1)})_\mathbb{R}$ so that $\pi_\emptyset (\sigma_{\epsilon_1, \ldots, \epsilon_n}) \subseteq (M_\emptyset^{(1)})_{\mathbb{R}}$ is given by $\epsilon_1 y_1, \ldots, \epsilon_n y_n \geq 0$ for $(x_q)_{q \in \widehat{\Pi}_1 \setminus \widehat{\Pi}_1^\ast} \in (M_\emptyset^{(1)})_{\mathbb{R}}$, where $y_k \coloneqq x_{p_k} - x_{p_{k+1}}$ for $1 \leq k \leq n$.
For each $\mathcal{C} \in 2^{\widehat{\Pi}_1 \setminus \widehat{\Pi}_1^\ast}$, we see by the definition of $\mu_{\mathcal{C}}$ that $\mu_{\mathcal{C}}$ is linear on the cone $\pi_\emptyset (\sigma_{\epsilon_1, \ldots, \epsilon_n})$.
More strongly, the set $\{\sigma_{\epsilon_1, \ldots, \epsilon_n} \mid \epsilon_1, \ldots, \epsilon_n \in \{\pm 1\}\}$ forms a complete list of maximal cones of the PL fan $\Sigma(\mathcal{M}^{(1)})$. 
In particular, the linearity space $V(\mathcal{M}^{(1)}, \emptyset)$ is spanned by 
\[{\bm e}_{p_1} + \cdots + {\bm e}_{p_{n+1}}, {\bm e}_{q_1}, \ldots, {\bm e}_{q_n}.\]
For every $1 \leq k \leq n+1$, define $\varepsilon_{1,\ldots, k}, \varepsilon_{1,\ldots, k}^\prime \in \mathcal{M}^{(1)}$ by $\pi_\emptyset (\varepsilon_{1,\ldots, k}) = {\bm e}_{1, \ldots, k} \coloneqq {\bm e}_{p_1} + \cdots + {\bm e}_{p_k} \in M_\emptyset^{(1)}$ and $\pi_\emptyset (\varepsilon_{1,\ldots, k}^\prime) = -({\bm e}_{p_1} + \cdots + {\bm e}_{p_k}) \in M_\emptyset^{(1)}$.
Under the coordinate transformation 
\[(x_{q_1}, \ldots, x_{q_n}, x_{p_1}, \ldots, x_{p_{n+1}}) \mapsto (x_{q_1}, \ldots, x_{q_n}, y_1, \ldots, y_n, x_{p_{n+1}}),\]
${\bm e}_{p_1} + \cdots + {\bm e}_{p_k} \in M_\emptyset^{(1)}$ for $1 \leq k \leq n$ corresponds to the unit vector for the coordinate $y_k$.
Hence $\sigma_{\epsilon_1, \ldots, \epsilon_n}$ is spanned by 
\begin{align*}
&\sigma_{\epsilon_1, \ldots, \epsilon_n} \cap (\{\varepsilon_{1,\ldots, k}, \varepsilon_{1,\ldots, k}^\prime \mid 1 \leq k \leq n+1\} \cup \{\varepsilon_{q_1}, \ldots, \varepsilon_{q_n}, \varepsilon_{q_1}^\prime, \ldots, \varepsilon_{q_n}^\prime\})\\
&= \{\varepsilon_{1,\ldots, k}^{(\epsilon_k)} \mid 1 \leq k \leq n\} \cup \{\varepsilon_{1,\ldots, n+1}, \varepsilon_{1,\ldots, n+1}^\prime, \varepsilon_{q_1}, \ldots, \varepsilon_{q_n}, \varepsilon_{q_1}^\prime, \ldots, \varepsilon_{q_n}^\prime\}
\end{align*}
as a real cone in every coordinate $(M_{\mathcal{C}}^{(1)})_\mathbb{R}$ of $(\mathcal{M}^{(1)})_\mathbb{R}$, where 
\[\varepsilon_{1,\ldots, k}^{(\epsilon_k)} \coloneqq 
\begin{cases} 
\varepsilon_{1,\ldots, k} &\text{if}\ \epsilon_k = 1,\\ 
\varepsilon_{1,\ldots, k}^\prime &\text{if}\ \epsilon_k = -1.
\end{cases}\]
Define $\mathbb{M}^{(1)}$ (respectively, $\mathbb{M}^{(1)}_\mathbb{R}$) to be the set of 
\[(a_1, b_1, a_{1,2}, b_{1,2}, \ldots, a_{1,\ldots, n}, b_{1,\ldots, n}, a, b, c_1, \ldots, c_n) \in \mathbb{Z}^{3n+2}\ (\text{respectively,} \in \mathbb{R}^{3n+2})\] 
satisfying the following equalities: 
\[a + b = 0,\quad a_{1,\ldots, k} + b_{1,\ldots, k} = \min \{0, c_k\}\quad \text{for}\ 1 \leq k \leq n.\]

\begin{prop}\label{p:relation_sp_points_M_1}
Each point $\varphi \in {\rm Sp}(\mathcal{M}^{(1)})$ satisfies $\varphi(\varepsilon_{1,\ldots, n+1}) + \varphi(\varepsilon_{1,\ldots, n+1}^\prime) = 0$ and 
\begin{equation}\label{eq:relation_sp_points_M_1}
\varphi(\varepsilon_{1,\ldots, k}) + \varphi(\varepsilon_{1,\ldots, k}^\prime) = \min \{0, \varphi(\varepsilon_{q_k}^\prime)\}
\end{equation}
for $1 \leq k \leq n$. 
In addition, there exist bijective maps ${\rm Sp}(\mathcal{M}^{(1)}) \rightarrow \mathbb{M}^{(1)}$ and ${\rm Sp}_\mathbb{R}(\mathcal{M}^{(1)}) \rightarrow \mathbb{M}_\mathbb{R}^{(1)}$ given by 
\begin{equation}\label{eq:identification_polyptych_M_1}
\varphi \mapsto (a_1, b_1, a_{1,2}, b_{1,2}, \ldots, a_{1,\ldots, n}, b_{1,\ldots, n}, \varphi(\varepsilon_{1,\ldots, n+1}), \varphi(\varepsilon_{1,\ldots, n+1}^\prime), \varphi(\varepsilon_{q_1}^\prime), \ldots, \varphi(\varepsilon_{q_n}^\prime)),
\end{equation}
where $a_{1,\ldots, k} \coloneqq \varphi(\varepsilon_{1,\ldots, k})$ and $b_{1,\ldots, k} \coloneqq \varphi(\varepsilon_{1,\ldots, k}^\prime)$ for $1 \leq k \leq n$.
\end{prop}

\begin{proof}
By the definition of ${\rm Sp}(\mathcal{M}^{(1)})$, each point $\varphi \in {\rm Sp}(\mathcal{M}^{(1)})$ satisfies 
\[\varphi(\varepsilon_{1,\ldots, k}) + \varphi(\varepsilon_{1,\ldots, k}^\prime) = \min \{\varphi(\varepsilon_{1,\ldots, k} +_{\mathcal{C}} \varepsilon_{1,\ldots, k}^\prime) \mid \mathcal{C} \in 2^{\widehat{\Pi}_1 \setminus \widehat{\Pi}_1^\ast}\}\]
for $1 \leq k \leq n+1$. 
Since $\varepsilon_{1,\ldots, n+1} +_{\mathcal{C}} \varepsilon_{1,\ldots, n+1}^\prime = 0$ for all $\mathcal{C} \in 2^{\widehat{\Pi}_1 \setminus \widehat{\Pi}_1^\ast}$, we have $\varphi(\varepsilon_{1,\ldots, n+1}) + \varphi(\varepsilon_{1,\ldots, n+1}^\prime) = 0$. 
Let $1 \leq k \leq n$, and $\mathcal{C} \in 2^{\widehat{\Pi}_1 \setminus \widehat{\Pi}_1^\ast}$.
If $q_k \notin \mathcal{C}$, then $\varepsilon_{1,\ldots, k} +_{\mathcal{C}} \varepsilon_{1,\ldots, k}^\prime = 0$.
If $q_k \notin \mathcal{C}$, then $\varepsilon_{1,\ldots, k} +_{\mathcal{C}} \varepsilon_{1,\ldots, k}^\prime = \varepsilon_{1,\ldots, k} +_{\{q_k\}} \varepsilon_{1,\ldots, k}^\prime = \varepsilon_{q_k}^\prime$. 
From these, we deduce \eqref{eq:relation_sp_points_M_1}. 
Hence we obtain a map ${\rm Sp}(\mathcal{M}^{(1)}) \rightarrow \mathbb{M}^{(1)}$ given by \eqref{eq:identification_polyptych_M_1}. 
To construct its inverse map, let us take 
${\bm a} = (a_1, b_1, \ldots, a_{1,\ldots, n}, b_{1,\ldots, n}, a, b, c_1, \ldots, c_n) \in \mathbb{M}^{(1)}$.
Then we define $\varphi_{\bm a} \in {\rm Sp}(\mathcal{M}^{(1)})$ as follows. 
Note that the set $\{{\bm e}_{1}, {\bm e}_{1, 2}, \ldots, {\bm e}_{1, \ldots, n+1}, {\bm e}_{q_1}, \ldots, {\bm e}_{q_n}\}$ forms a $\mathbb{Z}$-basis of $M_{\emptyset}^{(1)} = \mathbb{Z}^{\widehat{\Pi}_1 \setminus \widehat{\Pi}_1^\ast}$.
We take an arbitrary element ${\bm z} = z_1 {\bm e}_{1} + \cdots + z_{1, \ldots, n+1} {\bm e}_{1, \ldots, n+1} + z_{q_1} {\bm e}_{q_1} + \cdots + z_{q_n} {\bm e}_{q_n} \in M_{\emptyset}^{(1)}$ with $z_1, \ldots, z_{1, \ldots, n+1} , z_{q_1}, \ldots, z_{q_n} \in \mathbb{Z}$, and define $\varphi_{{\bm a},\emptyset} ({\bm z}) \in \mathbb{Z}$ by 
\[\varphi_{{\bm a},\emptyset} ({\bm z}) \coloneqq \varphi_{{\bm a},\emptyset}(z_1 {\bm e}_{1}) + \cdots + \varphi_{{\bm a},\emptyset}(z_{1, \ldots, n} {\bm e}_{1, \ldots, n}) + z_{1, \ldots, n+1} a - z_{q_1} c_1 - \cdots - z_{q_n} c_n,\]
where 
\[\varphi_{{\bm a},\emptyset}(z_{1, \ldots, k} {\bm e}_{1, \ldots, k}) \coloneqq 
\begin{cases} 
z_{1, \ldots, k} a_{1, \ldots, k} &\text{if}\ z_{1, \ldots, k} \geq 0,\\ 
-z_{1, \ldots, k} b_{1, \ldots, k} &\text{if}\ z_{1, \ldots, k} \leq 0
\end{cases}\]
for $1 \leq k \leq n$. 
Setting $\varphi_{\bm a} \coloneqq \varphi_{{\bm a},\emptyset} \circ \pi_\emptyset$, let us prove that $\varphi_{\bm a} \in {\rm Sp}(\mathcal{M}^{(1)})$. 
By definition, we have $\varphi_{\bm a} (k m) = k \varphi_{\bm a} (m)$ for all $m \in \mathcal{M}^{(1)}$ and $k \in \mathbb{Z}_{\geq 0}$.
Take $m, m^\prime \in \mathcal{M}^{(1)}$.
We write $\pi_\emptyset (m) = z_1 {\bm e}_{1} + \cdots + z_{1, \ldots, n+1} {\bm e}_{1, \ldots, n+1} + z_{q_1} {\bm e}_{q_1} + \cdots + z_{q_n} {\bm e}_{q_n}$ and $\pi_\emptyset (m^\prime) = z_1^\prime {\bm e}_{1} + \cdots + z_{1, \ldots, n+1}^\prime {\bm e}_{1, \ldots, n+1} + z_{q_1}^\prime {\bm e}_{q_1} + \cdots + z_{q_n}^\prime {\bm e}_{q_n}$. 
Since $\{{\bm e}_{1}, \ldots, {\bm e}_{1, \ldots, n+1}, {\bm e}_{q_1}, \ldots, {\bm e}_{q_n}\} \setminus \{{\bm e}_{1, \ldots, k}\}$ is included in the linearity space $V(\mathcal{M}^{(1)}, \mu_{\{q_k\}})$, we see by the definition of $\mu_{\mathcal{C}}$ that 
\[\pi_\emptyset (m +_\mathcal{C} m^\prime) = \sum_{1 \leq k \leq n+1} F_{k, \mathcal{C}}(z_{1, \ldots, k}, z_{1, \ldots, k}^\prime) + \sum_{1 \leq l \leq n} (z_{q_l}+ z_{q_l}^\prime) {\bm e}_{q_l},\]
where $F_{k, \mathcal{C}}(z_{1, \ldots, k}, z_{1, \ldots, k}^\prime) \coloneqq \pi_\emptyset(\pi_\emptyset^{-1} (z_{1, \ldots, k} {\bm e}_{1, \ldots, k}) +_{\mathcal{C}} \pi_\emptyset^{-1} (z_{1, \ldots, k}^\prime {\bm e}_{1, \ldots, k}))$.
Hence it follows that 
\[\varphi_{\bm a}(m +_\mathcal{C} m^\prime) = \left(\sum_{1 \leq k \leq n} \varphi_{{\bm a},\emptyset} (F_{k, \mathcal{C}}(z_{1, \ldots, k}, z_{1, \ldots, k}^\prime))\right) + (z_{1, \ldots, n+1} + z_{1, \ldots, n+1}^\prime) a - \sum_{1 \leq l \leq n} (z_{q_l}+ z_{q_l}^\prime) c_l.\]
In addition, since $F_{k, \mathcal{C}}(z_{1, \ldots, k}, z_{1, \ldots, k}^\prime)$ only depends on whether $q_k \in \mathcal{C}$ or not, we have
\begin{align*}
\min \{\varphi_{\bm a}(m +_\mathcal{C} m^\prime) \mid \mathcal{C} \in 2^{\widehat{\Pi}_1 \setminus \widehat{\Pi}_1^\ast}\} = &\left(\sum_{1 \leq k \leq n} \min\{\varphi_{{\bm a},\emptyset} (F_{k, \mathcal{C}}(z_{1, \ldots, k}, z_{1, \ldots, k}^\prime)) \mid \mathcal{C} \in 2^{\widehat{\Pi}_1 \setminus \widehat{\Pi}_1^\ast}\}\right) \\
&+ (z_{1, \ldots, n+1} + z_{1, \ldots, n+1}^\prime) a - \sum_{1 \leq l \leq n} (z_{q_l}+ z_{q_l}^\prime) c_l.
\end{align*}
Hence it suffices to show that 
\begin{equation}\label{eq:point_condition_for_varphi_empty}
\varphi_{{\bm a},\emptyset}(z_{1, \ldots, k} {\bm e}_{1, \ldots, k}) + \varphi_{{\bm a},\emptyset}(z_{1, \ldots, k}^\prime {\bm e}_{1, \ldots, k}) = \min\{\varphi_{{\bm a},\emptyset} (F_{k, \mathcal{C}}(z_{1, \ldots, k}, z_{1, \ldots, k}^\prime)) \mid \mathcal{C} \in 2^{\widehat{\Pi}_1 \setminus \widehat{\Pi}_1^\ast}\}
\end{equation}
for all $1 \leq k \leq n$.
Note that 
\[\{F_{k, \mathcal{C}}(z_{1, \ldots, k}, z_{1, \ldots, k}^\prime) \mid \mathcal{C} \in 2^{\widehat{\Pi}_1 \setminus \widehat{\Pi}_1^\ast}\} = \{F_{k, \emptyset}(z_{1, \ldots, k}, z_{1, \ldots, k}^\prime), F_{k, \{q_k\}}(z_{1, \ldots, k}, z_{1, \ldots, k}^\prime)\}\]
and that $F_{k, \emptyset}(z_{1, \ldots, k}, z_{1, \ldots, k}^\prime) = (z_{1, \ldots, k} + z_{1, \ldots, k}^\prime){\bm e}_{1, \ldots, k}$. 
Let us first consider the case $z_{1, \ldots, k} z_{1, \ldots, k}^\prime \geq 0$, that is, $z_{1, \ldots, k}, z_{1, \ldots, k}^\prime \geq 0$ or $z_{1, \ldots, k}, z_{1, \ldots, k}^\prime \leq 0$. 
Then we see by the definition of $\mu_{\{q_k\}}$ that 
\[F_{k, \{q_k\}}(z_{1, \ldots, k}, z_{1, \ldots, k}^\prime) = (z_{1, \ldots, k} + z_{1, \ldots, k}^\prime){\bm e}_{1, \ldots, k} = F_{k, \emptyset}(z_{1, \ldots, k}, z_{1, \ldots, k}^\prime)\]
and hence that 
\begin{align*}
\min\{\varphi_{{\bm a},\emptyset} (F_{k, \mathcal{C}}(z_{1, \ldots, k}, z_{1, \ldots, k}^\prime)) \mid \mathcal{C} \in 2^{\widehat{\Pi}_1 \setminus \widehat{\Pi}_1^\ast}\} &= \varphi_{{\bm a},\emptyset} ((z_{1, \ldots, k} + z_{1, \ldots, k}^\prime){\bm e}_{1, \ldots, k})\\
&= 
\begin{cases} 
(z_{1, \ldots, k} + z_{1, \ldots, k}^\prime) a_{1, \ldots, k} &\text{if}\ z_{1, \ldots, k}, z_{1, \ldots, k}^\prime \geq 0,\\ 
-(z_{1, \ldots, k} + z_{1, \ldots, k}^\prime) b_{1, \ldots, k} &\text{if}\ z_{1, \ldots, k}, z_{1, \ldots, k}^\prime \leq 0,
\end{cases}
\end{align*}
which coincides with $\varphi_{{\bm a},\emptyset}(z_{1, \ldots, k} {\bm e}_{1, \ldots, k}) + \varphi_{{\bm a},\emptyset}(z_{1, \ldots, k}^\prime {\bm e}_{1, \ldots, k})$.
Let us next consider the case $z_{1, \ldots, k} z_{1, \ldots, k}^\prime < 0$.
Without loss of generality, we may assume that $z_{1, \ldots, k} > 0$ and $z_{1, \ldots, k}^\prime < 0$.

If $z_{1, \ldots, k} + z_{1, \ldots, k}^\prime \geq 0$, then it follows by the definition of $\mu_{\{q_k\}}$ that 
\[F_{k, \{q_k\}}(z_{1, \ldots, k}, z_{1, \ldots, k}^\prime) = (z_{1, \ldots, k} + z_{1, \ldots, k}^\prime){\bm e}_{1, \ldots, k} + z_{1, \ldots, k}^\prime {\bm e}_{q_k},\]
which implies that 
\[\varphi_{{\bm a},\emptyset}(F_{k, \{q_k\}}(z_{1, \ldots, k}, z_{1, \ldots, k}^\prime)) = (z_{1, \ldots, k} + z_{1, \ldots, k}^\prime) a_{1, \ldots, k} - z_{1, \ldots, k}^\prime c_k.\]
Hence we deduce that
\begin{align*}
&\min\{\varphi_{{\bm a},\emptyset} (F_{k, \mathcal{C}}(z_{1, \ldots, k}, z_{1, \ldots, k}^\prime)) \mid \mathcal{C} \in 2^{\widehat{\Pi}_1 \setminus \widehat{\Pi}_1^\ast}\}\\
&= \min\{(z_{1, \ldots, k} + z_{1, \ldots, k}^\prime) a_{1, \ldots, k}, (z_{1, \ldots, k} + z_{1, \ldots, k}^\prime) a_{1, \ldots, k} - z_{1, \ldots, k}^\prime c_k\}\\
&= (z_{1, \ldots, k} + z_{1, \ldots, k}^\prime) a_{1, \ldots, k} - z_{1, \ldots, k}^\prime \min \{0, c_k\}\\
&= (z_{1, \ldots, k} + z_{1, \ldots, k}^\prime) a_{1, \ldots, k} - z_{1, \ldots, k}^\prime (a_{1, \ldots, k} + b_{1, \ldots, k})\\
&= z_{1, \ldots, k} a_{1, \ldots, k} - z_{1, \ldots, k}^\prime b_{1, \ldots, k} \\
&= \varphi_{{\bm a},\emptyset}(z_{1, \ldots, k} {\bm e}_{1, \ldots, k}) + \varphi_{{\bm a},\emptyset}(z_{1, \ldots, k}^\prime {\bm e}_{1, \ldots, k}).
\end{align*}

If $z_{1, \ldots, k} + z_{1, \ldots, k}^\prime \leq 0$, then it follows by the definition of $\mu_{\{q_k\}}$ that 
\[F_{k, \{q_k\}}(z_{1, \ldots, k}, z_{1, \ldots, k}^\prime) = (z_{1, \ldots, k} + z_{1, \ldots, k}^\prime){\bm e}_{1, \ldots, k} - z_{1, \ldots, k} {\bm e}_{q_k},\]
which implies that 
\[\varphi_{{\bm a},\emptyset}(F_{k, \{q_k\}}(z_{1, \ldots, k}, z_{1, \ldots, k}^\prime)) = -(z_{1, \ldots, k} + z_{1, \ldots, k}^\prime) b_{1, \ldots, k} + z_{1, \ldots, k} c_k.\]
Hence we deduce that
\begin{align*}
&\min\{\varphi_{{\bm a},\emptyset} (F_{k, \mathcal{C}}(z_{1, \ldots, k}, z_{1, \ldots, k}^\prime)) \mid \mathcal{C} \in 2^{\widehat{\Pi}_1 \setminus \widehat{\Pi}_1^\ast}\} \\
&= \min\{-(z_{1, \ldots, k} + z_{1, \ldots, k}^\prime) b_{1, \ldots, k}, -(z_{1, \ldots, k} + z_{1, \ldots, k}^\prime) b_{1, \ldots, k} + z_{1, \ldots, k} c_k\}\\
&= -(z_{1, \ldots, k} + z_{1, \ldots, k}^\prime) b_{1, \ldots, k} + z_{1, \ldots, k} \min \{0, c_k\}\\
&= -(z_{1, \ldots, k} + z_{1, \ldots, k}^\prime) b_{1, \ldots, k} + z_{1, \ldots, k} (a_{1, \ldots, k} + b_{1, \ldots, k})\\
&= z_{1, \ldots, k} a_{1, \ldots, k} - z_{1, \ldots, k}^\prime b_{1, \ldots, k} \\
&= \varphi_{{\bm a},\emptyset}(z_{1, \ldots, k} {\bm e}_{1, \ldots, k}) + \varphi_{{\bm a},\emptyset}(z_{1, \ldots, k}^\prime {\bm e}_{1, \ldots, k}).
\end{align*}
Thus, we conclude \eqref{eq:point_condition_for_varphi_empty}, which proves $\varphi_{{\bm a}} \in {\rm Sp}(\mathcal{M}^{(1)})$. 
Hence the map ${\rm Sp}(\mathcal{M}^{(1)}) \rightarrow \mathbb{M}^{(1)}$ given by \eqref{eq:identification_polyptych_M_1} is bijective since ${\bm a} \mapsto \varphi_{{\bm a}}$ is the inverse map. 
The assertion on the map ${\rm Sp}_\mathbb{R}(\mathcal{M}^{(1)}) \rightarrow \mathbb{M}_\mathbb{R}^{(1)}$ is similarly proved.
\end{proof}

We next consider a marked poset $(\Pi_2, \Pi_2^\ast, \lambda)$ whose marked Hasse diagram is given in Figure \ref{Hasse_marked}, where the circles denote the elements of $\Pi_2 \setminus \Pi_2^\ast$, and we write 
\[\Pi_2 \setminus \Pi_2^\ast = \{q_1, \ldots, q_n, p_1, \ldots, p_n\}.\]
Note that the square is only one element of $\Pi_2^\ast$, denoted by $p_{n+1}$, and the marking is given by $\lambda_{p_{n+1}} \in \mathbb{R}$.
Adding a new minimum element $\hat{0}$ and maximum element $\hat{1}$ to $\Pi_2$, we have $\widehat{\Pi}_2 \coloneqq \Pi_2 \sqcup \{\hat{0}, \hat{1}\}$.
The poset $\widehat{\Pi}_2$ is regarded as a marked poset with $\widehat{\Pi}_2^\ast = \{\hat{0}, \hat{1}, p_{n+1}\}$ and with some marking $\hat{\lambda}$ such that $\hat{\lambda}_{p_{n+1}} = \lambda_{p_{n+1}}$.
\begin{figure}[!ht]
\begin{center}
   \includegraphics[width=7.0cm,bb=50mm 210mm 160mm 230mm,clip]{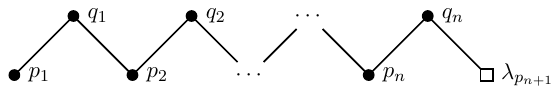}
	\caption{The marked Hasse diagram of $(\Pi_2, \Pi_2^\ast, \lambda)$.}
	\label{Hasse_marked}
\end{center}
\end{figure}

We denote by $\mathcal{M}^{(2)}$ the polyptych lattice corresponding to $(\widehat{\Pi}_2, \widehat{\Pi}_2^\ast, \hat{\lambda})$, and write $M_{\mathcal{C}}^{(2)} \coloneqq \pi_{\mathcal{C}} (\mathcal{M}^{(2)})$ for $\mathcal{C} \in 2^{\widehat{\Pi}_2 \setminus \widehat{\Pi}_2^\ast}$.
For $\epsilon_1, \ldots, \epsilon_n \in \{\pm 1\}$, define $\sigma_{\epsilon_1, \ldots, \epsilon_n} \subseteq (\mathcal{M}^{(2)})_\mathbb{R}$ so that $\pi_\emptyset (\sigma_{\epsilon_1, \ldots, \epsilon_n}) \subseteq (M_\emptyset^{(2)})_{\mathbb{R}}$ is given by $\epsilon_1 y_1, \ldots, \epsilon_n y_n \geq 0$ for $(x_q)_{q \in \widehat{\Pi}_2 \setminus \widehat{\Pi}_2^\ast} \in (M_\emptyset^{(2)})_{\mathbb{R}}$, where $y_k \coloneqq x_{p_k} - x_{p_{k+1}}$ for $1 \leq k \leq n-1$ and $y_n \coloneqq x_{p_n}$.
For each $\mathcal{C} \in 2^{\widehat{\Pi}_2 \setminus \widehat{\Pi}_2^\ast}$, we see by the definition of $\mu_{\mathcal{C}}$ that $\mu_{\mathcal{C}}$ is linear on the cone $\pi_\emptyset (\sigma_{\epsilon_1, \ldots, \epsilon_n})$.
More strongly, the set $\{\sigma_{\epsilon_1, \ldots, \epsilon_n} \mid \epsilon_1, \ldots, \epsilon_n \in \{\pm 1\}\}$ forms a complete list of maximal cones of the PL fan $\Sigma(\mathcal{M}^{(2)})$. 
In particular, the linearity space $V(\mathcal{M}^{(2)}, \emptyset)$ is spanned by ${\bm e}_{q_1}, \ldots, {\bm e}_{q_n}$.
For every $1 \leq k \leq n$, define $\varepsilon_{1,\ldots, k}, \varepsilon_{1,\ldots, k}^\prime \in \mathcal{M}^{(2)}$ by $\pi_\emptyset (\varepsilon_{1,\ldots, k}) = {\bm e}_{p_1} + \cdots + {\bm e}_{p_k} \in M_\emptyset^{(2)}$ and $\pi_\emptyset (\varepsilon_{1,\ldots, k}^\prime) = -({\bm e}_{p_1} + \cdots + {\bm e}_{p_k}) \in M_\emptyset^{(2)}$.
Under the coordinate transformation 
\[(x_{q_1}, \ldots, x_{q_n}, x_{p_1}, \ldots, x_{p_n}) \mapsto (x_{q_1}, \ldots, x_{q_n}, y_1, \ldots, y_n),\]
${\bm e}_{p_1} + \cdots + {\bm e}_{p_k} \in M_\emptyset^{(2)}$ corresponds to the unit vector for the coordinate $y_k$.
Hence $\sigma_{\epsilon_1, \ldots, \epsilon_n}$ is spanned by 
\begin{align*}
&\sigma_{\epsilon_1, \ldots, \epsilon_n} \cap (\{\varepsilon_{1,\ldots, k}, \varepsilon_{1,\ldots, k}^\prime \mid 1 \leq k \leq n\} \cup \{\varepsilon_{q_1}, \ldots, \varepsilon_{q_n}, \varepsilon_{q_1}^\prime, \ldots, \varepsilon_{q_n}^\prime\})\\
&= \{\varepsilon_{1,\ldots, k}^{(\epsilon_k)} \mid 1 \leq k \leq n\} \cup \{\varepsilon_{q_1}, \ldots, \varepsilon_{q_n}, \varepsilon_{q_1}^\prime, \ldots, \varepsilon_{q_n}^\prime\}
\end{align*}
as a real cone in every coordinate $(M_{\mathcal{C}}^{(2)})_\mathbb{R}$ of $(\mathcal{M}^{(2)})_\mathbb{R}$, where $\varepsilon_{1,\ldots, k}^{(\epsilon_k)}$ is defined as that for $\mathcal{M}^{(1)}$. 
Define $\mathbb{M}^{(2)}$ (respectively, $\mathbb{M}^{(2)}_\mathbb{R}$) to be the set of 
\[(a_1, b_1, a_{1,2}, b_{1,2}, \ldots, a_{1,\ldots, n}, b_{1,\ldots, n}, c_1, \ldots, c_n) \in \mathbb{Z}^{3n}\ (\text{respectively,}\ \mathbb{R}^{3n})\] 
satisfying the following equalities: 
\[a_{1,\ldots, k} + b_{1,\ldots, k} = \min \{0, c_k\}\quad \text{for}\ 1 \leq k \leq n.\]
In a way similar to the proof of Proposition \ref{p:relation_sp_points_M_1}, we obtain the following.

\begin{prop}\label{p:relation_sp_points_M_2}
Each point $\varphi \in {\rm Sp}(\mathcal{M}^{(2)})$ satisfies 
\begin{equation}\label{eq:relation_sp_points_M_2}
\varphi(\varepsilon_{1,\ldots, k}) + \varphi(\varepsilon_{1,\ldots, k}^\prime) = \min \{0, \varphi(\varepsilon_{q_k}^\prime)\}
\end{equation}
for $1 \leq k \leq n$. 
In addition, there exist bijective maps ${\rm Sp}(\mathcal{M}^{(2)}) \rightarrow \mathbb{M}^{(2)}$ and ${\rm Sp}_\mathbb{R}(\mathcal{M}^{(2)}) \rightarrow \mathbb{M}_\mathbb{R}^{(2)}$ given by 
\[\varphi \mapsto (a_1, b_1, a_{1,2}, b_{1,2}, \ldots, a_{1,\ldots, n}, b_{1,\ldots, n}, \varphi(\varepsilon_{q_1}^\prime), \ldots, \varphi(\varepsilon_{q_n}^\prime)),\]
where $a_{1,\ldots, k} \coloneqq \varphi(\varepsilon_{1,\ldots, k})$ and $b_{1,\ldots, k} \coloneqq \varphi(\varepsilon_{1,\ldots, k}^\prime)$ for $1 \leq k \leq n$.
\end{prop}

\section{Case of Gelfand--Tsetlin posets of type $C$}\label{s:type_C}

\subsection{Spaces of points and strict dual pairs}

Fix $n \in \mathbb{Z}_{>0}$ and take $(\lambda_1, \ldots, \lambda_n) \in \mathbb{Z}^n$ such that $0 \leq \lambda_1 \leq \cdots \leq \lambda_n$.
In this section, we restrict ourselves to the \emph{Gelfand--Tsetlin poset} $(\Pi_C, \Pi^\ast_C, \lambda)$ of type $C_n$ whose marked Hasse diagram is described in Figure \ref{type_C_Hasse}, where we write 
\[\Pi_C \setminus \Pi^\ast_C = \{q_{i, j} \mid 1 \leq j \leq n,\ 1 \leq i \leq 2n+1 -2j\},\]
and denote by the circles (respectively, the squares) the elements of $\Pi_C \setminus \Pi^\ast_C$ (respectively, $\Pi^\ast_C$).
The marking $\lambda = (\lambda_a)_{a \in \Pi^\ast_C}$ is given as $(\underbrace{0, \ldots, 0}_n, \lambda_1, \lambda_{2}, \ldots, \lambda_{n})$.
For $1 \leq i \leq n$, let $q_i^\ast$ denote the only one element of $\Pi^\ast_C$ with $r(q_i^\ast) = 2i$, that is, $\lambda_{q_i^\ast} = \lambda_i$. 
\begin{figure}[!ht]
\begin{center}
   \includegraphics[width=10.0cm,bb=40mm 140mm 170mm 230mm,clip]{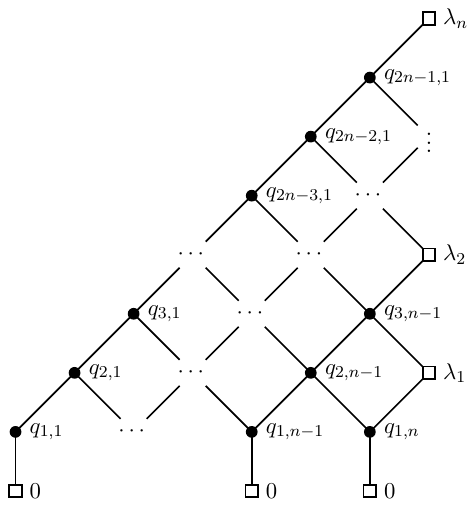}
	\caption{The marked Hasse diagram of the Gelfand--Tsetlin poset $(\Pi_C, \Pi^\ast_C, \lambda)$.}
	\label{type_C_Hasse}
\end{center}
\end{figure}

\begin{rem}
By definition, the marked order polytope $\mathcal{O}(\Pi_C, \Pi^\ast_C, \lambda)$ coincides with a Gelfand--Tsetlin polytope of type $C_n$ (see \cite[Section 6]{Lit} for the definition), and the marked chain polytope $\mathcal{C}(\Pi_C, \Pi^\ast_C, \lambda)$ coincides with an FFLV (Feigin--Fourier--Littelmann--Vinberg) polytope of type $C_n$ (see \cite[equation (2.2)]{FeFL2} for the definition). 
\end{rem}

For $i \in \mathbb{Z}_{>0}$, write $\Pi_C(i) \coloneqq \{p \in \Pi_C \mid r(p) = i\}$. 
More precisely, we have $\Pi_C(2i-1) = \{q_{2i-1, j} \mid 1 \leq j \leq n+1-i\}$ for $1 \leq i \leq n$ and $\Pi_C(2i) = \{q_{2i, j} \mid 1 \leq j \leq n-i+1\}$ for $1 \leq i \leq n$, where $q_{2i, n-i+1} \coloneqq q_i^\ast$.
Let $\mathcal{M}_C$ denote the polyptych lattice corresponding to $(\Pi_C, \Pi^\ast_C, \lambda)$, and write $M^{(C)}_{\mathcal{C}} \coloneqq \pi_{\mathcal{C}} (\mathcal{M}_C)$ for $\mathcal{C} \in 2^{\Pi_C \setminus \Pi^\ast_C}$.
We set 
\[\widehat{\Pi}_C^\circ \coloneqq \{q_{i,j} \in \Pi_C \setminus \Pi_C^\ast \mid i + 2j \leq 2n\}.\]
Note that for $q_{i,j} \in \Pi_C \setminus \Pi_C^\ast$, we have $q_{i,j} \in \widehat{\Pi}_C^\circ$ if and only if $q_{i+1,j} \in \Pi_C \setminus \Pi_C^\ast$.
For ${\bm \epsilon} = (\epsilon_{i,j} \mid q_{i,j} \in \widehat{\Pi}_C^\circ) \in \{\pm 1\}^{\widehat{\Pi}_C^\circ}$, define $\sigma_{\bm \epsilon} \subseteq (\mathcal{M}_C)_\mathbb{R}$ by 
\[\pi_\emptyset (\sigma_{\bm \epsilon}) = \{(x_{i,j})_{i,j} \in \mathbb{R}^{\Pi_C \setminus \Pi^\ast_C} \mid \epsilon_{i,j} (x_{i,j} - x_{i,j+1}) \geq 0\ \text{for all}\ q_{i,j} \in \widehat{\Pi}_C^\circ\},\]
where $x_{i,j}$ denotes the coordinate of $\mathbb{R}^{\Pi_C \setminus \Pi^\ast_C}$ corresponding to $q_{i,j} \in \Pi_C \setminus \Pi^\ast_C$. 
Then the set $\{\sigma_{\bm \epsilon} \mid {\bm \epsilon} \in \{\pm 1\}^{\widehat{\Pi}_C^\circ}\}$ forms a complete list of maximal cones of the PL fan $\Sigma(\mathcal{M}_C)$ of $\mathcal{M}_C$.
In particular, the linearity space $V(\mathcal{M}_C, \emptyset) \subseteq M_\emptyset^{(C)} = \mathbb{R}^{\Pi_C \setminus \Pi^\ast_C}$ is spanned by 
\[{\bm e}_{q_{2i-1,1}} + {\bm e}_{q_{2i-1,2}} + \cdots + {\bm e}_{q_{2i-1,n+1-i}}\]
for $1 \leq i \leq n$, where ${\bm e}_q \in \mathbb{R}^{\Pi_C \setminus \Pi^\ast_C}$ denotes the unit vector corresponding to $q \in \Pi_C \setminus \Pi^\ast_C$. 
For every $q_{i, j} \in \Pi_C \setminus \Pi^\ast_C$, define $\varepsilon_{i, \leq j}, \varepsilon_{i, \leq j}^\prime \in \mathcal{M}_C$ by $\pi_\emptyset (\varepsilon_{i, \leq j}) = {\bm e}_{q_{i,1}} + \cdots + {\bm e}_{q_{i, j}} \in M_\emptyset^{(C)}$ and $\pi_\emptyset (\varepsilon_{i, \leq j}^\prime) = -({\bm e}_{q_{i,1}} + \cdots + {\bm e}_{q_{i, j}}) \in M_\emptyset^{(C)}$.
In a way similar to that for $\mathcal{M}^{(1)}$ and $\mathcal{M}^{(2)}$ appearing in Section \ref{ss:basic_marked_posets}, we see that $\sigma_{\bm \epsilon}$ is spanned by 
\begin{align*}
&\sigma_{\bm \epsilon} \cap \{\varepsilon_{i, \leq j}, \varepsilon_{i, \leq j}^\prime \mid q_{i,j} \in \Pi_C \setminus \Pi^\ast_C\}\\
&= \{\varepsilon_{i, \leq j}^{(\epsilon_{i, j})} \mid q_{i,j} \in \widehat{\Pi}_C^\circ\} \cup \{\varepsilon_{2i-1, \leq n+1-i}, \varepsilon_{2i-1, \leq n+1-i}^\prime \mid 1 \leq i \leq n\}
\end{align*}
as a real cone in every coordinate $(M_{\mathcal{C}^{(C)}})_\mathbb{R} = \mathbb{R}^{\Pi_C \setminus \Pi^\ast_C}$ of $(\mathcal{M}_C)_\mathbb{R}$, where $\varepsilon_{i, \leq j}^{(\epsilon_{i, j})}$ is defined as $\varepsilon_{1,\ldots, k}^{(\epsilon_k)}$ for $\mathcal{M}^{(1)}$. 
Let $\mathbb{M}_C$ (respectively, $(\mathbb{M}_C)_\mathbb{R}$) denote the set of 
\[(y_{i,j}, y_{i,j}^\prime \mid q_{i, j} \in \Pi_C \setminus \Pi^\ast_C) \in (\mathbb{Z}^{\Pi_C \setminus \Pi^\ast_C})^2\ (\text{respectively,}\ \in (\mathbb{R}^{\Pi_C \setminus \Pi^\ast_C})^2)\] 
satisfying the following equalities: 
\[y_{i,j} - y_{i,j}^\prime = 
\begin{cases}
\min \{0, -y_{i+1,j}^\prime + y_{i+1,j-1}\} &\text{if}\ q_{i,j} \in \widehat{\Pi}_C^\circ,\\
0 &\text{otherwise},
\end{cases}\]
where $y_{i+1,0} \coloneqq 0$ when $j = 1$.
In a way similar to the proof of Proposition \ref{p:relation_sp_points_M_1} (see also Proposition \ref{p:relation_sp_points_M_2}), we deduce the following.

\begin{prop}\label{p:relation_sp_points_type_C_GT}
Each point $\varphi \in {\rm Sp}(\mathcal{M}_C)$ satisfies 
\[\varphi(\varepsilon_{i, \leq j}) + \varphi(\varepsilon_{i, \leq j}^\prime) = 
\begin{cases}
\min \{0, \varphi(\varepsilon_{i+1, \leq j}^\prime) + \varphi(\varepsilon_{i+1, \leq j-1})\} &\text{if}\ q_{i,j} \in \widehat{\Pi}_C^\circ,\\
0 &\text{otherwise}
\end{cases}\]
for $q_{i,j} \in \Pi_C \setminus \Pi^\ast_C$, where $\varepsilon_{i+1, \leq 0} \coloneqq 0 \in \mathcal{M}_C$ when $j = 1$. 
In addition, there exists a bijective map ${\rm Sp}(\mathcal{M}_C) \rightarrow \mathbb{M}_C$, $\varphi \mapsto (y_{i,j}, y_{i,j}^\prime \mid q_{i, j} \in \Pi_C \setminus \Pi^\ast_C)$, given by $y_{i,j} \coloneqq \varphi(\varepsilon_{i, \leq j})$ and $y_{i,j}^\prime \coloneqq -\varphi(\varepsilon_{i, \leq j}^\prime)$, which naturally extends to a bijective map ${\rm Sp}_\mathbb{R}(\mathcal{M}_C) \rightarrow (\mathbb{M}_C)_\mathbb{R}$.
\end{prop}

For $\mathcal{C} \in 2^{\Pi_C \setminus \Pi_C^\ast}$, define $\hat{\pi}_{\mathcal{C}} \colon \mathbb{M}_C \rightarrow \mathbb{Z}^{\Pi_C \setminus \Pi^\ast_C}$ by $\hat{\pi}_{\mathcal{C}} ({\bm y}) = (y_{i,j}^{(\mathcal{C})})_{i,j}$, where 
\[y_{i,j}^{(\mathcal{C})} \coloneqq \begin{cases}
y_{i,j} &\text{if}\ q_{i,j} \notin \mathcal{C},\\
y_{i,j}^\prime &\text{if}\ q_{i,j} \in \mathcal{C}.
\end{cases}\]
Then it follows by the definition of $\mathbb{M}_C$ that $\hat{\pi}_{\mathcal{C}}$ is bijective. 
Hence $\mathbb{M}_C$ can be identified with a polyptych lattice $\mathcal{N}_C = (\{N_{\mathcal{C}}^{(C)} \coloneqq \mathbb{Z}^{\Pi_C \setminus \Pi_C^\ast}\}_{\mathcal{C} \in 2^{\Pi_C \setminus \Pi_C^\ast}}, \{\hat{\mu}_{\mathcal{C}_1, \mathcal{C}_2}\}_{\mathcal{C}_1, \mathcal{C}_2 \in 2^{\Pi_C \setminus \Pi_C^\ast}})$ over $\mathbb{Z}$ defined so that $\pi_{\mathcal{C}} \colon \mathcal{N}_C \rightarrow \mathbb{Z}^{\Pi_C \setminus \Pi^\ast_C}$ is identical to $\hat{\pi}_{\mathcal{C}} \colon \mathbb{M}_C \rightarrow \mathbb{Z}^{\Pi_C \setminus \Pi^\ast_C}$.
In particular, the piecewise-linear transformation $\hat{\mu}_{\mathcal{C}_1, \mathcal{C}_2} = \pi_{\mathcal{C}_2} \circ \pi_{\mathcal{C}_1}^{-1}$ is given by changing the coordinate $y_{i,j}$ (respectively, $y_{i,j}^\prime$) by $y_{i,j}^\prime$ (respectively, $y_{i,j}$) if $q_{i,j} \in \mathcal{C}_2 \setminus \mathcal{C}_1$ (respectively, $q_{i,j} \in \mathcal{C}_1 \setminus \mathcal{C}_2$).
The identification $\mathbb{M}_C \simeq \mathcal{N}_C$ induces an identification $(\mathbb{M}_C)_\mathbb{R} \simeq (\mathcal{N}_C)_\mathbb{R}$. 
For ${\bm \epsilon} = (\epsilon_{i,j} \mid q_{i,j} \in \widehat{\Pi}_C^\circ) \in \{\pm 1\}^{\widehat{\Pi}_C^\circ}$, define $\sigma_{\bm \epsilon}^\prime \subseteq (\mathcal{N}_C)_\mathbb{R}$ by 
\[\sigma_{\bm \epsilon}^\prime = \{(y_{i,j}, y_{i,j}^\prime)_{i,j} \in (\mathbb{M}_C)_\mathbb{R} \mid \epsilon_{i,j} (y_{i+1,j}^\prime - y_{i+1,j-1}) \geq 0\ \text{for all}\ q_{i,j} \in \widehat{\Pi}_C^\circ\}.\]
Then the set $\{\sigma_{\bm \epsilon}^\prime \mid {\bm \epsilon} \in \{\pm 1\}^{\widehat{\Pi}_C^\circ}\}$ forms a complete list of maximal cones in the PL fan $\Sigma(\mathcal{N}_C)$ of $\mathcal{N}_C$.
In particular, the linearity space $V(\mathcal{N}_C, \emptyset) \subseteq \mathbb{R}^{\Pi_C \setminus \Pi^\ast_C}$ is spanned by ${\bm e}_{q_{1,j}}$ for $1 \leq j \leq n$.
For $q_{i,j} \in \Pi_C \setminus \Pi^\ast_C$, define $\varepsilon_{i,j} = (y_{k,l}, y_{k,l}^\prime \mid q_{k, l} \in \Pi_C \setminus \Pi^\ast_C) \in \mathbb{M}_C = \mathcal{N}_C$ by 
\[y_{k,l} = 
\begin{cases}
1 &\text{if}\ k = i,\ l \geq j,\\
-1 &\text{if}\ k = i-1,\ l \geq j,\\
0 &\text{otherwise},
\end{cases}\]
which implies that
\[y_{k,l}^\prime = 
\begin{cases}
1 &\text{if}\ k = i,\ l \geq j,\\
-1 &\text{if}\ k = i-1,\ l \geq j+1,\\
0 &\text{otherwise}.
\end{cases}\]
Similarly, we define $\varepsilon_{i,j}^\prime = (y_{k,l}, y_{k,l}^\prime \mid q_{k, l} \in \Pi_C \setminus \Pi^\ast_C) \in \mathbb{M}_C = \mathcal{N}_C$ by 
\[y_{k,l} = 
\begin{cases}
-1 &\text{if}\ k = i,\ l \geq j,\\
0 &\text{otherwise},
\end{cases}\]
which implies that $y_{k,l}^\prime = y_{k,l}$ for all $q_{k, l} \in \Pi_C \setminus \Pi^\ast_C$.
Then we see that $\sigma_{\bm \epsilon}^\prime$ is spanned by 
\begin{align*}
&\sigma_{\bm \epsilon}^\prime \cap \{\varepsilon_{i, j}, \varepsilon_{i, j}^\prime \mid q_{i,j} \in \Pi_C \setminus \Pi^\ast_C\}\\
&= \{\varepsilon_{i+1, j}^{(\epsilon_{i, j})} \mid q_{i,j} \in \widehat{\Pi}_C^\circ\} \cup \{\varepsilon_{1, j}, \varepsilon_{1, j}^\prime \mid 1 \leq j \leq n\}
\end{align*}
as a real cone in every coordinate $(N_{\mathcal{C}}^{(C)})_\mathbb{R} = \mathbb{R}^{\Pi_C \setminus \Pi^\ast_C}$ of $(\mathcal{N}_C)_\mathbb{R}$, where $\varepsilon_{i+1, j}^{(\epsilon_{i, j})}$ is defined as $\varepsilon_{1,\ldots, k}^{(\epsilon_k)}$ for $\mathcal{M}^{(1)}$. 

Define $\mathbb{M}_C^\vee$ (respectively, $(\mathbb{M}_C^\vee)_\mathbb{R}$) to be the set of 
\[(x_{i,j}, x_{i,j}^\prime \mid q_{i, j} \in \Pi_C \setminus \Pi^\ast_C) \in (\mathbb{Z}^{\Pi_C \setminus \Pi^\ast_C})^2\ (\text{respectively,}\ \in (\mathbb{R}^{\Pi_C \setminus \Pi^\ast_C})^2)\] 
satisfying the following equalities: 
\[x_{i,j}^\prime - x_{i,j} = \min \{-x_{i-1,j}, -x_{i-1,j+1}\}\]
for $q_{i,j} \in \Pi_C \setminus \Pi^\ast_C$, where we set $x_{k,l} \coloneqq 0$ unless $q_{k, l} \in \Pi_C \setminus \Pi^\ast_C$.
Then the polyptych lattice $\mathcal{M}_C$ is naturally identified with $\mathbb{M}_C^\vee$ so that $x_{i,j}$ (respectively, $x_{i,j}^\prime$) is the $q_{i, j}$-th coordinate of $M_\mathcal{C}^{(C)}$ if $q_{i, j} \notin \mathcal{C}$ (respectively, $q_{i, j} \in \mathcal{C}$). 
The identification $\mathbb{M}_C^\vee \simeq \mathcal{M}_C$ induces an identification $(\mathbb{M}_C^\vee)_\mathbb{R} \simeq (\mathcal{M}_C)_\mathbb{R}$. 
In a way similar to the proof of Proposition \ref{p:relation_sp_points_M_1}, we deduce the following.

\begin{prop}\label{p:dual_sp_points_type_C_GT}
Each point $\varphi \in {\rm Sp}(\mathcal{N}_C)$ satisfies 
\[\varphi(\varepsilon_{i, j}) + \varphi(\varepsilon_{i, j}^\prime) = \min \{\varphi(\varepsilon_{i-1, j}^\prime), \varphi(\varepsilon_{i-1, j+1}^\prime)\},\]
where $\varepsilon_{k, l}^\prime \coloneqq 0 \in \mathcal{N}_C$ unless $q_{k,l} \in \Pi_C \setminus \Pi^\ast_C$. 
In addition, there exists a bijective map ${\rm Sp}(\mathcal{N}_C) \rightarrow \mathbb{M}_C^\vee$, $\varphi \mapsto (x_{i,j}, x_{i,j}^\prime \mid q_{i, j} \in \Pi_C \setminus \Pi^\ast_C)$, given by $x_{i,j} \coloneqq -\varphi(\varepsilon_{i, j}^\prime)$ and $x_{i,j}^\prime \coloneqq \varphi(\varepsilon_{i, j})$, which naturally extends to a bijective map ${\rm Sp}_\mathbb{R}(\mathcal{N}_C) \rightarrow (\mathbb{M}_C^\vee)_\mathbb{R}$.
\end{prop}

\begin{ex}
Let $n = 2$. 
Then we have $\Pi_C \setminus \Pi^\ast_C = \{q_{1,1}, q_{1,2}, q_{2,1}, q_{3,1}\}$. 
In this case, $\mathbb{M}_C$ is the set of $(y_{1,1}, y_{1,2}, y_{2,1}, y_{3,1}, y_{1,1}^\prime, y_{1,2}^\prime, y_{2,1}^\prime, y_{3,1}^\prime) \in \mathbb{Z}^8$ satisfying the following equations:
\[y_{3,1}^\prime = y_{3,1},\ y_{1,2}^\prime = y_{1,2},\ y_{2,1}^\prime = y_{2,1} - \min \{0, -y_{3,1}^\prime\},\ y_{1,1}^\prime = y_{1,1} - \min \{0, -y_{2,1}^\prime\}.\]
In addition, $\mathbb{M}_C^\vee$ is the set of $(x_{1,1}, x_{1,2}, x_{2,1}, x_{3,1}, x_{1,1}^\prime, x_{1,2}^\prime, x_{2,1}^\prime, x_{3,1}^\prime) \in \mathbb{Z}^8$ satisfying the following equations:
\[x_{1,1}^\prime = x_{1,1},\ x_{1,2}^\prime = x_{1,2},\ x_{2,1}^\prime = x_{2,1} + \min \{-x_{1,1}, -x_{1,2}\},\ x_{3,1}^\prime = x_{3,1} + \min \{-x_{2,1}, 0\}.\]
\end{ex}

Let ${\mathsf v} \colon (\mathcal{M}_C)_\mathbb{R} \rightarrow {\rm Sp}_\mathbb{R}(\mathcal{N}_C)$ and ${\mathsf w} \colon (\mathcal{N}_C)_\mathbb{R} \rightarrow {\rm Sp}_\mathbb{R}(\mathcal{M}_C)$ denote the bijective maps given as the above identifications $(\mathcal{M}_C)_\mathbb{R} \simeq (\mathbb{M}_C^\vee)_\mathbb{R} \simeq {\rm Sp}_\mathbb{R}(\mathcal{N}_C)$ and $(\mathcal{N}_C)_\mathbb{R} \simeq (\mathbb{M}_C)_\mathbb{R} \simeq {\rm Sp}_\mathbb{R}(\mathcal{M}_C)$, respectively.
In particular, for $m = (x_{i,j}, x_{i,j}^\prime \mid q_{i, j} \in \Pi_C \setminus \Pi^\ast_C) \in (\mathcal{M}_C)_\mathbb{R}$, $n = (y_{i,j}, y_{i,j}^\prime \mid q_{i, j} \in \Pi_C \setminus \Pi^\ast_C) \in (\mathcal{N}_C)_\mathbb{R}$, and $q_{i,j} \in \Pi_C \setminus \Pi^\ast_C$, we have 
\begin{equation}\label{eq:dual_map}
\begin{aligned}
&({\mathsf v} (m))(\varepsilon_{i, j}) =  x_{i, j}^\prime,\quad ({\mathsf v} (m))(\varepsilon_{i, j}^\prime) = -x_{i, j},\\
&({\mathsf w} (n)) (\varepsilon_{i, \leq j}) =  y_{i, j},\quad ({\mathsf w} (n)) (\varepsilon_{i, \leq j}^\prime) = -y_{i, j}^\prime.
\end{aligned}
\end{equation}

\begin{thm}\label{t:strict_dual_pair}
The polyptych lattices $\mathcal{M}_C$ and $\mathcal{N}_C$ form a strict dual $\mathbb{Z}$-pair via the above bijective maps ${\mathsf v} \colon (\mathcal{M}_C)_\mathbb{R} \rightarrow {\rm Sp}_\mathbb{R}(\mathcal{N}_C)$ and ${\mathsf w} \colon (\mathcal{N}_C)_\mathbb{R} \rightarrow {\rm Sp}_\mathbb{R}(\mathcal{M}_C)$.
\end{thm}

\begin{proof}
The conditions (i) and (iii) in Definition \ref{d:strict_dual} are obvious from the above constructions of ${\mathsf v}$ and ${\mathsf w}$.

Let us prove condition (ii).
By \eqref{eq:dual_map}, we see that 
\begin{equation}\label{p:strict_dual_basis}
\begin{aligned}
&({\mathsf v} (\varepsilon_{k, \leq l}))(\varepsilon_{i, j}) = ({\mathsf w} (\varepsilon_{i, j}))(\varepsilon_{k, \leq l}) =
\begin{cases}
1 &\text{if}\ k = i,\ l \geq j,\\
-1 &\text{if}\ i = k+1,\ j \leq l,\\
0 &\text{otherwise},
\end{cases}\\
&({\mathsf v} (\varepsilon_{k, \leq l}))(\varepsilon_{i, j}^\prime) = ({\mathsf w} (\varepsilon_{i, j}^\prime))(\varepsilon_{k, \leq l}) =
\begin{cases}
-1 &\text{if}\ k = i,\ l \geq j,\\
0 &\text{otherwise},
\end{cases}\\
&({\mathsf v} (\varepsilon_{k, \leq l}^\prime))(\varepsilon_{i, j}) = ({\mathsf w} (\varepsilon_{i, j}))(\varepsilon_{k, \leq l}^\prime) =
\begin{cases}
-1 &\text{if}\ k = i,\ l \geq j,\\
1 &\text{if}\ i = k+1,\ j \leq l-1,\\
0 &\text{otherwise},
\end{cases}\\
&({\mathsf v} (\varepsilon_{k, \leq l}^\prime))(\varepsilon_{i, j}^\prime) = ({\mathsf w} (\varepsilon_{i, j}^\prime))(\varepsilon_{k, \leq l}^\prime) = 
\begin{cases}
1 &\text{if}\ k = i,\ l \geq j,\\
0 &\text{otherwise}.
\end{cases}
\end{aligned}
\end{equation}
For $m = (x_{i,j}, x_{i,j}^\prime \mid q_{i, j} \in \Pi_C \setminus \Pi^\ast_C) \in \mathcal{M}_C$ and $n = (y_{i,j}, y_{i,j}^\prime \mid q_{i, j} \in \Pi_C \setminus \Pi^\ast_C) \in \mathcal{N}_C$, let us take ${\bm \epsilon} = (\epsilon_{i,j} \mid q_{i,j} \in \widehat{\Pi}_C^\circ) \in \{\pm 1\}^{\widehat{\Pi}_C^\circ}$ and ${\bm \epsilon}^\prime = (\epsilon_{i,j}^\prime \mid q_{i,j} \in \widehat{\Pi}_C^\circ) \in \{\pm 1\}^{\widehat{\Pi}_C^\circ}$ such that $m \in \sigma_{\bm \epsilon}$ and $n \in \sigma_{{\bm \epsilon}^\prime}^\prime$. 
Then there exist $a_{i,j} \in \mathbb{Z}_{\geq 0}$ for $q_{i,j} \in \widehat{\Pi}_C^\circ$ and $c_1, \ldots, c_n \in \mathbb{Z}$ such that 
\[\pi_{\mathcal{C}} (m) = \sum_{q_{i,j} \in \widehat{\Pi}_C^\circ} a_{i,j} \pi_{\mathcal{C}} (\varepsilon_{i, \leq j}^{(\epsilon_{i, j})}) + \sum_{1 \leq t \leq n} c_t \pi_{\mathcal{C}} (\varepsilon_{2t-1, \leq n+1-t})\]
for all $\mathcal{C} \in 2^{\Pi_C \setminus \Pi^\ast_C}$.
Similarly, there exist $a_{i,j}^\prime \in \mathbb{Z}_{\geq 0}$ for $q_{i,j} \in \widehat{\Pi}_C^\circ$ and $c_1^\prime, \ldots, c_n^\prime \in \mathbb{Z}$ such that 
\[\pi_{\mathcal{C}} (n) = \sum_{q_{i,j} \in \widehat{\Pi}_C^\circ} a_{i,j}^\prime \pi_{\mathcal{C}} (\varepsilon_{i+1, j}^{(\epsilon_{i, j}^\prime)}) + \sum_{1 \leq t \leq n} c_t^\prime \pi_{\mathcal{C}} (\varepsilon_{1, t})\]
for all $\mathcal{C} \in 2^{\Pi_C \setminus \Pi^\ast_C}$.
Hence we have 
\begin{align*}
({\mathsf v} (m))(n) =& \sum_{(k,l),(i,j)} a_{k,l} a_{i,j}^\prime ({\mathsf v} (\varepsilon_{k, \leq l}^{(\epsilon_{k, l})}))(\varepsilon_{i+1, j}^{(\epsilon_{i, j}^\prime)}) + \sum_{s,(i,j)} c_s a_{i,j}^\prime ({\mathsf v} (\varepsilon_{2s-1, \leq n+1-s}))(\varepsilon_{i+1, j}^{(\epsilon_{i, j}^\prime)})\\ 
&+ \sum_{(k,l),t} a_{k,l} c_{t}^\prime ({\mathsf v} (\varepsilon_{k, \leq l}^{(\epsilon_{k, l})}))(\varepsilon_{1, t}) + \sum_{s,t} c_s c_t^\prime ({\mathsf v} (\varepsilon_{2s-1, \leq n+1-s}))(\varepsilon_{1, t})
\end{align*}
and
\begin{align*}
({\mathsf w} (n))(m) =& \sum_{(k,l),(i,j)} a_{k,l} a_{i,j}^\prime ({\mathsf w} (\varepsilon_{i+1, j}^{(\epsilon_{i, j}^\prime)}))(\varepsilon_{k, \leq l}^{(\epsilon_{k, l})}) + \sum_{s,(i,j)} c_s a_{i,j}^\prime ({\mathsf w} (\varepsilon_{i+1, j}^{(\epsilon_{i, j}^\prime)}))(\varepsilon_{2s-1, \leq n+1-s})\\ 
&+ \sum_{(k,l),t} a_{k,l} c_{t}^\prime ({\mathsf w} (\varepsilon_{1, t}))(\varepsilon_{k, \leq l}^{(\epsilon_{k, l})}) + \sum_{s,t} c_s c_t^\prime ({\mathsf w} (\varepsilon_{1, t}))(\varepsilon_{2s-1, \leq n+1-s}),
\end{align*}
which implies by \eqref{p:strict_dual_basis} that $({\mathsf v} (m))(n) = ({\mathsf w} (n))(m)$.
This proves condition (ii) in Definition \ref{d:strict_dual}. 

Fix $\mathcal{C}_0 \in 2^{\Pi_C \setminus \Pi^\ast_C}$.
To prove condition (iv) in Definition \ref{d:strict_dual}, let us first show that ${\mathsf w}^{-1}({\rm Sp}_\mathbb{R}(\mathcal{M}_C, \mathcal{C}_0))$ coincides with $\sigma_{{\bm \epsilon}(\mathcal{C}_0)}^\prime$, where ${\bm \epsilon}(\mathcal{C}_0) = (\epsilon_{i,j} \mid q_{i,j} \in \widehat{\Pi}_C^\circ)$ is given by 
\[\epsilon_{i,j} = \begin{cases}
1 &\text{if}\ q_{i+1,j} \in \mathcal{C}_0,\\
-1 &\text{otherwise}.
\end{cases}\]
By \cite[Proposition 3.10]{EHM}, it suffices to prove that for ${\bm y} = (y_{i,j}, y_{i,j}^\prime \mid q_{i, j} \in \Pi_C \setminus \Pi^\ast_C) \in (\mathcal{N}_C)_\mathbb{R}$, we have ${\bm y} \in \sigma_{{\bm \epsilon}(\mathcal{C}_0)}^\prime$ if and only if 
\begin{equation}\label{eq:duality_linearity}
{\mathsf w} ({\bm y}) (m +_{\mathcal{C}_0} m^\prime) = \min\{{\mathsf w} ({\bm y}) (m +_{\mathcal{C}} m^\prime) \mid \mathcal{C} \in 2^{\Pi_C \setminus \Pi^\ast_C}\}
\end{equation}
for all $m, m^\prime \in (\mathcal{M}_C)_\mathbb{R}$. 
For $q_{i,j} \in \Pi_C \setminus \Pi^\ast_C$, we see by \eqref{eq:dual_map} that 
\begin{align*}
\min\{{\mathsf w} ({\bm y}) (\varepsilon_{i, \leq j} +_{\mathcal{C}} \varepsilon_{i, \leq j}^\prime) \mid \mathcal{C} \in 2^{\Pi_C \setminus \Pi^\ast_C}\} &= {\mathsf w} ({\bm y}) (\varepsilon_{i, \leq j}) + {\mathsf w} ({\bm y}) (\varepsilon_{i, \leq j}^\prime)\\ 
&= y_{i,j} -y_{i,j}^\prime \\
&= \begin{cases}
\min \{0, -y_{i+1,j}^\prime + y_{i+1,j-1}\} &\text{if}\ q_{i,j} \in \widehat{\Pi}_C^\circ,\\
0 &\text{otherwise}.
\end{cases}
\end{align*}
Since
\[{\mathsf w} ({\bm y}) (\varepsilon_{i, \leq j} +_{\mathcal{C}_0} \varepsilon_{i, \leq j}^\prime) = 
\begin{cases}
{\mathsf w} ({\bm y}) (\varepsilon_{i+1, \leq j}^\prime) + {\mathsf w} ({\bm y}) (\varepsilon_{i+1, \leq j-1}) = -y_{i+1,j}^\prime + y_{i+1,j-1}&\text{if}\ q_{i+1,j} \in \mathcal{C}_0,\\
0&\text{otherwise}, 
\end{cases}\]
it follows that 
\[{\mathsf w} ({\bm y}) (\varepsilon_{i, \leq j} +_{\mathcal{C}_0} \varepsilon_{i, \leq j}^\prime) =   \min\{{\mathsf w} ({\bm y}) (\varepsilon_{i, \leq j} +_{\mathcal{C}} \varepsilon_{i, \leq j}^\prime) \mid \mathcal{C} \in 2^{\Pi_C \setminus \Pi^\ast_C}\}\]
for all $q_{i, j} \in \Pi_C \setminus \Pi^\ast_C$ if and only if $y_{i+1,j}^\prime \geq y_{i+1,j-1}$ for all $q_{i,j} \in \widehat{\Pi}_C^\circ$ with $q_{i+1,j} \in \mathcal{C}_0$ and $y_{i+1,j}^\prime \leq y_{i+1,j-1}$ for all $q_{i,j} \in \widehat{\Pi}_C^\circ$ with $q_{i+1,j} \notin \mathcal{C}_0$, which is equivalent to the condition that ${\bm y} \in \sigma_{{\bm \epsilon}(\mathcal{C}_0)}^\prime$.
Using the argument in the proof of Proposition \ref{p:relation_sp_points_M_1}, we deduce that this condition is equivalent to \eqref{eq:duality_linearity}, which proves the assertion. 

Similarly, for ${\bm x} = (x_{i,j}, x_{i,j}^\prime \mid q_{i, j} \in \Pi_C \setminus \Pi^\ast_C) \in (\mathcal{M}_C)_\mathbb{R}$ and $q_{i,j} \in \Pi_C \setminus \Pi^\ast_C$,
\begin{align*}
\min\{{\mathsf v} ({\bm x}) (\varepsilon_{i, j} +_{\mathcal{C}} \varepsilon_{i, j}^\prime) \mid \mathcal{C} \in 2^{\Pi_C \setminus \Pi^\ast_C}\} &= {\mathsf v} ({\bm x}) (\varepsilon_{i, j}) + {\mathsf v} ({\bm x}) (\varepsilon_{i, j}^\prime)\\ 
&= x_{i,j}^\prime -x_{i,j} \\
&= \min \{-x_{i-1,j}, -x_{i-1,j+1}\}.
\end{align*}
Since
\[{\mathsf v} ({\bm x}) (\varepsilon_{i, j} +_{\mathcal{C}_0} \varepsilon_{i, j}^\prime) = 
\begin{cases}
{\mathsf v} ({\bm x}) (\varepsilon_{i-1, j}^\prime) = -x_{i-1,j}&\text{if}\ q_{i,j} \in \mathcal{C}_0,\\
0&\text{otherwise}, 
\end{cases}\]
it follows that \[{\mathsf v} ({\bm x}) (\varepsilon_{i, j} +_{\mathcal{C}_0} \varepsilon_{i, j}^\prime) =   \min\{{\mathsf v} ({\bm x}) (\varepsilon_{i, j} +_{\mathcal{C}} \varepsilon_{i, j}^\prime) \mid \mathcal{C} \in 2^{\Pi_C \setminus \Pi^\ast_C}\}\]
for all $q_{i, j} \in \Pi_C \setminus \Pi^\ast_C$ if and only if ${\bm x} \in \sigma_{{\bm \epsilon}(\mathcal{C}_0)}$. 
As in the proof of the assertion on ${\mathsf w}^{-1}({\rm Sp}_\mathbb{R}(\mathcal{M}_C, \mathcal{C}_0))$, we conclude that ${\mathsf v}^{-1}({\rm Sp}_\mathbb{R}(\mathcal{N}_C, \mathcal{C}_0))$ coincides with $\sigma_{{\bm \epsilon}(\mathcal{C}_0)}$. 
This proves condition (iv) in Definition \ref{d:strict_dual}, which shows the theorem.
\end{proof}

By \cite[Lemma 4.2]{EHM}, the polyptych lattices $\mathcal{M}_C$ and $\mathcal{N}_C$ are both full in the sense of \cite[Definition 3.11]{EHM}.

\subsection{Detropicalizations and toric degenerations}

Let $\Bbbk$ be an algebraically closed field of characteristic $0$, and
consider a polynomial ring $\Bbbk [X_q, Y_q \mid q \in \Pi_C \setminus \Pi^\ast_C]$.
For $q_{i, j} \in \Pi_C \setminus \Pi^\ast_C$, we write $X_{i,j} \coloneqq X_{q_{i,j}}$ and $Y_{i,j} \coloneqq Y_{q_{i,j}}$.
Define $I = I_C = (g_q \mid q \in \Pi_C \setminus \Pi^\ast_C)$, $\mathcal{A} = \mathcal{A}_C$, and ${\bf B} = {\bf B}_C$ as in Introduction. 
More precisely, $g_{i,j} \coloneqq g_{q_{i,j}}$ for $q_{i,j} \in \Pi_C \setminus \Pi^\ast_C$ is given by 
\[g_{i,j} = X_{i,j} Y_{i,j} - 1 - X_{i+1,j-1} Y_{i+1, j} \in \Bbbk [X_q, Y_q \mid q \in \Pi_C \setminus \Pi^\ast_C],\]
where $X_{i+1,j-1} \coloneqq 1$ if $j = 1$ and $Y_{i+1, j} \coloneqq 0$ unless $q_{i+1,j} \in \Pi_C \setminus \Pi^\ast_C$; cf.\ \cref{p:relation_sp_points_type_C_GT}.
In addition, we have $\mathcal{A}_C = \Bbbk [X_q, Y_q \mid q \in \Pi_C \setminus \Pi^\ast_C]/I_C$ and
\begin{equation}\label{eq:standard_monomial_type_C}
\begin{aligned}
{\bf B}_C = \left\{\prod_{q \in \Pi_C \setminus \Pi^\ast_C} X_q^{a_q} Y_q^{b_q} \mid \min\{a_q, b_q\} = 0\ \text{for all}\ q \in \Pi_C \setminus \Pi^\ast_C\right\}.
\end{aligned}
\end{equation}
Hence we obtain a bijective map 
\[{\bf B}_C \rightarrow \mathbb{Z}^{\Pi_C \setminus \Pi^\ast_C},\quad \prod_{q \in \Pi_C \setminus \Pi^\ast_C} X_{q}^{a_q} Y_{q}^{b_q} \mapsto (a_q -b_q)_{q \in \Pi_C \setminus \Pi^\ast_C}.\]
For $f \in \Bbbk [X_q, Y_q \mid q \in \Pi_C \setminus \Pi^\ast_C]$, set $\overline{f} \coloneqq f \bmod I_C \in \mathcal{A}_C$. 
Fix a total order on the set $\{X_{i,j}, Y_{i,j} \mid q_{i, j} \in \Pi_C \setminus \Pi^\ast_C\}$ of variables such that $X_{i,j} > X_{i+1,j-1}$ and $Y_{i,j} > Y_{i+1, j}$ for $q_{i,j} \in \widehat{\Pi}_C^\circ$, and define a monomial order $<$ on $\Bbbk [X_{i,j}, Y_{i,j} \mid q_{i, j} \in \Pi_C \setminus \Pi^\ast_C]$ to be the corresponding lexicographic order.
Then $\{g_{i,j} \mid q_{i,j} \in \Pi_C \setminus \Pi^\ast_C\}$ forms a Gr\"{o}bner basis of $I_C$ with respect to the monomial order $<$ since the leading monomials of these polynomials $g_{i,j}$ for $q_{i, j} \in \Pi_C \setminus \Pi^\ast_C$ are relatively prime.
In addition, ${\bf B}_C$ is the set of standard monomials with respect to this Gr\"{o}bner basis. 

\begin{ex}
Let $n = 2$. 
Then we have 
\[\mathcal{A}_C = \Bbbk [X_q, Y_q \mid q = q_{1,1}, q_{1,2}, q_{2,1}, q_{3,1}]/I_C,\]
where 
\[I_C = (X_{3,1} Y_{3,1} - 1,\ X_{2,1} Y_{2,1} - 1 - Y_{3,1},\ X_{1,2} Y_{1,2} - 1,\ X_{1,1} Y_{1,1} - 1 - Y_{2,1}).\]
The total order $<$ satisfies $Y_{1,1} > Y_{2,1} > Y_{3,1}$, which implies that $X_{2,1} Y_{2,1}$ and $X_{1,1} Y_{1,1}$ are the initial terms of $X_{2,1} Y_{2,1} - 1 - Y_{3,1}$ and $X_{1,1} Y_{1,1} - 1 - Y_{2,1}$, respectively.
\end{ex}

We arrange the elements of $\Pi_C \setminus \Pi^\ast_C$ as 
\begin{equation}\label{eq:arrangement_type_C_nonmarked}
\begin{aligned}
(q_1, \ldots, q_{n^2}) \coloneqq (q_{1,1}, q_{1,2}, \ldots, q_{1,n}, q_{2,1}, \ldots, q_{2n-2,1}, q_{2n-1,1}),
\end{aligned}
\end{equation}
and write $X_k \coloneqq X_{i,j}$, $Y_k \coloneqq Y_{i,j}$, and $g_k \coloneqq g_{i,j}$ when $q_k = q_{i,j}$.
Note that for $1 \leq k < \ell \leq n^2$, the polynomial $g_{\ell}$ does not contain the variables $X_k, Y_k$.
We set 
\[\mathcal{A}_C^{(k)} \coloneqq \Bbbk [X_k, \ldots, X_{n^2}, Y_k, \ldots, Y_{n^2}]/(g_k, \ldots, g_{n^2})\]
for $1 \leq k \leq n^2$. 
For $1 < k \leq n^2$, the inclusion map 
\[\Bbbk [X_k, \ldots, X_{n^2}, Y_k, \ldots, Y_{n^2}] \hookrightarrow \Bbbk [X_{k-1}, \ldots, X_{n^2}, Y_{k-1}, \ldots, Y_{n^2}]\] 
induces an injective $\Bbbk$-algebra homomorphism $\mathcal{A}_C^{(k)} \hookrightarrow \mathcal{A}_C^{(k-1)}$.
In particular, we can regard $\mathcal{A}_C^{(2)}, \ldots, \mathcal{A}_C^{(n^2)}$ as $\Bbbk$-subalgebras of $\mathcal{A}_C^{(1)} = \mathcal{A}_C$.

\begin{prop}\label{p:integral_domain}
The $\Bbbk$-algebra $\mathcal{A}_C$ is an integral domain. 
\end{prop}

\begin{proof}
Assume for a contradiction that there exist $F_1, F_1^\prime \in \mathcal{A}_C \setminus \{0\}$ such that $F_1 F_1^\prime = 0$. 
Since ${\bf B}_C$ is a $\Bbbk$-basis of $\mathcal{A}_C$, the elements $F_1$ and $F_1^\prime$ can be written as 
\begin{equation}\label{eq:polynomial_one_variable}
F_1 =\sum_{k \in \mathbb{Z}} \overline{f_k Z_{1,k}},\quad F_1^\prime = \sum_{k \in \mathbb{Z}} \overline{f_k^\prime Z_{1,k}}
\end{equation}
for some $f_k, f_k^\prime \in \Bbbk [X_\ell, Y_\ell \mid 2 \leq \ell \leq n^2]$, where 
$f_k = f_k^\prime = 0$ for all but finitely many $k$, and 
\[Z_{1,k} \coloneqq \begin{cases}
X_1^k&\text{if}\ k \geq 0,\\
Y_1^{-k} &\text{if}\ k \leq 0.
\end{cases}\]
Let us regard $\Bbbk [X_{i,j}, Y_{i,j} \mid q_{i, j} \in \Pi_C \setminus \Pi^\ast_C]$ as a $\mathbb{Z}$-graded ring by 
\[\deg (X_1) = 1,\quad \deg (Y_1) = -1,\quad \deg (X_k) = \deg (Y_k) = 0\ (k > 1).\]
Since there are no $X_1, Y_1$ in $g_2, \ldots, g_{n^2}$, the ideal $I_C$ is a homogeneous ideal, and hence $\mathcal{A}_C$ is also a $\mathbb{Z}$-graded ring. 
Then we have 
\[F_1 F_1^\prime = \overline{f_{k_1} Z_{1, k_1} f_{k_2}^\prime Z_{1, k_2}} + (\text{lower terms})\]
with respect to this grading, where $k_1 \coloneqq \max \{k \in \mathbb{Z} \mid \overline{f_k Z_{1,k}} \neq 0\}$ and $k_2 \coloneqq \max \{k \in \mathbb{Z} \mid \overline{f_k^\prime Z_{1,k}} \neq 0\}$.
Since $F_1 F_1^\prime = 0$, it holds that $\overline{f_{k_1} Z_{1, k_1} f_{k_2}^\prime Z_{1, k_2}} = 0$.
Note that
\[\overline{Z_{1, k_1} Z_{1, k_2}} = \begin{cases}
\overline{Z_{1, k_1 + k_2}}&\text{if}\ k_1 k_2 \geq 0,\\
\overline{(1 + Y_{2,1})^{k_3} Z_{1, k_1 + k_2}} &\text{if}\ k_1 k_2 < 0,
\end{cases}\]
where $k_3 \coloneqq \min\{|k_1|, |k_2|\}$.
Since $\{\overline{Z_{1, k}} \mid k \in \mathbb{Z}\} \subseteq \mathcal{A}_C$ is linearly independent over $\mathcal{A}_C^{(2)}$, we see by $\overline{f_{k_1} Z_{1, k_1} f_{k_2} Z_{1, k_2}} = 0$ that there exist $F_2, F_2^\prime \in \mathcal{A}_C^{(2)} \setminus \{0\}$ such that $F_2 F_2^\prime = 0$.
Continuing this argument, we conclude a contradiction.  
This proves the proposition.
\end{proof}

\begin{prop}\label{p:unit_elements}
For each unit element $F \in \mathcal{A}_C^\times$, there exists $c \in \Bbbk^\times$ such that $c F$ is a Laurent monomial in variables $\overline{X_{i,j}}$ for $q_{i, j} \in \Pi_C \setminus (\Pi^\ast_C \cup \widehat{\Pi}_C^\circ)$. 
In particular, $\mathcal{A}_C^\times$ is isomorphic to $\Bbbk^\times \times \mathbb{Z}^n$.
\end{prop}

\begin{proof}
Let us use the argument in the proof of Proposition \ref{p:integral_domain}.
We set $F_1 \coloneqq F$, $F_1^\prime \coloneqq F^{-1}$, and write these elements as in \eqref{eq:polynomial_one_variable}.
Since $\mathcal{A}_C^{(2)}$ is an integral domain by Proposition \ref{p:integral_domain}, the equality $F_1 F_1^\prime = 1$ implies that $F_1 = \overline{f_k Z_{1, k}}$ and $F_1^\prime = \overline{f_{-k}^\prime Z_{1, -k}}$ for some $k \in \mathbb{Z}$. 
Then we see that 
\begin{align*}
1 = F_1 F_1^\prime = \overline{f_k Z_{1, k} f_{-k}^\prime Z_{1, -k}} = \overline{(1 + Y_{2,1})^{|k|} f_k f_{-k}^\prime}.
\end{align*}
Set $F_2 \coloneqq \overline{(1 + Y_{2,1})^{|k|} f_k}$, $F_2^\prime \coloneqq \overline{f_{-k}^\prime} \in \Bbbk [X_{i,j}, Y_{i,j} \mid q_{i, j} \in \Pi_C \setminus (\Pi^\ast_C \cup \{q_{1,1}\})]$, and continue this argument. 
Then we conclude that $k = 0$ if $q_{2,1} \in \Pi_C \setminus \Pi^\ast_C$ and that for some $c \in \Bbbk^\times$, $c \overline{f_k}$ is a Laurent monomial in variables $\overline{X_{i,j}}$ for $q_{i, j} \in \Pi_C \setminus (\Pi^\ast_C \cup \widehat{\Pi}_C^\circ \cup \{q_{1,1}\})$. 
This proves the proposition.
\end{proof}

Localizing $\mathcal{A}_C$ at the elements $\overline{X_{i,j}}$ for $q_{i, j} \in \widehat{\Pi}_C^\circ$, we obtain 
\begin{equation}\label{eq:detropicalization_localized}
\mathcal{A}_C[\overline{X_{i,j}}^{-1} \mid q_{i, j} \in \widehat{\Pi}_C^\circ] \simeq \Bbbk [X_{i,j}^{\pm 1} \mid q_{i, j} \in \Pi_C \setminus \Pi^\ast_C],
\end{equation}
which is a Laurent polynomial ring.
Hence the Krull dimension of $\mathcal{A}_C$ coincides with the cardinality of $\Pi_C \setminus \Pi^\ast_C$, i.e.\ $n^2$, which equals the rank of $\mathcal{M}_C$. 

\begin{prop}\label{p:nonsingular}
The affine scheme ${\rm Spec}(\mathcal{A}_C)$ is nonsingular and hence normal.
\end{prop}

\begin{proof}
Let $N \coloneqq |\Pi_C \setminus \Pi^\ast_C| = n^2$, and take an arbitrary closed point ${\bm a}$ of ${\rm Spec}(\mathcal{A}_C)$, that is, ${\bm a}$ is an element of $\Bbbk^{2N}$ satisfying $g_{i,j}({\bm a}) = 0$ for all $q_{i,j} \in \Pi_C \setminus \Pi^\ast_C$. 
We write ${\bm a} = (a_{i,j}, b_{i,j})_{i,j} \in \Bbbk^{2N}$. 
Since the Krull dimension of $\mathcal{A}_C$ coincides with $N$, it suffices to show that the rank of the $N \times 2N$-matrix $D = \left(\frac{\partial g_{i,j}}{\partial X_{k,l}} ({\bm a}), \frac{\partial g_{i,j}}{\partial Y_{k,l}} ({\bm a})\right)_{(i,j), (k,l)}$ is equal to $N$. 
We arrange the elements of $\Pi_C \setminus \Pi^\ast_C$ as \eqref{eq:arrangement_type_C_nonmarked}.
Note that 
\[\frac{\partial g_{i,j}}{\partial X_{k,l}} ({\bm a}) =
\begin{cases}
b_{i,j} &\text{if}\ (k,l) = (i,j),\\
-b_{i+1,j} &\text{if}\ q_{i,j} \in \widehat{\Pi}_C^\circ\ \text{and}\ (k,l) = (i+1,j-1),\\
0 &\text{otherwise}
\end{cases}\]
and that 
\[\frac{\partial g_{i,j}}{\partial Y_{k,l}} ({\bm a}) =
\begin{cases}
a_{i,j} &\text{if}\ (k,l) = (i,j),\\
-a_{i+1,j-1} &\text{if}\ q_{i,j} \in \widehat{\Pi}_C^\circ,\ j > 1,\ \text{and}\ (k,l) = (i+1,j),\\
0 &\text{otherwise}.
\end{cases}\]
Hence if $a_{i,j} \neq 0$ or $b_{i,j} \neq 0$ for all $(i,j)$, then the rank of $D$ equals $N$ since the $N \times N$-submatrix $\widehat{D}$ of $D$ whose $(i,j)$-th column corresponds to $\frac{\partial}{\partial X_{i,j}}$ (respectively, $\frac{\partial}{\partial Y_{i,j}}$) if $a_{i,j} \neq 0$ (respectively, if $a_{i,j} = 0$, $b_{i,j} \neq 0$) is an upper triangular matrix whose diagonal entries are nonzero. 
Assume that $a_{i,j} = 0$ and $b_{i,j} = 0$ for some $(i,j)$, and write the set of all such $(i,j)$ as
$\{(i_1, j_1), \ldots, (i_t, j_t)\}$. 
Take $1 \leq k \leq t$. 
Since $g_{i_k, j_k}({\bm a}) = 0$, we see by $a_{i_k,j_k} = 0$ that $q_{i_k,j_k} \in \widehat{\Pi}_C^\circ$ and that $a_{i_k+1,j_k-1} b_{i_k+1, j_k} = 1$, where $a_{i_k+1,0} \coloneqq 1$ if $j_k = 1$.
In particular, $a_{i_k+1,j_k-1}, b_{i_k+1, j_k} \neq 0$, which implies that $(i_k+1,j_k-1), (i_k+1, j_k) \notin \{(i_1, j_1), \ldots, (i_t, j_t)\}$. 
If $j_k > 1$, then let $D_k$ denote the submatrix of $D$ with rows corresponding to $g_{i_k,j_k}, g_{i_k+1,j_k-1}, g_{i_k+1, j_k}$ and with columns corresponding to $\frac{\partial}{\partial X_{i_k+1,j_k-1}}, \frac{\partial}{\partial Y_{i_k+1,j_k-1}}, \frac{\partial}{\partial X_{i_k+1,j_k}}, \frac{\partial}{\partial Y_{i_k+1,j_k}}$. 
Then we have 
\[D_k = \begin{pmatrix}
-b_{i_k+1,j_k} & 0 & 0 & -a_{i_k+1,j_k-1}\\
b_{i_k+1,j_k-1} & a_{i_k+1,j_k-1} & 0 & 0 \\ 
0 & 0 &b_{i_k+1,j_k} & a_{i_k+1,j_k}
\end{pmatrix},\]
which is of rank $3$ since $a_{i_k+1,j_k-1}, b_{i_k+1, j_k} \neq 0$.
If $j_k = 1$, then let $D_k$ denote the submatrix of $D$ with rows corresponding to $g_{i_k,j_k}, g_{i_k+1, j_k}$ and with columns corresponding to $\frac{\partial}{\partial X_{i_k+1,j_k}}, \frac{\partial}{\partial Y_{i_k+1,j_k}}$. 
Then we have 
\[D_k = \begin{pmatrix}
0 & -1\\
b_{i_k+1,j_k} & a_{i_k+1,j_k}
\end{pmatrix},\]
which is of rank $2$ since $b_{i_k+1, j_k} \neq 0$.
Thus, we conclude that $D$ is of rank $N$.
\end{proof}

\begin{prop}\label{p:UFD}
The $\Bbbk$-algebra $\mathcal{A}_C$ is a unique factorization domain (UFD). 
\end{prop}

\begin{proof}
The localization \eqref{eq:detropicalization_localized} is a Laurent polynomial ring and hence a UFD. 
Thus, by \cite[Theorem 20.2]{Mat}, it suffices to prove that $\overline{X_{k,l}}$ is a prime element of $\mathcal{A}_C$ for $q_{k,l} \in \widehat{\Pi}_C^\circ$.
To do that, let us show that $\mathcal{A}_C/(\overline{X_{k, l}})$ is an integral domain.
For $f \in \Bbbk [X_{i,j}, Y_{i,j} \mid q_{i, j} \in \Pi_C \setminus \Pi^\ast_C]$, we denote by $\rho(f) \in \mathcal{A}_C/(\overline{X_{k,l}})$ the element corresponding to $\overline{f} \in \mathcal{A}_C$. 
Since $\rho(g_{k,l}) = 0$ and $\rho(X_{k, l}) = 0$, it follows that $- \rho(X_{k+1,l-1}) = \rho(Y_{k+1, l})^{-1}$. 
Since $- \rho(X_{k+1,l-1} g_{k+1,l}) = \rho(Y_{k+1, l})^{-1} \rho(g_{k+1,l}) = 0$, we deduce that 
\[\rho(X_{k+1,l}) = - \rho(X_{k+1,l-1}) - \rho(X_{k+1,l-1} X_{k+2,l-1} Y_{k+2, l}).\]
Similarly, if $q_{k+1, l-1} \in \Pi_C \setminus \Pi^\ast_C$, then it follows by $- \rho(Y_{k+1,l} g_{k+1,l-1}) = 0$ that 
\[\rho(Y_{k+1,l-1}) = -\rho(Y_{k+1,l}) - \rho (Y_{k+1,l} X_{k+2,l-2} Y_{k+2, l-1}).\] 
In particular, if $q_{k, l+1} \in \Pi_C \setminus \Pi^\ast_C$, then 
\[0 = \rho(g_{k,l+1}) = \rho(X_{k,l+1} Y_{k,l+1} - 1 + (X_{k+1,l-1} + X_{k+1,l-1} X_{k+2,l-1} Y_{k+2, l}) Y_{k+1, l+1}).\]
In addition, if $q_{k, l-1} \in \Pi_C \setminus \Pi^\ast_C$, then 
\[0 = \rho(g_{k,l-1}) = \rho(X_{k,l-1} Y_{k,l-1} - 1 + X_{k+1,l-2} (Y_{k+1,l} + Y_{k+1,l} X_{k+2,l-2} Y_{k+2, l-1})).\]
We set $\mathcal{I}_{k,l} \coloneqq \{q_{i, j} \in \Pi_C \setminus \Pi^\ast_C \mid (i,j) \neq (k,l), (k+1,l)\}$, and define a $\Bbbk$-algebra $\mathcal{A}_C^{(k,l)}$ by 
\[\mathcal{A}_C^{(k,l)} \coloneqq \Bbbk [X_{i,j}, Y_{i,j} \mid q_{i, j} \in \mathcal{I}_{k,l}][Y_{k,l}]/(\tilde{g}_{i,j}^{(k,l)} \mid q_{i, j} \in \mathcal{I}_{k,l}),\]
where 
\[\tilde{g}_{i,j}^{(k,l)} \coloneqq \begin{cases}
X_{k+1,l-1} Y_{k+1,l-1} - 1&\text{if}\ (i,j) = (k+1,l-1),\\
X_{k,l+1} Y_{k,l+1} - 1 + (X_{k+1,l-1} + X_{k+1,l-1} X_{k+2,l-1} Y_{k+2, l}) Y_{k+1, l+1}&\text{if}\ (i,j) = (k,l+1),\\
X_{k,l-1} Y_{k,l-1} - 1 - X_{k+1,l-2} (Y_{k+1,l-1} + Y_{k+1,l-1} X_{k+2,l-2} Y_{k+2, l-1})&\text{if}\ (i,j) = (k,l-1),\\
X_{k-1,l+1} Y_{k-1,l+1} - 1 &\text{if}\ (i,j) = (k-1,l+1),\\
g_{i,j} = X_{i,j} Y_{i,j} - 1 - X_{i+1,j-1} Y_{i+1, j} &\text{otherwise}
\end{cases}\]
for $q_{i, j} \in \mathcal{I}_{k,l}$. 
By the arguments above, we obtain a $\Bbbk$-algebra isomorphism
\[\varphi \colon \mathcal{A}_C/(\overline{X_{k,l}}) \xrightarrow{\sim} \mathcal{A}_C^{(k,l)}\]
by
\[\varphi (\rho(X_{i,j})) \coloneqq 
\begin{cases}
\rho^{(k,l)} (-X_{k+1,l-1} - X_{k+1,l-1} X_{k+2,l-1} Y_{k+2, l})&\text{if}\ (i,j) = (k+1,l),\\
\rho^{(k,l)}(0) &\text{if}\ (i,j) = (k,l),\\
\rho^{(k,l)} (X_{i,j}) &\text{otherwise}
\end{cases}\]
and
\[\varphi (\rho(Y_{i,j})) \coloneqq 
\begin{cases}
\rho^{(k,l)}(-Y_{k+1,l-1})&\text{if}\ (i,j) = (k+1,l),\\
\rho^{(k,l)} (Y_{k+1,l-1} + Y_{k+1,l-1} X_{k+2,l-2} Y_{k+2, l-1})&\text{if}\ (i,j) = (k+1,l-1),\\
\rho^{(k,l)} (Y_{i,j}) &\text{otherwise}
\end{cases}\]
for $q_{i, j} \in \Pi_C \setminus \Pi^\ast_C$, where $\rho^{(k,l)} \colon \Bbbk [X_{i,j}, Y_{i,j} \mid q_{i, j} \in \mathcal{I}_{k,l}][Y_{k,l}] \twoheadrightarrow \mathcal{A}_C^{(k,l)}$ denotes the canonical projection. 
In a way similar to the proof of Proposition \ref{p:integral_domain}, we conclude that $\mathcal{A}_C^{(k,l)}$ is an integral domain, which implies the assertion of the proposition.
\end{proof}

Recall the explicit description of ${\bf B}_C$ given in \eqref{eq:standard_monomial_type_C}.
We define a map $\nu_C \colon \mathcal{A}_C \rightarrow S_{\mathcal{M}_C}$ by 
\begin{itemize}
\item $\nu_C (0) \coloneqq \infty$,
\item for $b = \prod_{q_{i, j} \in \Pi_C \setminus \Pi^\ast_C}X_{i,j}^{a_{i,j}} Y_{i,j}^{b_{i,j}} \in {\bf B}_C$, set
\[\displaystyle \nu_C (\overline{b}) \coloneqq \underset{q_{i, j} \in \Pi_C \setminus \Pi^\ast_C}{\bigstar} \varepsilon_{i, \leq j}^{\star a_{i,j}} \star (\varepsilon_{i, \leq j}^\prime)^{\star b_{i,j}},\] 
\item for each finite collection $c_i \in \Bbbk^\times$ and $b_i \in {\bf B}_C$, define $\nu_C (\sum_i c_i \overline{b_i})$ to be $\bigoplus_i \nu_C (\overline{b_i})$.
\end{itemize}

\begin{thm}\label{t:detropicalization_type_C}
The pair $(\mathcal{A}_C, \nu_C)$ is a detropicalization of $\mathcal{M}_C$ in the sense of Definition \ref{d:detropicalization}.
In addition, $\overline{{\bf B}_C}$ is a convex adapted basis for $\nu_C$ in the sense of Definition \ref{d:adapted_basis}.
\end{thm}

\begin{proof}
Let us prove that
\begin{equation}\label{eq:Grobner_cond_in_type_C}
\varepsilon_{i, \leq j} \star \varepsilon_{i, \leq j}^\prime = \nu_C \left(1 + X_{i+1,j-1} Y_{i+1, j}\right).
\end{equation}
Since 
\[\displaystyle \Upsilon (\varepsilon_{i, \leq j}, \varepsilon_{i, \leq j}^\prime) = 
\begin{cases}
\{0, \varepsilon_{i+1, \leq j}^\prime + \varepsilon_{i+1, \leq j-1}\} &\text{if}\ q_{i,j} \in \widehat{\Pi}_C^\circ,\\
\{0\} &\text{otherwise},
\end{cases}\]
it follows that
\[\varepsilon_{i, \leq j} \star \varepsilon_{i, \leq j}^\prime = 
\begin{cases}
0 \oplus \left(\varepsilon_{i+1, \leq j}^\prime \star \varepsilon_{i+1, \leq j-1}\right)&\text{if}\ q_{i,j} \in \widehat{\Pi}_C^\circ,\\
0 &\text{otherwise},
\end{cases}\]
where $\varepsilon_{i+1, \leq 0} \coloneqq 0$ when $j = 1$.
If $q_{i,j} \in \widehat{\Pi}_C^\circ$, then we have $\nu_C \left(1 + X_{i+1,j-1} Y_{i+1, j}\right) = 0 \oplus \left(\varepsilon_{i+1, \leq j}^\prime \star \varepsilon_{i+1, \leq j-1}\right)$ since $1, X_{i+1,j-1} Y_{i+1, j} \in {\bf B}_C$. 
If $q_{i,j} \notin \widehat{\Pi}_C^\circ$, then it follows that $\nu_C \left(1 + X_{i+1,j-1} Y_{i+1, j}\right) = \nu_C (1) = 0$. 
From these, we conclude \eqref{eq:Grobner_cond_in_type_C}. 

For $1 \leq k \leq n^2$, set 
\[{\bf B}_C^{(k)} \coloneqq \left\{\prod_{k \leq \ell \leq n^2} X_\ell^{a_\ell} Y_\ell^{b_\ell} \mid \min\{a_\ell, b_\ell\} = 0\ \text{for all}\ k \leq \ell \leq n^2\right\},\]
which is a $\Bbbk$-basis of $\mathcal{A}_C^{(k)}$. 
We define a map $\nu_C^{(k)} \colon \mathcal{A}_C^{(k)} \rightarrow S_{\mathcal{M}_C}$ by replacing ${\bf B}_C$ with ${\bf B}_C^{(k)}$ in the definition of $\nu_C$, which coincides with the restriction of $\nu_C$ to $\mathcal{A}_C^{(k)} \hookrightarrow \mathcal{A}_C$. 
Let us prove that $\nu_C^{(k)} \colon \mathcal{A}_C^{(k)} \rightarrow S_{\mathcal{M}_C}$ is a valuation by descending induction on $k$. 
If $k = n^2$, then we have 
\[\mathcal{A}_C^{(n^2)} = \Bbbk [X_{n^2}, Y_{n^2}]/(g_{n^2}) = \Bbbk [X_{n^2}, Y_{n^2}]/(X_{n^2} Y_{n^2} - 1),\]
and 
\[\nu_C^{(n^2)} \left(\overline{X_{n^2}^a Y_{n^2}^b}\right) = \varepsilon_{2n-1, \leq 1}^{\star a} \star (\varepsilon_{2n-1, \leq 1}^\prime)^{\star b}\] 
for $a, b \in \mathbb{Z}_{\geq 0}$ with $\min\{a, b\} = 0$. 
Since $\varepsilon_{2n-1, \leq 1} \star \varepsilon_{2n-1, \leq 1}^\prime = 0$, it is easy to see that $\nu_C^{(n^2)} \colon \mathcal{A}_C^{(n^2)} \rightarrow S_{\mathcal{M}_C}$ is a valuation. 
Let $1 \leq k < n^2$, and assume that $\nu_C^{(k+1)}$ is a valuation. 
We take $a, b \in \mathbb{Z}_{\geq 0}$ and $F \in \mathcal{A}_C^{(k+1)} \hookrightarrow \mathcal{A}_C^{(k)}$ with $F \neq 0$. 

Let us show that 
\begin{equation}\label{eq:type_C_inductive_step}
\begin{aligned}
\nu_C^{(k)} (\overline{X_k^a Y_k^b} \cdot F) = \varepsilon_{i, \leq j}^{\star a} \star (\varepsilon_{i, \leq j}^\prime)^{\star b} \star \nu_C^{(k+1)} (F),
\end{aligned}
\end{equation}
where we write $q_k = q_{i,j}$. 
Setting $c \coloneqq \min\{a, b\} \in \mathbb{Z}_{\geq 0}$, it holds that 
\[\overline{X_k^a Y_k^b} \cdot F = \overline{X_k^{a-c} Y_k^{b-c} (1 + X_{i+1,j-1} Y_{i+1, j})^c} \cdot F.\]
Since $\overline{(1 + X_{i+1,j-1} Y_{i+1, j})^c} \cdot F \in \mathcal{A}_C^{(k+1)}$, we can write $\overline{(1 + X_{i+1,j-1} Y_{i+1, j})^c} \cdot F = \sum_i c_i \overline{b_i}$ for some finite collection $c_i \in \Bbbk^\times$ and $b_i \in {\bf B}_C^{(k+1)}$. 
Then it follows that 
\[\overline{X_k^a Y_k^b} \cdot F = \sum_i c_i \overline{X_k^{a-c} Y_k^{b-c} b_i}.\]
Since $X_k^{a-c} Y_k^{b-c} b_i \in {\bf B}_C^{(k)}$, we deduce by the definition of $\nu_C^{(k)}$ that 
\begin{align*}
\nu_C^{(k)} \left(\overline{X_k^a Y_k^b} \cdot F\right) &= \bigoplus_i \nu_C^{(k)} \left(\overline{X_k^{a-c} Y_k^{b-c} b_i}\right) = \bigoplus_i \left(\varepsilon_{i, \leq j}^{\star a-c} \star (\varepsilon_{i, \leq j}^\prime)^{\star b-c} \star \nu_C^{(k+1)} \left(\overline{b_i}\right)\right)\\
&= \varepsilon_{i, \leq j}^{\star a-c} \star (\varepsilon_{i, \leq j}^\prime)^{\star b-c} \star \left(\bigoplus_i \nu_C^{(k+1)} \left(\overline{b_i}\right)\right)\\
&= \varepsilon_{i, \leq j}^{\star a-c} \star (\varepsilon_{i, \leq j}^\prime)^{\star b-c} \star \nu_C^{(k+1)} \left(\overline{(1 + X_{i+1,j-1} Y_{i+1, j})^c} \cdot F\right).
\end{align*}
Since $\nu_C^{(k+1)}$ is a valuation, it follows that 
\begin{align*}
\nu_C^{(k+1)} (\overline{(1 + X_{i+1,j-1} Y_{i+1, j})^c} \cdot F) &= \nu_C^{(k+1)} \left(\overline{1 + X_{i+1,j-1} Y_{i+1, j}}\right)^{\star c} \star \nu_C^{(k+1)} (F)\\
&= \varepsilon_{i, \leq j}^{\star c} \star (\varepsilon_{i, \leq j}^\prime)^{\star c} \star \nu_C^{(k+1)} (F)\quad (\text{by}\ \eqref{eq:Grobner_cond_in_type_C}),
\end{align*}
which implies \eqref{eq:type_C_inductive_step}. 

Combining \eqref{eq:type_C_inductive_step} with the argument in \cite[Proof of Lemma 8.10]{EHM}, we deduce that the map $\nu_C^{(k)} \colon \mathcal{A}_C^{(k)} \rightarrow S_{\mathcal{M}_C}$ is a valuation. 
Setting $k = 1$, it follows that $\nu_C \colon \mathcal{A}_C \rightarrow S_{\mathcal{M}_C}$ is a valuation. 

For $b = \prod_{q_{i, j} \in \Pi_C \setminus \Pi^\ast_C}X_{i,j}^{a_{i,j}} Y_{i,j}^{b_{i,j}} \in {\bf B}_C$, define $m_b \in \mathcal{M}_C$ to be the element corresponding to 
\[\sum_{q_{i, j} \in \Pi_C \setminus \Pi^\ast_C} (a_{i,j}- b_{i,j}) {\bm e}_{i, \leq j} \in M_\emptyset^{(C)}.\]
Then a map ${\bf B}_C \rightarrow \mathcal{M}_C$ defined by $b \mapsto m_b$ is bijective.
Since $\min\{a_{i,j}, b_{i,j}\} = 0$ for all $q_{i, j} \in \Pi_C \setminus \Pi^\ast_C$, we deduce that 
\[\displaystyle \underset{q_{i, j} \in \Pi_C \setminus \Pi^\ast_C}{\bigstar} \varepsilon_{i, \leq j}^{\star a_{i,j}} \star (\varepsilon_{i, \leq j}^\prime)^{\star b_{i,j}} = m_b\]
and hence that $\nu_C (\overline{b}) = m_b$. 
This implies that $\nu_C$ induces a bijective map $\overline{{\bf B}_C} \rightarrow \mathcal{M}_C$ given by $\overline{b} \mapsto \nu_C (\overline{b}) = m_b$. 
Hence the pair $(\mathcal{A}_C, \nu_C)$ is a detropicalization of $\mathcal{M}_C$, and $\overline{{\bf B}_C}$ is an adapted basis for $\nu_C$. 
This concludes the theorem. 
\end{proof}

In the rest of this section, we assume the following conditions:
\begin{enumerate}
\item[($\dagger 1$)] $\lambda_k - \lambda_{k-1} \geq 2$ for all $1 \leq k \leq n$, where $\lambda_0 \coloneqq 0$;
\item[($\dagger 2$)] the fixed element ${\bm u} = (u_p)_{p \in \Pi_C \setminus \Pi_C^\ast} \in {\mathcal O}(\Pi_C, \Pi_C^\ast, \lambda)  \cap \z^{\Pi_C \setminus \Pi_C^\ast}$ satisfies that $u_{p_1} < u_{p_2}$ for all $p_1, p_2 \in \Pi_C$ with $p_1 \lessdot p_2$, where $u_p \coloneqq \lambda_p$ for $p \in \Pi_C^\ast$.
\end{enumerate}
Note that the condition ($\dagger 1$) is equivalent to the existence of ${\bm u}$ satisfying \eqref{eq:assumption} and ($\dagger 2$).
For instance, the marking $\lambda^r$ given by $(\lambda^r)_a = r(a)$ for $a \in \Pi_C^\ast$ and the element ${\bm u}^r = (u^r_p)_{p \in \Pi_C \setminus \Pi_C^\ast} \in {\mathcal O}(\Pi_C, \Pi_C^\ast, \lambda^r)  \cap \z^{\Pi_C \setminus \Pi_C^\ast}$ given by $u^r_p \coloneqq r(p)$ satisfy these conditions. 
The condition ($\dagger 2$) implies that ${\bm u}$ is an interior lattice point of ${\mathcal O}(\Pi_C, \Pi_C^\ast, \lambda)$, which implies that $\widehat{\Delta}(\Pi_C, \Pi^\ast_C, \lambda)$ contains the origin as an interior point. 
For $f = c_1 \overline{b_1} + \cdots + c_\ell \overline{b_\ell} \in \mathcal{A}_C$ with $c_1, \ldots, c_\ell \in \Bbbk^\times$ and distinct $b_1, \ldots, b_\ell \in {\bf B}_C$, define ${\rm supp}(f) \subseteq (\mathcal{M}_C)_{\mathbb{R}}$ to be $\text{p-conv}_\mathbb{R}(\{\nu_C(\overline{b_1}), \ldots, \nu_C(\overline{b_\ell})\})$ (see \cite[Definition 7.1]{EHM}). 
For $\widehat{\Delta}_C \coloneqq \widehat{\Delta}(\Pi_C, \Pi^\ast_C, \lambda)$ and $k \in \mathbb{Z}_{\geq 0}$, we set 
\[\Gamma (\mathcal{A}_C, k \widehat{\Delta}_C) \coloneqq \{f \in \mathcal{A}_C \setminus \{0\} \mid {\rm supp}(f) \subseteq k \widehat{\Delta}_C\} \cup \{0\}.\]
By \cite[Lemma 5.3]{EHM}, for $f = c_1 \overline{b_1} + \cdots + c_\ell \overline{b_\ell} \in \mathcal{A}_C$ with $c_1, \ldots, c_\ell \in \Bbbk^\times$ and distinct $b_1, \ldots, b_\ell \in {\bf B}_C$, we have ${\rm supp}(f) \subseteq k \widehat{\Delta}_C$ if and only if $\nu_C(\overline{b_1}), \ldots, \nu_C(\overline{b_\ell}) \in k \widehat{\Delta}_C$. 
In particular, $\Gamma (\mathcal{A}_C, k \widehat{\Delta}_C)$ is a $\Bbbk$-linear subspace of $\mathcal{A}_C$. 
We define a $\mathbb{Z}_{\geq 0}$-graded $\Bbbk$-algebra $\mathcal{A}_C^{\widehat{\Delta}_C}$ by
\[\mathcal{A}_C^{\widehat{\Delta}_C} \coloneqq \bigoplus_{k \in \mathbb{Z}_{\geq 0}} \Gamma (\mathcal{A}_C, k \widehat{\Delta}_C)  \cdot t^k \subseteq \mathcal{A}_C[t],\]
where $t$ is a formal variable (see \cite[Definition 7.6]{EHM}).
Let 
\[X_{\mathcal{A}_C} (\widehat{\Delta}_C) \coloneqq {\rm Proj}(\mathcal{A}_C^{\widehat{\Delta}_C})\]
be the projective scheme corresponding to the $\mathbb{Z}_{\geq 0}$-graded $\Bbbk$-algebra $\mathcal{A}_C^{\widehat{\Delta}_C}$. 
Since $\mathcal{A}_C$ is an integral domain by Proposition \ref{p:integral_domain}, it follows that $\mathcal{A}_C^{\widehat{\Delta}_C}$ is also an integral domain and that $X_{\mathcal{A}_C} (\widehat{\Delta}_C)$ is irreducible. 
By Proposition \ref{p:PL_description} and Theorems \ref{t:strict_dual_pair}, \ref{t:detropicalization_type_C}, we see that the assumption of \cite[Theorem 7.11]{EHM} is satisfied, which implies the following.

\begin{thm}\label{t:compactification}
The inclusion $\mathcal{A}_C^{\widehat{\Delta}_C} \hookrightarrow \mathcal{A}_C[t]$ of $\mathbb{Z}_{\geq 0}$-graded $\Bbbk$-algebras induces an open embedding ${\rm Spec}(\mathcal{A}_C) \simeq {\rm Proj}(\mathcal{A}_C[t]) \hookrightarrow X_{\mathcal{A}_C} (\widehat{\Delta}_C)$. 
In addition, the complement of ${\rm Spec}(\mathcal{A}_C)$ in $X_{\mathcal{A}_C} (\widehat{\Delta}_C)$ is a union of divisors $D_p$ for $p \in \Pi_C \setminus \Pi_C^\ast$ and $D_{p,p^\prime}$ for $p \in \Pi_C^\ast, p^\prime \in \Pi_C \setminus \Pi_C^\ast$ with $p^\prime \lessdot p$ which are in bijection with $\varphi_p$ and $\varphi_{p,p^\prime}$ in Propositions \ref{p:PL_description}, respectively. 
\end{thm}

By \cite[Theorem 18]{FFP}, it follows that $\widehat{\Delta}(\Pi_C, \Pi^\ast_C, \lambda)$ is normal in the sense of \cite[Definition 7.13]{EHM}, which implies by \cite[Lemma 7.14]{EHM} that the $\mathbb{Z}_{\geq 0}$-graded algebra $\mathcal{A}_C^{\widehat{\Delta}_C}$ is generated in degree $1$. 
By Proposition \ref{p:nonsingular} and \cite[Proposition 7.16 and Theorem 7.18]{EHM}, we obtain the following.

\begin{thm}
The scheme $X_{\mathcal{A}_C} (\widehat{\Delta}_C)$ is normal, and the valuation ${\rm ord}_{D_p} \colon \mathcal{A}_C \setminus \{0\} \rightarrow \mathbb{Z}$ corresponding to the divisor $D_p$ for $p \in \Pi_C \setminus \Pi_C^\ast$ is given by $\varphi_p \circ \nu_C$. 
Similarly, it holds that ${\rm ord}_{D_{p,p^\prime}} = \varphi_{p,p^\prime} \circ \nu_C$ on $\mathcal{A}_C \setminus \{0\}$ for $p \in \Pi_C^\ast, p^\prime \in \Pi_C \setminus \Pi_C^\ast$ with $p^\prime \lessdot p$.
In addition, $\mathcal{A}_C^{\widehat{\Delta}_C}$ is Cohen--Macaulay, and $X_{\mathcal{A}_C} (\widehat{\Delta}_C)$ is arithmetic Cohen--Macaulay in the sense of \cite[Definition 7.17]{EHM}.
\end{thm}

Let $\tilde{\rho}$ be an (ordered) $\mathbb{Z}$-basis of $\mathbb{Z}^{\Pi_C \setminus \Pi_C^\ast}$ that is contained in $\sigma_{{\bm \epsilon}(\mathcal{C})}^\prime = {\mathsf w}^{-1} ({\rm Sp}_{\mathbb{R}}(\mathcal{M}_C, \mathcal{C}))$. 
Then we obtain a valuation $\mathfrak{v}_{\mathcal{C}, \tilde{\rho}} \colon \mathcal{A}_C^{\widehat{\Delta}_C} \setminus \{0\} \rightarrow \mathbb{Z}^{\Pi_C \setminus \Pi_C^\ast} \times \mathbb{Z}$ as in \cite[Section 7.1]{EHM}. 
The Newton--Okounkov body $\Delta(\mathcal{A}_C^{\widehat{\Delta}_C}, \mathfrak{v}_{\mathcal{C}, \tilde{\rho}})$ of $\mathcal{A}_C^{\widehat{\Delta}_C}$ is defined by 
\[\Delta(\mathcal{A}_C^{\widehat{\Delta}_C}, \mathfrak{v}_{\mathcal{C}, \tilde{\rho}}) \coloneqq \overline{\bigcup_{k \in \mathbb{Z}_{>0}} \left\{\frac{1}{k} {\bm a} \mid ({\bm a}, k) \in \mathfrak{v}_{\mathcal{C}, \tilde{\rho}} \left(\mathcal{A}_C^{\widehat{\Delta}_C} \setminus \{0\}\right)\right\}}.\]
By \cite[Theorems 7.21 and 7.22]{EHM}, we realize the marked chain-order polytope $\widehat{\Delta}_{\mathcal{C}}(\Pi_C, \Pi_C^\ast, \lambda)$ as a Newton--Okounkov body of $\mathcal{A}_C^{\widehat{\Delta}_C}$ and a toric degeneration of $X_{\mathcal{A}_C} (\widehat{\Delta}_C)$ as follows. 

\begin{thm}\label{t:NO_degeneration_type_C}
The Newton--Okounkov body $\Delta(\mathcal{A}_C^{\widehat{\Delta}_C}, \mathfrak{v}_{\mathcal{C}, \tilde{\rho}})$ coincides with the marked chain-order polytope $\widehat{\Delta}_{\mathcal{C}}(\Pi_C, \Pi_C^\ast, \lambda)$. 
In addition, there exists a flat morphism $\mathcal{X}_{\mathcal{C}, \tilde{\rho}} \rightarrow {\rm Spec}(\Bbbk [t])$ whose (scheme-theoretic) fibers over $t \in \Bbbk^\times$ are all isomorphic to $X_{\mathcal{A}_C} (\widehat{\Delta}_C)$ and whose central fiber over $0 \in \Bbbk$ is isomorphic to the normal projective toric variety associated to the integral polytope $\widehat{\Delta}_{\mathcal{C}}(\Pi_C, \Pi_C^\ast, \lambda)$.   
\end{thm}

\subsection{Cox rings of compactifications}

In this subsection, we compute the Cox ring of $X_{\mathcal{A}_C} (\widehat{\Delta}_C)$, using the arguments in \cite[Construction 1.2.4.1]{ADHL}, \cite[Theorem 7.19]{EHM}, and \cite[Section 7]{CEHM}. 
Let ${\rm Cl}(X_{\mathcal{A}_C} (\widehat{\Delta}_C))$ and ${\rm Cox}(X_{\mathcal{A}_C} (\widehat{\Delta}_C))$ denote the divisor class group and the Cox ring of $X_{\mathcal{A}_C} (\widehat{\Delta}_C)$, respectively. 
Since $\mathcal{A}_C$ is a UFD by Proposition \ref{p:UFD}, the divisor class group ${\rm Cl}(X_{\mathcal{A}_C} (\widehat{\Delta}_C))$ is generated by the images of the divisors $D_p$ for $p \in \Pi_C \setminus \Pi_C^\ast$ and $D_{p,p^\prime}$ for $p \in \Pi_C^\ast, p^\prime \in \Pi_C \setminus \Pi_C^\ast$ with $p^\prime \lessdot p$ in Theorem \ref{t:compactification} (see the proof of \cite[Theorem 7.19]{EHM}).
By \cite[Construction 1.2.4.1]{ADHL}, the Cox ring of $X_{\mathcal{A}_C} (\widehat{\Delta}_C)$ is described as 
\[{\rm Cox}(X_{\mathcal{A}_C} (\widehat{\Delta}_C)) \simeq H^0(X_{\mathcal{A}_C} (\widehat{\Delta}_C), \mathcal{S})/H^0(X_{\mathcal{A}_C} (\widehat{\Delta}_C), \mathcal{I}),\]
where $\mathcal{S}$ is a sheaf of divisorial algebras given by
\[\mathcal{S} \coloneqq \bigoplus_{D \in K} \mathcal{O}_{X_{\mathcal{A}_C} (\widehat{\Delta}_C)} (D)\]
for $K \coloneqq \{D_p \mid p \in \Pi_C \setminus \Pi_C^\ast\} \cup \{D_{p,p^\prime} \mid p \in \Pi_C^\ast,\ p^\prime \in \Pi_C \setminus \Pi_C^\ast,\ p^\prime \lessdot p\}$, and $\mathcal{I}$ is a sheaf of ideals of $\mathcal{S}$. 
Write $K = \{D^{(1)}, \ldots, D^{(L)}\}$, where $L =L_{\mathcal{M}_C}$ in \eqref{eq:number_of_divisors}.
By \cite[(7.10)]{CEHM}, $H^0(X_{\mathcal{A}_C} (\widehat{\Delta}_C), \mathcal{S})$ can be described as a $\Bbbk$-subalgebra 
\[\bigoplus_{{\bm r} = (r_1, \ldots, r_L) \in \mathbb{Z}^L} \mathcal{A}_C({\bm r}) t_1^{r_1} \cdots t_L^{r_L}\]
of a Laurent polynomial ring $\mathcal{A}_C[t_1^{\pm 1}, \ldots, t_L^{\pm 1}]$, where $\mathcal{A}_C({\bm r})$ is the $\Bbbk$-linear subspace of $\mathcal{A}_C$ spanned by
\[\{\overline{b} \mid b \in {\bf B}_C,\ {\rm ord}_{D^{(k)}} (\overline{b}) + r_k \geq 0\ (1 \leq k \leq L)\}.\]
In addition, we see by \cite[Lemma 7.12]{CEHM} that the inclusion map $H^0(X_{\mathcal{A}_C} (\widehat{\Delta}_C), \mathcal{S}) \hookrightarrow \mathcal{A}_C[t_1^{\pm 1}, \ldots, t_L^{\pm 1}]$ induces an injective $\Bbbk$-algebra homomorphism 
\[H^0(X_{\mathcal{A}_C} (\widehat{\Delta}_C), \mathcal{S})/H^0(X_{\mathcal{A}_C} (\widehat{\Delta}_C), \mathcal{I}) \hookrightarrow \mathcal{A}_C[t_1^{\pm 1}, \ldots, t_L^{\pm 1}]/\mathcal{J},\]
where $\mathcal{J}$ denotes the ideal of $\mathcal{A}_C[t_1^{\pm 1}, \ldots, t_L^{\pm 1}]$ generated by $u - t_1^{{\rm ord}_{D^{(1)}} (u)} \cdots t_L^{{\rm ord}_{D^{(L)}} (u)}$ for all units $u \in \mathcal{A}_C^\times$. 
When $D^{(k)} = D_p$ (respectively, $D^{(k)} = D_{p,p^\prime}$), we set $t_p \coloneqq t_k$ and $\varphi^{(k)} \coloneqq \varphi_p$ (respectively, $t_{p,p^\prime} \coloneqq t_k$ and $\varphi^{(k)} \coloneqq \varphi_{p,p^\prime}$).
By Proposition \ref{p:unit_elements}, it holds that $\overline{X_{i,j}} \in \mathcal{A}_C^\times$ for $q_{i, j} \in \Pi_C \setminus (\Pi^\ast_C \cup \widehat{\Pi}_C^\circ)$.
In this case, we can write $(i,j) = (2s-1, n+1-s)$ for some $1 \leq s \leq n$. 
Note that the marking of $q_{i+1,j} = q_{2s,n+1-s}$ is $\lambda_{s}$.
Setting $u = \overline{X_{i,j}}$, we have
\[{\rm ord}_{D_p} (u) = \varphi_{p} \circ \nu_C (u) = \varphi_{p} (\varepsilon_{i, \leq j}) = 
\begin{cases}
1 &\text{if}\ p = q_{i,\ell}\ \text{for some}\ \ell,\\
-1 &\text{if}\ p = q_{i+1,\ell}\ \text{for some}\ \ell,\\
0 &\text{otherwise}
\end{cases}
\]
for $p \in \Pi_C \setminus \Pi_C^\ast$, and 
\[{\rm ord}_{D_{p,p^\prime}} (u) = \varphi_{p,p^\prime} \circ \nu_C (u) = \varphi_{p,p^\prime} (\varepsilon_{i, \leq j}) = 
\begin{cases}
-1 &\text{if}\ p = q_{i+1, j}\ \text{and}\ p^\prime = q_{i, j},\\
0 &\text{otherwise}
\end{cases}
\]
for $p \in \Pi_C^\ast$ and $p^\prime \in \Pi_C \setminus \Pi_C^\ast$ such that $p^\prime \lessdot p$.
Hence it follows that
\begin{equation}\label{eq:ideal_generator}
\begin{aligned}
u - t_1^{{\rm ord}_{D^{(1)}} (u)} \cdots t_L^{{\rm ord}_{D^{(L)}} (u)} = \overline{X_{i,j}} - t_{q_{i,1}} \cdots t_{q_{i,j}} t_{q_{i+1,1}}^{-1} \cdots t_{q_{i+1,j-1}}^{-1} t_{q_{i+1, j}, q_{i, j}}^{-1}.
\end{aligned}
\end{equation}
By Propositon \ref{p:unit_elements}, the ideal $\mathcal{J}$ is generated by \eqref{eq:ideal_generator} for all $q_{i, j} \in \Pi_C \setminus (\Pi^\ast_C \cup \widehat{\Pi}_C^\circ)$.
Define $\eta_{s} \in \mathcal{A}_C[t_1^{\pm 1}, \ldots, t_L^{\pm 1}]$ by 
\[\eta_{s} \coloneqq \overline{X_{i,j}} t_{q_{i,1}}^{-1} \cdots t_{q_{i,j}}^{-1} t_{q_{i+1,1}} \cdots t_{q_{i+1,j-1}} t_{q_{i+1, j}, q_{i, j}}.\]
Then we have $\eta_{s} - 1 \in \mathcal{J}$. 
Set 
\[{\bm v}_s \coloneqq ({\bm e}_{q_{i,j}} + \cdots + {\bm e}_{q_{i,1}}) - ({\bm e}_{r_{i,j}} + \cdots + {\bm e}_{r_{i,1}}) + ({\bm e}_{r_{i+1,j-1}} + \cdots + {\bm e}_{r_{i+1,1}}) + {\bm e}_{r^\ast_s} \in \mathbb{Z}^{\Pi_C \setminus \Pi_C^\ast} \times \mathbb{Z}^L,\]
where ${\bm e}_{r_{k,\ell}} \in \mathbb{Z}^L$ for $q_{k,\ell} \in \Pi_C \setminus \Pi^\ast_C$ (respectively, ${\bm e}_{r^\ast_s} \in \mathbb{Z}^L$) denotes the unit vector for the coordinate corresponding to $t_{q_{k,\ell}}$ (respectively, $t_{q_{i+1, j}, q_{i, j}}$).
Then it follows that ${\bm v}_s \in V(\mathcal{M}_C, \emptyset) \times \mathbb{Z}^L$.
Setting 
\[\mathscr{H}_C \coloneqq \{(m, {\bm r}) \in \mathcal{M}_C \times \mathbb{Z}^L \mid \varphi^{(k)} (m) + r_k \geq 0\ (1 \leq k \leq L)\},\]
we have 
\[\mathcal{A}_C({\bm r}) = {\rm Span}_{\Bbbk} \{\overline{b} \mid b \in {\bf B}_C,\ (m_b, {\bm r}) \in \mathscr{H}_C\}.\]
Recall that $\Sigma(\mathcal{M}_C)$ denotes the PL fan of $\mathcal{M}_C$. 
For $\sigma \in \Sigma(\mathcal{M}_C)$, write $\mathscr{H}_C(\sigma) \coloneqq \mathscr{H}_C \cap (\sigma \times \mathbb{Z}^L)$ and 
\[\mathscr{H}_{C, \emptyset}(\sigma) \coloneqq \{(\pi_\emptyset(m), {\bm r}) \mid (m, {\bm r}) \in \mathscr{H}_C (\sigma)\} \subseteq \mathbb{Z}^{\Pi_C \setminus \Pi_C^\ast} \times \mathbb{Z}^L.\] 
For $1 \leq k \leq L$, we define $\widehat{\varphi}^{(k)}_\emptyset \colon \mathbb{Z}^{\Pi_C \setminus \Pi_C^\ast} \times \mathbb{Z}^L \rightarrow \mathbb{Z}$ by $\widehat{\varphi}^{(k)}_\emptyset ({\bm x}, (r_1, \ldots, r_L)) \coloneqq \varphi^{(k)} (\pi_\emptyset^{-1} ({\bm x})) + r_k$ for $({\bm x}, (r_1, \ldots, r_L)) \in \mathbb{Z}^{\Pi_C \setminus \Pi_C^\ast} \times \mathbb{Z}^L$. 
Then it follows that 
\[\mathscr{H}_{C, \emptyset}(\sigma) = \{({\bm x}, {\bm r}) \in \pi_\emptyset(\sigma) \times \mathbb{Z}^L \mid \widehat{\varphi}^{(k)}_\emptyset ({\bm x}, {\bm r}) \geq 0\ (1 \leq k \leq L)\}.\] 
Since $\varphi^{(k)} \circ \pi_\emptyset^{-1}$ is linear on $\pi_\emptyset(\sigma)$ for each $1 \leq k \leq L$, we see that $\mathscr{H}_{C, \emptyset}(\sigma)$ is the set of lattice points in a rational convex polyhedral cone. 
In particular, $\mathscr{H}_{C, \emptyset}(\sigma)$ is a semigroup. 
By the definition of $\varphi^{(k)}$, we see that $\widehat{\varphi}^{(k)}_\emptyset ({\bm v}_{s}) = \widehat{\varphi}^{(k)}_\emptyset (-{\bm v}_s) = 0$ for all $1 \leq k \leq L$, which implies that $\pm {\bm v}_s \in \mathscr{H}_{C, \emptyset}(\sigma)$ for all $\sigma \in \Sigma(\mathcal{M}_C)$ and $1 \leq s \leq n$.
For $q_{i,j} \in \widehat{\Pi}_C^\circ$ and $\epsilon \in \{\pm 1\}$, we define ${\bm f}_{q_{i,j},\epsilon}$ by
\[{\bm f}_{q_{i,j}, \epsilon} \coloneqq
\begin{cases}
{\bm e}_{q_{i,j}} - {\bm e}_{r_{i,j}} + {\bm e}_{r_{i+1,j}} + \sum_{\ell = 1}^{j-1} ({\bm e}_{q_{i,\ell}} - {\bm e}_{r_{i,\ell}} + {\bm e}_{r_{i+1,\ell}}) &\text{if}\ \epsilon = 1,\\ 
-{\bm e}_{q_{i,j}} + {\bm e}_{r_{i,j}} + \sum_{\ell = 1}^{j-1} (-{\bm e}_{q_{i,\ell}} + {\bm e}_{r_{i,\ell}} - {\bm e}_{r_{i+1,\ell}}) &\text{if}\ \epsilon = -1.
\end{cases}\]
In particular, we have ${\bm f}_{q_{i,j}, 1} + {\bm f}_{q_{i,j}, -1} = {\bm e}_{r_{i+1,j}}$. 

\begin{ex}
Let $n = 2$. 
Then we have 
\begin{align*}
&{\bm v}_1 = ({\bm e}_{q_{1,2}} + {\bm e}_{q_{1,1}}) - ({\bm e}_{r_{1,2}} + {\bm e}_{r_{1,1}}) + {\bm e}_{r_{2,1}} + {\bm e}_{r^\ast_1}, &&&{\bm v}_2 = {\bm e}_{q_{2,1}} - {\bm e}_{r_{2,1}} + {\bm e}_{r^\ast_2},\\
&{\bm f}_{q_{1,1}, 1} = {\bm e}_{q_{1,1}} - {\bm e}_{r_{1,1}} + {\bm e}_{r_{2,1}},&&&{\bm f}_{q_{1,1}, -1} = -{\bm e}_{q_{1,1}} + {\bm e}_{r_{1,1}},\\
&{\bm f}_{q_{2,1}, 1} = {\bm e}_{q_{2,1}} - {\bm e}_{r_{2,1}} + {\bm e}_{r_{3,1}},&&&{\bm f}_{q_{2,1}, -1} = -{\bm e}_{q_{2,1}} + {\bm e}_{r_{2,1}}.
\end{align*}
\end{ex}

\begin{prop}\label{p:generator_Cox}
For $\sigma_{\bm \epsilon} \in \Sigma(\mathcal{M}_C)$ with ${\bm \epsilon} = (\epsilon_{i,j} \mid q_{i,j} \in \widehat{\Pi}_C^\circ)$, the semigroup $\mathscr{H}_{C, \emptyset}(\sigma_{\bm \epsilon})$ is generated by ${\bm e}_{r_1}, \ldots, {\bm e}_{r_L}, \pm {\bm v}_1, \ldots, \pm {\bm v}_n$, and ${\bm f}_{q_{i,j}, \epsilon_{i,j}}$ for $q_{i,j} \in \widehat{\Pi}_C^\circ$, where ${\bm e}_{r_1}, \ldots, {\bm e}_{r_L} \in \mathbb{Z}^L$ denote the unit vectors.
\end{prop}

\begin{proof}
For $q_{i,j} \in \widehat{\Pi}_C^\circ$ and ${\bm x} = (x_{i,j})_{i,j} \in \mathbb{Z}^{\Pi_C \setminus \Pi_C^\ast}$, define $\psi_{i+1,j}({\bm x}) \in \mathbb{Z}$ by 
\[\psi_{i+1,j}({\bm x}) \coloneqq 
\begin{cases}
x_{i+1,j}-x_{i,j} &\text{if}\ \epsilon_{i,j} = 1,\\ 
x_{i+1,j}-x_{i,j+1} &\text{if}\ \epsilon_{i,j} = -1.
\end{cases}\]
For $m \in \sigma_{\bm \epsilon}$, it follows that $\varphi_{q_{i+1,j}} (m) = \psi_{i+1,j}(\pi_\emptyset (m))$. 
For $1 \leq s \leq n$ and ${\bm x} = (x_{i,j})_{i,j} \in \mathbb{Z}^{\Pi_C \setminus \Pi_C^\ast}$, write $\psi_{s}({\bm x}) \coloneqq -x_{i^\prime, j^\prime}$, where $(i^\prime, j^\prime) \coloneqq (2s-1,n+1-s)$.
Then we have $\varphi_{q_{i^\prime +1, j^\prime}, q_{i^\prime, j^\prime}} (m) = \psi_{s}(\pi_\emptyset(m))$ for $m \in \mathcal{M}_C$.
By definition, the semigroup $\mathscr{H}_{C, \emptyset}(\sigma_{\bm \epsilon})$ is the set of $((x_{i,j})_{i,j}, {\bm r}) \in \mathbb{Z}^{\Pi_C \setminus \Pi_C^\ast} \times \mathbb{Z}^L$ satisfying 
\begin{align*}
&\gamma_{i,j}^{(1)} \coloneqq \epsilon_{i,j} (x_{i,j} - x_{i,j+1}) \geq 0\quad \text{for all}\ q_{i,j} \in \widehat{\Pi}_C^\circ,\\
&\gamma_{i,j}^{(2)} \coloneqq \psi_{i+1,j}({\bm x}) + r_{i+1,j} \geq 0\quad \text{for all}\ q_{i,j} \in \widehat{\Pi}_C^\circ,\\
&\gamma_{s}^{(1)} \coloneqq x_{1,s} + r_{1,s} \geq 0\quad \text{for all}\ 1 \leq s \leq n,\\
&\gamma_{s}^{(2)} \coloneqq \psi_{s} ({\bm x}) + r^\ast_{s} \geq 0\quad \text{for all}\ 1 \leq s \leq n,
\end{align*}
where $r_{k,\ell}$ for $q_{k,\ell} \in \Pi_C \setminus \Pi^\ast_C$ (respectively, $r^\ast_s$) denotes the coordinate of ${\bm r}$ corresponding to $t_{q_{k,\ell}}$ (respectively, $t_{q_{i^\prime+1, j^\prime}, q_{i^\prime, j^\prime}}$).
Note that the number of inequalities is $2|\Pi_C \setminus \Pi^\ast_C|$, which coincides with ${\rm rank}(\mathbb{Z}^{\Pi_C \setminus \Pi_C^\ast} \times \mathbb{Z}^L) - \dim_{\mathbb{R}} (V(\mathcal{M}_C, \emptyset))$ since $L = |\Pi_C \setminus \Pi_C^\ast| + n$ and $\dim_{\mathbb{R}} (V(\mathcal{M}_C, \emptyset)) = n$. 
Hence if we fix ${\bm \epsilon} = (\epsilon_{i,j} \mid q_{i,j} \in \widehat{\Pi}_C^\circ)$, then we obtain a unimodular transformation $\mathbb{Z}^{\Pi_C \setminus \Pi_C^\ast} \times \mathbb{Z}^L \rightarrow \mathbb{Z}^{\Pi_C \setminus \Pi_C^\ast} \times \mathbb{Z}^L$ given by 
\begin{align*}
((x_{i,j})_{i,j}, {\bm r}) \mapsto ((\gamma_{i,j}^{(1)}, \gamma_{i,j}^{(2)} \mid q_{i,j} \in \widehat{\Pi}_C^\circ), (\gamma_{s}^{(1)}, \gamma_{s}^{(2)}, x_{2s-1,n+1-s} \mid 1 \leq s \leq n)).
\end{align*}
Since the elements ${\bm e}_{r_1}, \ldots, {\bm e}_{r_L}, {\bm v}_1, \ldots, {\bm v}_n$, and ${\bm f}_{q_{i,j}, \epsilon_{i,j}}$ for $q_{i,j} \in \widehat{\Pi}_C^\circ$ bijectively correspond to the unit vectors under this unimodular transformation, we conclude the assertion of the proposition. 
\end{proof}

Let us consider a polynomial ring $\Bbbk[W_q \mid q \in \widehat{\Pi}_C^\circ][Z_q, t_q \mid q \in \Pi_C \setminus \Pi_C^\ast]$, and define a $\Bbbk$-algebra $\widetilde{{\rm Cox}}(X_{\mathcal{A}_C} (\widehat{\Delta}_C))$ to be the quotient ring 
\[\Bbbk[W_q \mid q \in \widehat{\Pi}_C^\circ][Z_q, t_q \mid q \in \Pi_C \setminus \Pi_C^\ast]/(W_{q_{i,j}} Z_{q_{i,j}} - t_{q_{i+1,j}} - W_{q_{i+1,j-1}} Z_{q_{i+1,j}} \mid q_{i,j} \in \widehat{\Pi}_C^\circ),\]
where $W_{q_{i+1,0}} \coloneqq 1$ when $j = 1$.

\begin{thm}\label{t:Cox_type_C}
The Cox ring ${\rm Cox}(X_{\mathcal{A}_C} (\widehat{\Delta}_C))$ of $X_{\mathcal{A}_C} (\widehat{\Delta}_C)$ is isomorphic to $\widetilde{{\rm Cox}}(X_{\mathcal{A}_C} (\widehat{\Delta}_C))$, 
and hence to the polynomial ring in $|\Pi_C \setminus \Pi_C^\ast| - n +L = 2|\Pi_C \setminus \Pi_C^\ast| = 2n^2$ variables as a $\Bbbk$-algebra.
\end{thm}

\begin{proof}
By Proposition \ref{p:generator_Cox}, $H^0(X_{\mathcal{A}_C} (\widehat{\Delta}_C), \mathcal{S})$ is generated by $t_1, \ldots, t_L,\eta_1^{\pm 1}, \ldots, \eta_n^{\pm 1}$, and $\xi_{i,j}, \xi_{i,j}^\prime$ for $q_{i,j} \in \widehat{\Pi}_C^\circ$, where 
\begin{align*}
&\xi_{i,j} \coloneqq \overline{X_{i,j}} \cdot t_{q_{i,j}}^{-1} t_{q_{i+1,j}} \cdot \prod_{\ell = 1}^{j-1} (t_{q_{i,\ell}}^{-1} t_{q_{i+1,\ell}}),\\
&\xi_{i,j}^\prime \coloneqq \overline{Y_{i,j}} \cdot t_{q_{i,j}} \cdot \prod_{\ell = 1}^{j-1} (t_{q_{i,\ell}} t_{q_{i+1,\ell}}^{-1}).
\end{align*} 
Recall that the ideal $\mathcal{J}$ is generated by $\eta_1-1, \ldots, \eta_n-1$. 
For $1 \leq s \leq n$, write $\xi_{i^\prime, j^\prime}^\prime \coloneqq t_{q_{i^\prime +1, j^\prime}, q_{i^\prime, j^\prime}}$, where $(i^\prime, j^\prime) \coloneqq (2s-1,n+1-s)$. 
Then the relation $\eta_{s} = 1 \bmod \mathcal{J}$ can be written as 
\[\xi_{i^\prime, j^\prime}^\prime = \overline{Y_{i^\prime, j^\prime}} \cdot t_{q_{i^\prime, j^\prime}} \cdot \prod_{\ell = 1}^{j^\prime-1} (t_{q_{i^\prime,\ell}} t_{q_{i^\prime +1,\ell}}^{-1}) \bmod \mathcal{J}.\]
For $q_{i,j} \in \widehat{\Pi}_C^\circ$, we see that 
\[\xi_{i,j} \xi_{i,j}^\prime = \overline{X_{i,j} Y_{i,j}} \cdot t_{q_{i+1,j}} = \overline{1 + X_{i+1,j-1} Y_{i+1,j}} \cdot t_{q_{i+1,j}} = t_{q_{i+1,j}} + \xi_{i+1,j-1} \xi_{i+1,j}^\prime,\]
where $\xi_{i+1,0} \coloneqq 1$ when $j = 1$.
From these, we obtain a $\Bbbk$-algebra isomorphism 
\[{\rm Cox}(X_{\mathcal{A}_C} (\widehat{\Delta}_C)) \xrightarrow{\sim} \widetilde{{\rm Cox}}(X_{\mathcal{A}_C} (\widehat{\Delta}_C))\]
given by $\eta_s \mapsto 1, \xi_{2s-1,n+1-s}^\prime \mapsto Z_{q_{2s-1,n+1-s}}$ for $1 \leq s \leq n$ and $\xi_{i,j} \mapsto W_{q_{i,j}}, \xi_{i,j}^\prime \mapsto Z_{q_{i,j}}$ for $q_{i,j} \in \widehat{\Pi}_C^\circ$.
Since the relation $W_{q_{i,j}} Z_{q_{i,j}} = t_{q_{i+1,j}} + W_{q_{i+1,j-1}} Z_{q_{i+1,j}}$ can be written as $t_{q_{i+1,j}} = W_{q_{i,j}} Z_{q_{i,j}} - W_{q_{i+1,j-1}} Z_{q_{i+1,j}}$, $\widetilde{{\rm Cox}}(X_{\mathcal{A}_C} (\widehat{\Delta}_C))$ is isomorphic to the polynomial ring 
\[\Bbbk[W_q \mid q \in \widehat{\Pi}_C^\circ][Z_q \mid q \in \Pi_C \setminus \Pi_C^\ast][t_{q_{1,1}}, \ldots, t_{q_{1,n}}].\]
This proves the theorem.
\end{proof}

\begin{rem}\label{r:general_case}
Let $(\Pi, \Pi^\ast, \lambda)$ be a graded marked poset with $\lambda \in \mathbb{Z}^{\Pi^\ast}$ satisfying $\lambda_a = \lambda_b$ for all $a, b \in \Pi^\ast$ such that $r(a) = r(b)$. 
We assume that $(\Pi, \Pi^\ast)$ satisfies condition $(\spadesuit)$ in Introduction. 
In addition, suppose that there exists an interior lattice point ${\bm u} \in \mathcal{O} (\Pi, \Pi^\ast, \lambda)$ satisfying \eqref{eq:assumption}.
Then our proofs of all the results in this section can be naturally extended to $(\Pi, \Pi^\ast, \lambda)$ and $\widehat{\Delta}_{\mathcal{C}} (\Pi, \Pi^\ast, \lambda)$. 
Hence we obtain Theorems \ref{t:1}, \ref{t:2}, and \ref{t:3} in Introduction. 
\end{rem}

\begin{ex}
Let $(\Pi, \Pi^\ast, \lambda)$ be a graded marked poset with $\lambda = (\lambda_1, \lambda_2) \in \mathbb{Z}^2$ whose marked Hasse diagram is given in Figure \ref{Hasse_example}.
This $(\Pi, \Pi^\ast)$ satisfies condition $(\spadesuit)$ in Introduction. 
In addition, if $\lambda_2 - \lambda_1 \geq 4$, then there exists an interior lattice point ${\bm u} \in \mathcal{O} (\Pi, \Pi^\ast, \lambda)$ satisfying \eqref{eq:assumption}.
Hence Theorems \ref{t:1}, \ref{t:2}, and \ref{t:3} in Introduction can be applied to this $(\Pi, \Pi^\ast, \lambda)$. 
\begin{figure}[!ht]
\begin{center}
   \includegraphics[width=8.0cm,bb=60mm 180mm 160mm 230mm,clip]{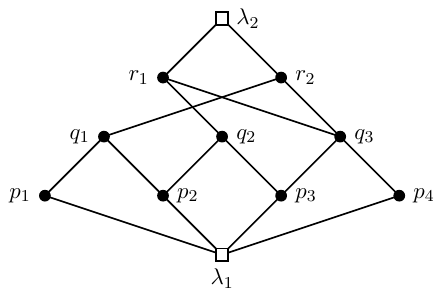}
	\caption{A Hasse diagram satisfying condition $(\spadesuit)$.}
	\label{Hasse_example}
\end{center}
\end{figure}
\end{ex}

\section{Case of Gelfand--Tsetlin posets of type $A$}

Fix $n \in \mathbb{Z}_{>0}$ and take $\lambda = (\lambda_1, \ldots, \lambda_n, \lambda_{n+1}) \in \mathbb{Z}^{n+1}$ such that $\lambda_1 \leq \cdots \leq \lambda_n \leq \lambda_{n+1}$.
In this section, let us consider the \emph{Gelfand--Tsetlin poset} $(\Pi_A, \Pi^\ast_A, \lambda)$ of type $A_n$ whose marked Hasse diagram is described in Figure \ref{type_A_Hasse}, where we write 
\[\Pi_A \setminus \Pi^\ast_A = \{q_{i,j} \mid 1 \leq j \leq n,\ n+1-j \leq i \leq 2n+1 -2j\}\]
and denote by the circles (respectively, the squares) the elements of $\Pi_A \setminus \Pi^\ast_A$ (respectively, $\Pi^\ast_A$). 
The marking $\lambda = (\lambda_a)_{a \in \Pi^\ast_A}$ is given as $(\lambda_1, \lambda_{2}, \ldots, \lambda_{n+1})$.
For $1 \leq i \leq n+1$, let $q_i^\ast$ denote the only one element of $\Pi^\ast_A$ with $r(q_i^\ast) = 2i-2$, that is, $\lambda_{q_i^\ast} = \lambda_i$. 
\begin{figure}[!ht]
\begin{center}
   \includegraphics[width=10.0cm,bb=40mm 120mm 170mm 230mm,clip]{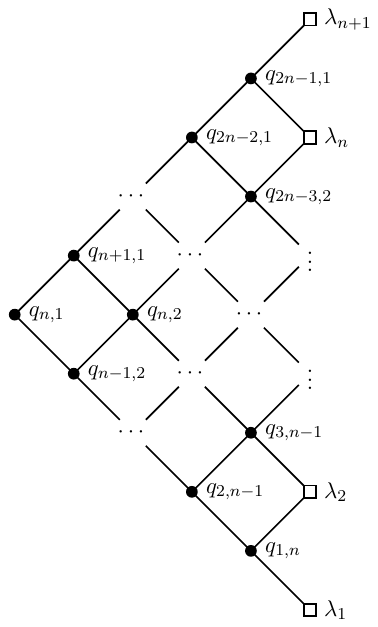}
	\caption{The marked Hasse diagram of the Gelfand--Tsetlin poset $(\Pi_A, \Pi^\ast_A, \lambda)$.}
	\label{type_A_Hasse}
\end{center}
\end{figure}

\begin{rem}
By definition, the marked order polytope $\mathcal{O}(\Pi_A, \Pi^\ast_A, \lambda)$ coincides with a Gelfand--Tsetlin polytope (see \cite[Section 5]{Lit} for the definition), and the marked chain polytope $\mathcal{C}(\Pi_A, \Pi^\ast_A, \lambda)$ coincides with an FFLV polytope (see \cite[equation (0.1)]{FeFL1} for the definition). 
\end{rem}

Let $\mathcal{M}_A$ denote the polyptych lattice corresponding to $(\Pi_A, \Pi^\ast_A, \lambda)$, and write $M_{\mathcal{C}}^{(A)} \coloneqq \pi_\mathcal{C} (\mathcal{M}_A)$ for $\mathcal{C} \in 2^{\Pi_A \setminus \Pi^\ast_A}$.
For $i \in \mathbb{Z}_{\geq 0}$, write $\Pi_A(i) \coloneqq \{p \in \Pi_A \mid r(p) = i\}$. 
More precisely, we have $\Pi_A(2i-1) = \{q_{2i-1, j} \mid \max\{n+2(1-i), 1\} \leq j \leq n+1-i\}$ for $1 \leq i \leq n$ and $\Pi_A(2i) = \{q_{2i, j} \mid \max\{n+1-2i, 1\} \leq j \leq n-i+1\}$ for $0 \leq i \leq n$, where $q_{2i, n-i+1} \coloneqq q_{i+1}^\ast$.
For every $q_{i, j} \in \Pi_A \setminus \Pi^\ast_A$, let $\varepsilon_{i, \leq j}, \varepsilon_{i, \leq j}^\prime$ denote the elements of $\mathcal{M}_A$ given by $\pi_\emptyset (\varepsilon_{i, \leq j}) = {\bm e}_{q_{i,\max\{n+1-i, 1\}}} + \cdots + {\bm e}_{q_{i, j}} \in M_\emptyset^{(A)}$ and $\pi_\emptyset (\varepsilon_{i, \leq j}^\prime) = -({\bm e}_{q_{i,\max\{n+1-i, 1\}}} + \cdots + {\bm e}_{q_{i, j}}) \in M_\emptyset^{(A)}$, where ${\bm e}_q \in \mathbb{R}^{\Pi_A \setminus \Pi^\ast_A}$ denotes the unit vector corresponding to $q \in \Pi_A \setminus \Pi^\ast_A$.
We define $\mathbb{M}_A$ (respectively, $(\mathbb{M}_A)_\mathbb{R}$) to be the set of 
\[(y_{i,j}, y_{i,j}^\prime \mid q_{i, j} \in \Pi_A \setminus \Pi^\ast_A) \in (\mathbb{Z}^{\Pi_A \setminus \Pi^\ast_A})^2\ (\text{respectively,}\ \in (\mathbb{R}^{\Pi_A \setminus \Pi^\ast_A})^2)\] 
satisfying the following equalities: 
\[y_{i,j} - y_{i,j}^\prime = 
\begin{cases}
\min \{0, -y_{i+1,j}^\prime + y_{i+1,j-1}\} &\text{if}\ q_{i+1,j} \in \Pi_A \setminus \Pi^\ast_A,\\
0 &\text{otherwise},
\end{cases}\]
where $y_{i+1,0} \coloneqq 0$ when $j = 1$.
In a way similar to the proof of Proposition \ref{p:relation_sp_points_M_1} (see also Proposition \ref{p:relation_sp_points_M_2}), we deduce the following analog of Proposition \ref{p:relation_sp_points_type_C_GT}.

\begin{prop}\label{p:relation_sp_points_type_A_GT}
Each point $\varphi \in {\rm Sp}(\mathcal{M}_A)$ satisfies 
\[\varphi(\varepsilon_{i, \leq j}) + \varphi(\varepsilon_{i, \leq j}^\prime) = 
\begin{cases}
\min \{0, \varphi(\varepsilon_{i+1, \leq j}^\prime) + \varphi(\varepsilon_{i+1, \leq j-1})\} &\text{if}\ q_{i+1,j} \in \Pi_A \setminus \Pi^\ast_A,\\
0 &\text{otherwise}
\end{cases}\]
for $q_{i,j} \in \Pi_A \setminus \Pi^\ast_A$, where $\varepsilon_{i+1, \leq 0} \coloneqq 0 \in \mathcal{M}_A$ when $j = 1$. 
In addition, there exists a bijective map ${\rm Sp}(\mathcal{M}_A) \rightarrow \mathbb{M}_A$, $\varphi \mapsto (y_{i,j}, y_{i,j}^\prime \mid q_{i, j} \in \Pi_A \setminus \Pi^\ast_A)$, given by $y_{i,j} \coloneqq \varphi(\varepsilon_{i, \leq j})$ and $y_{i,j}^\prime \coloneqq -\varphi(\varepsilon_{i, \leq j}^\prime)$, which naturally extends to a bijective map ${\rm Sp}_\mathbb{R}(\mathcal{M}_A) \rightarrow (\mathbb{M}_A)_\mathbb{R}$.
\end{prop}

Note that the marked poset $(\Pi_A, \Pi^\ast_A, \lambda)$ satisfies ($\spadesuit$) in Introduction. 
In addition, if $\lambda$ satisfies
\begin{enumerate}
\item[($\dagger 3$)] $\lambda_{k+1} - \lambda_{k} \geq 2$ for all $1 \leq k \leq n$, 
\end{enumerate}
then there exists ${\bm u} = (u_p)_{p \in \Pi_A \setminus \Pi_A^\ast} \in {\mathcal O}(\Pi_A, \Pi_A^\ast, \lambda) \cap \z^{\Pi_A \setminus \Pi_A^\ast}$ satisfying condition \eqref{eq:assumption} and $u_{p_1} < u_{p_2}$ for all $p_1, p_2 \in \Pi_A$ with $p_1 \lessdot p_2$, where $u_p \coloneqq \lambda_p$ for $p \in \Pi_A^\ast$.
This implies that ${\bm u}$ is an interior lattice point of ${\mathcal O}(\Pi_A, \Pi_A^\ast, \lambda)$.  
From these, it follows that all the results in Section \ref{s:type_C} are naturally extended to $\mathcal{M}_A$ as in Introduction. 
More precisely, let us consider a polynomial ring $\Bbbk [X_q, Y_q \mid q \in \Pi_A \setminus \Pi^\ast_A]$.
For $q_{i, j} \in \Pi_A \setminus \Pi^\ast_A$, we write $X_{i,j} \coloneqq X_{q_{i,j}}$, $Y_{i,j} \coloneqq Y_{q_{i,j}}$, and define $g_{i,j} = g_{q_{i,j}} \in \Bbbk [X_q, Y_q \mid q \in \Pi_A \setminus \Pi^\ast_A]$ by 
\[g_{i,j} \coloneqq X_{i,j} Y_{i,j} - 1 - X_{i+1,j-1} Y_{i+1, j},\]
where $X_{i+1,0} \coloneqq 1$ when $j = 1$ and $Y_{i+1, j} \coloneqq 0$ unless $q_{i+1,j} \in \Pi_A \setminus \Pi^\ast_A$; cf.\ \cref{p:relation_sp_points_type_A_GT}.
Then the detropicalization $\mathcal{A}_A$ of $\mathcal{M}_A$ is given by $\mathcal{A}_A = \Bbbk [X_q, Y_q \mid q \in \Pi_A \setminus \Pi^\ast_A]/I_A$ for $I_A \coloneqq (g_{i,j} \mid q_{i,j} \in \Pi_A \setminus \Pi^\ast_A)$. 
Note that $\widehat{\Delta}_A \coloneqq \widehat{\Delta}(\Pi_A, \Pi^\ast_A, \lambda) \subseteq (\mathcal{M}_A)_\mathbb{R}$ depends on the choice of ${\bm u}$.
We fix ${\bm u} \in {\mathcal O}(\Pi_A, \Pi_A^\ast, \lambda) \cap \z^{\Pi_A \setminus \Pi_A^\ast}$ as above, and consider the compactification $X_{\mathcal{A}_A}(\widehat{\Delta}_A) \coloneqq {\rm Proj}(\mathcal{A}_A^{\widehat{\Delta}_A})$ of ${\rm Spec}(\mathcal{A}_A)$ associated with $\widehat{\Delta}_A$.  
By Theorem \ref{t:2}, the marked chain-order polytope $\widehat{\Delta}_\mathcal{C}(\Pi_A, \Pi^\ast_A, \lambda)$ is a Newton--Okounkov body of $X_{\mathcal{A}_A}(\widehat{\Delta}_A)$ for each $\mathcal{C} \in 2^{\Pi_A \setminus \Pi_A^\ast}$, and there exists a flat (toric) degeneration of $X_{\mathcal{A}_A}(\widehat{\Delta}_A)$ to the normal projective toric variety associated with the integral polytope $\widehat{\Delta}_\mathcal{C}(\Pi_A, \Pi^\ast_A, \lambda)$. 
Let us consider the numbers $U_{\mathcal{M}_A}$ in \eqref{eq:rank_of_units} and $L_{\mathcal{M}_A}$ in \eqref{eq:number_of_divisors}. 
Since we have $U_{\mathcal{M}_A} = n$ and $L_{\mathcal{M}_A} = |\Pi_A \setminus \Pi_A^\ast| + n$, we see by Theorem \ref{t:3} that the Cox ring of $X_{\mathcal{A}_A}(\widehat{\Delta}_A)$ is isomorphic to the polynomial ring over $\Bbbk$ in $2|\Pi_A \setminus \Pi_A^\ast| = n (n+1)$ variables. 

\begin{ex}
Let $n = 2$. 
Then the mutation $\mu_{\emptyset, \mathcal{C}} \colon M_\emptyset^{(A)} \rightarrow M_{\mathcal{C}}^{(A)}$ for the polyptych lattice $\mathcal{M}_A$ is nonlinear if and only if $q_{3,1} \in \mathcal{C}$. 
In this case, if we write $\mu_{\emptyset, \mathcal{C}} (x_{1,2}, x_{2,1}, x_{3,1}) = (x_{1,2}^\prime, x_{2,1}^\prime, x_{3,1}^\prime)$, then $x_{1,2}^\prime, x_{2,1}^\prime$ are linear on $x_{1,2}, x_{2,1}, x_{3,1}$ while $x_{3,1}^\prime$ is nonlinear, where $x_{i,j}$ and $x_{i,j}^\prime$ are the coordinates corresponding to $q_{i,j} \in \Pi_A \setminus \Pi_A^\ast$.
More precisely, 
\[x_{3,1}^\prime = x_{3,1} + \min \{0, -x_{2,1}\};\]
cf.\ \cite[Example 2.6]{EHM}.
In addition, we have 
\[\mathcal{A}_A = \Bbbk [X_{i,j}, Y_{i,j} \mid (i,j) = (1,2), (2,1), (3,1)]/I_A,\]
where 
\[I_A = (X_{3,1} Y_{3,1} - 1,\ X_{2,1} Y_{2,1} - 1 - Y_{3,1},\ X_{1,2} Y_{1,2} - 1),\]
which implies that 
\[\mathcal{A}_A \simeq \Bbbk [X_{2,1}, Y_{2,1}, Y_{3,1}^{\pm 1}, Y_{1,2}^{\pm 1}]/(X_{2,1} Y_{2,1} - 1 - Y_{3,1});\]
cf.\ \cite[Example 6.7]{EHM}.
\end{ex}

\bibliographystyle{plain} 
\def\cprime{$'$} 

\end{document}